\documentclass[11pt,oneside,english]{amsart}
\usepackage[T1]{fontenc}
\usepackage[latin9]{inputenc}
\usepackage{geometry}
\usepackage{color}
\usepackage{babel}
\usepackage{units}
\usepackage{amsmath}
\usepackage{amsbsy}
\usepackage{amstext}
\usepackage{amsthm}
\usepackage{amssymb}
\usepackage{setspace}
\usepackage{esint}
\usepackage[numbers]{natbib}
\usepackage[unicode=true,pdfusetitle,
 bookmarks=true,bookmarksnumbered=true,bookmarksopen=true,bookmarksopenlevel=1,
 breaklinks=false,pdfborder={0 0 1},backref=section,colorlinks=true]
 {hyperref}

\makeatletter
\numberwithin{equation}{section}
\numberwithin{figure}{section}
\theoremstyle{plain}
\newtheorem{thm}{\protect\theoremname}[section]
\theoremstyle{plain}
\newtheorem{conjecture}[thm]{\protect\conjecturename}
\theoremstyle{definition}
\newtheorem{defn}[thm]{\protect\definitionname}
\theoremstyle{definition}
\newtheorem*{defn*}{\protect\definitionname}
\theoremstyle{remark}
\newtheorem{rem}[thm]{\protect\remarkname}
\theoremstyle{plain}
\newtheorem{cor}[thm]{\protect\corollaryname}
\theoremstyle{plain}
\newtheorem{prop}[thm]{\protect\propositionname}
\theoremstyle{plain}
\newtheorem{lem}[thm]{\protect\lemmaname}

\global\long\def\C{\mathbb{C}}%
\global\long\def\R{\mathbb{R}}%
\global\long\def\Q{\mathbb{Q}}%
\global\long\def\N{\mathbb{N}}%
\global\long\def\Z{\mathbb{Z}}%
\global\long\def\a{\mathfrak{a}}%

\DeclareMathOperator{\Hom}{Hom}

\DeclareMathOperator{\SL}{SL}
\DeclareMathOperator{\GL}{GL}
\DeclareMathOperator{\PGL}{PGL}

\DeclareMathOperator{\SU}{SU}

\DeclareMathOperator{\End}{End}
\DeclareMathOperator{\tr}{tr}
\DeclareMathOperator{\sph}{sph}
\DeclareMathOperator{\rank}{rank}
\DeclareMathOperator{\spann}{span}
\DeclareMathOperator{\re}{Re}
\DeclareMathOperator{\triv}{triv}
\DeclareMathOperator{\Aut}{Aut}
\DeclareMathOperator{\Ad}{Ad}
\DeclareMathOperator{\diag}{diag}
\DeclareMathOperator{\charr}{char}
\DeclareMathOperator{\Det}{Det}
\DeclareMathOperator{\cusp}{cusp}
\DeclareMathOperator{\nt}{nt}

\global\long\def\n#1{\left\Vert #1\right\Vert }%
\usepackage{tikz}
\usetikzlibrary{matrix}

\makeatother

\providecommand{\conjecturename}{Conjecture}
\providecommand{\corollaryname}{Corollary}
\providecommand{\definitionname}{Definition}
\providecommand{\lemmaname}{Lemma}
\providecommand{\propositionname}{Proposition}
\providecommand{\remarkname}{Remark}
\providecommand{\theoremname}{Theorem}

\begin{document}

\title{On Sarnak's Density Conjecture and its Applications}
\author{Konstantin Golubev and Amitay Kamber}
\begin{abstract}
Sarnak's Density Conjecture is an explicit bound on the multiplicities of non-tempered representations in a sequence of cocompact congruence arithmetic lattices in a semisimple Lie group, which is motivated by the work of Sarnak and Xue (\citep{sarnak1991bounds}). The goal of this work is to discuss similar hypotheses, their interrelation and applications. We mainly focus on two properties -- the spectral Spherical Density Hypothesis and the geometric Weak Injective Radius Property. Our results are strongest in the $p$-adic case, where we show that the two properties are equivalent, and both imply Sarnak's
General Density Hypothesis. One possible application is that either the Spherical Density Hypothesis or the Weak Injective Radius Property imply Sarnak's Optimal Lifting Property (\citep{sarnak2015lettermiller}).
Conjecturally, all those properties should hold in great generality. We hope that this work will motivate their proofs in new cases.
\end{abstract}

\thanks{Konstantin Golubev, ETH, k.golubev@gmail.com\\
Amitay Kamber, Centre for Mathematical Sciences, Wilberforce Road, Cambridge CB3 0WB, UK.
email: ak2356@dpmms.cam.ac.uk}

\maketitle

\section{Introduction}

Let $k$ be a local field (Archimedean or non-Archimedean), let $G$
be the $k$-rational points of a semisimple algebraic group defined
over $k$, let $\Gamma_1\subset G$ be a lattice, and let $(\Gamma_N) $
be a sequence of finite index subgroups of $\Gamma_1$ with $[\Gamma_{1}:\Gamma_N]\to\infty$.
There are various results about the multiplicities in the decomposition of $L^{2}(\Gamma_N\backslash G)$ into irreducible representation
(e.g., \citep{degeorge1978limit,sauvageot1997principe,abert2017growth}).
An extremely strong property is the \emph{very naive Ramanujan property}, stating that if $\pi$ is non-tempered and non-trivial, then it does not appear in the decomposition. However, the \emph{very naive Ramanujan property} is usually not true in high rank (see, e.g., \citep{burger1992ramanujan}).
Notice that we do not make a distinction between cusp forms and non-cusp
forms -- the \emph{naive Ramanujan conjecture} states that cusp forms are tempered, and it is also not true in general (\citep{howe1979counterexample}).
Moreover, even when the naive Ramanujan conjecture is expected to be true, it is usually out of reach by the existing methods.

Recently, Sarnak made a density conjecture which is an approximation to the very naive Ramanujan property and should serve as a replacement of it for applications. Some instances of this general idea were previously
given for hyperbolic surfaces (\citep{sarnak2015lettermiller,golubev2019cutoff}),
and for graphs (\citep{bordenave2021cutoff,kamber2019p}). Our goal here is to give a general framework for similar density conjectures and their use in applications.

We give a geometric and somewhat elementary approach to the problem. An alternative approach based on deep results in the Langlands program may be found in an ongoing work of Shai Evra. 

To state Sarnak's Density Conjecture, we first set some notations.
Let $(\pi,V)$ be a unitary irreducible representation
of $G$ and let $0\ne v_{0}\in V$ be a $K$-finite vector, where
$K$ is a maximal compact subgroup of $G$, which is good in the sense of Bruhat and Tits in the non-Archimedean case (\citep{bruhat1972groupes}).
We let $1\le p(\pi)\le\infty$ be the infimum over $p\ge1$ such that
the matrix coefficient $\beta\colon G\to\C$, $\beta(g)=\left\langle v_{0},\pi(g)v_{0}\right\rangle $
is in $L^{p}(G)$. It is a well-known fact that $p(\pi)$
does not depend on the choice of $v_{0}$. Let $\Pi(G)$
be the set of isomorphism classes of irreducible unitary representations
of $G$ endowed with the Fell topology. For a cocompact lattice $\Gamma$
and $(\pi,V)\in\Pi(G)$, denote $m(\pi,\Gamma)=\dim\Hom_{G}(V,L^{2}(\Gamma\backslash G))$,
i.e., the multiplicity of $\pi$ in the decomposition of $L^{2}(\Gamma\backslash G)$ into irreducible representations.
For a subset $A\subset\Pi(G)$, denote $M(A,\Gamma,p)=\sum_{\pi\in A,p(\pi)\ge p}m(\pi,\Gamma)$.
\begin{conjecture}[Sarnak's Density Conjecture]
Let $G$ be a real, semisimple, almost-simple and simply connected
Lie group, let $\Gamma_1$ be a cocompact arithmetic lattice of
$G$ and let $(\Gamma_N)$ be a sequence of finite index congruence
subgroups of $\Gamma_1$, with $[\Gamma_{1}:\Gamma_N]\to\infty$.
Then for every precompact subset $A\subset\Pi(G)$ and
$\epsilon>0$ there exists a constant $C_{\epsilon,A}$ such that
for every $N$ and $p>2$,
\[
M(A,\Gamma,p)\le C_{A,\epsilon}[\Gamma_{1}:\Gamma_N]^{2/p+\epsilon}.
\]
\end{conjecture}

We refer to a sequence of lattices satisfying this multiplicity property
as a sequence which satisfies the \emph{General Density Hypothesis}.
A similar conjecture appeared in the work of Sarnak and Xue (\citep{sarnak1991bounds}), but they only considered the case when $A=\left\{ \pi\right\} $ is a singleton. In such a case, we say that the sequence of lattices satisfies the \emph{Pointwise Multiplicity Hypothesis}.

We will prefer to work with a different spectral definition, the \emph{Spherical
Density Hypothesis}, which is easier to use for applications, and
concerns only spherical representations. Let $\Pi(G)_{\sph}\subset\Pi(G)$
be the set of isomorphism classes of spherical representations, i.e.,
of irreducible unitary representations with a non-zero $K$-invariant
vector. In the $p$-adic case, or the rank $1$ case, the set $\left\{ \pi\in\Pi(G)_{\sph}:p(\pi)>2\right\} $
is precompact, and the Spherical Density Hypothesis is simply the
case when $A=\left\{ \pi\in\Pi(G)_{\sph}:p(\pi)>2\right\} $
in the General Density Hypothesis. When $G$ is Archimedean of high rank, $\left\{ \pi\in\Pi(G)_{\sph}:p(\pi)>2\right\} $
is not necessarily precompact, so we associate to a spherical representation $(\pi,V)\in\Pi(G)$ a number $\lambda(\pi)\in\R_{\ge0}$, which is the eigenvalue of the Casimir operator on the $K$-invariant subspace of $\pi$, and define:
\begin{defn}
The sequence $(\Gamma_N) $ of cocompact lattices
satisfies the \emph{Spherical Density Hypothesis} if:
\begin{itemize}
\item In the $p$-adic or rank $1$ case, for every $\epsilon>0$ there
exists $C_{\epsilon}$ such that for every $N\ge1$, $p>2$, 
\[
M(\Pi(G)_{\sph},\Gamma_N,p)\le C_{\epsilon}[\Gamma_{1}:\Gamma_N]^{2/p+\epsilon}.
\]
\item In the general Archimedean case, there exists $L>0$ large enough,
such that for every $\epsilon>0$ there exists $C_{\epsilon}$ such
that for every $\lambda\ge0$, $N\ge1$, $p>2$, 
\[
M(\left\{ \pi\in\Pi(G)_{\sph}:\lambda(\pi)\le\lambda\right\} ,\Gamma_N,p)\le C_{\epsilon}(1+\lambda)^{L}[\Gamma_{1}:\Gamma_N]^{2/p+\epsilon}.
\]
\end{itemize}
\end{defn}

In the second case, the Spherical Density Hypothesis does not a priori
follows from the General Density Hypothesis.

To state our main geometric definition, we need to set some more notations. We view $G$, $\Gamma_1$ and the sequence $(\Gamma_N)$ as fixed. We use the standard $O,\Theta,o$ notations, where for example $f(N,\epsilon)=O_{\epsilon}(g(N,\epsilon))$
says that for every $\epsilon$ there exists $C$ depending only on
$\epsilon$ (and on $G,\Gamma_1,(\Gamma)_N$), such that $f(N,\epsilon)\le Cg(N,\epsilon)$
for $N$ large enough. The notation $f(N,\epsilon)\ll_{\epsilon}g(N,\epsilon)$
is the same as $f(N,\epsilon)=O_{\epsilon}(g(N,\epsilon))$
and $f(N,\epsilon)\asymp_{\epsilon}g(N,\epsilon)$ is the same as
$f(N,\epsilon)\ll_{\epsilon}g(N,\epsilon)$ and $g(N,\epsilon)\ll_{\epsilon}f(N,\epsilon)$.

We fix a Cartan decomposition $G=KA_{+}K$ and an Iwasawa decomposition
$G=KP$. Let $\delta\colon P\to \R_{>0}$ be the left modular character of
$P$ (see Section~\ref{sec:Preliminaries} for more details).

We define a length $l\colon G\to \R_{\ge 0}$ by first denoting for $a\in A_{+}$, $l(a)=\log_{q}\delta(a)$,
where $q$ is equal to $e$ in the Archimedean case and is equal to the size of the quotient field of $k$ otherwise. Then we extend $l$ to $G$ using the Cartan decomposition, i.e., $l(k_{1}ak_{2})=l(a)$. Finally,
we define a metric on $G/K$ by $d(x,y)=l(x^{-1}y)$.

The Weak Injective Radius Property is based on the lattice point counting approach of Sarnak and Xue (\citep[Conjecture 2]{sarnak1991bounds}).
Given an element $y=\Gamma_N y\in\Gamma_N\backslash\Gamma_1$,
we denote 
\[
\boldsymbol{N}(\Gamma_N,d_{0},y)=\#\left\{ \gamma\in\Gamma_N:l(y^{-1}\gamma y)\le d_{0}\right\} .
\]

\begin{defn}
The sequence $(\Gamma_N)$ satisfies the \emph{Weak Injective Radius Property} if for every $0\le d_{0}\le2\log_{q}([\Gamma_1:\Gamma_N])$,
$\epsilon>0$,
\[
\frac{1}{[\Gamma_{1}:\Gamma_N]}\sum_{y\in\Gamma_N\backslash\Gamma_{1}}\boldsymbol{N}(\Gamma_N,d_{0},y)\ll_{\epsilon}[\Gamma_1:\Gamma_N]^{\epsilon}q^{d_{0}/2}.
\]
\end{defn}

This definition is somewhat different from \citep[Conjecture 2]{sarnak1991bounds}.
For rank $1$, it is slightly weaker (see Proposition~\ref{prop:Counting for large}),
while for higher rank we use a different length. In this form, the
Weak Injective Radius Property follows from the Spherical Density
Hypothesis -- see Theorem~\ref{thm:density implies counting} below. On the other hand, the Weak Injective Radius Property makes sense also for non-uniform lattices.

We can now state our intended application. First, we say that a sequence
of lattices $(\Gamma_N) $ has a Spectral Gap if there
exists $p_{0}<\infty$ such that $p(\pi)\le p_{0}$ for
every non-trivial spherical $\pi\in\Pi(G)$ weakly contained
in $L^{2}(\Gamma_N\backslash G)$. This definition can be applied to the non-uniform case as well, and in the cocompact case we may replace
\emph{weakly contained} by $m(\pi,\Gamma)>0$.

We look at the natural action $\pi_N\colon \Gamma_1\to\Aut(\Gamma_N\backslash\Gamma_1)$,
defined by $\pi_N(\gamma)(\Gamma_N\gamma')=\Gamma_N\gamma'\gamma^{-1}$.
Given $x,y\in\Gamma_N\backslash\Gamma_1$, we look for a \emph{small}
element $\gamma\in\Gamma_1$ such that $\pi_N(\gamma)x=y$.
A very general way of measuring how small is an element is by the
Cartan Decomposition, $G=KA_{+}K$. For $\gamma\in\Gamma\subset G$
we let $a_{\gamma}$ be the element in $A_{+}$ in the Cartan decomposition
of $\gamma$. We also fix some norm $\n{\cdot}_{\a}$ on the underlying
coroot space of $A_{+}$. By \citep{duke1993density}, the number
of $\gamma\in\Gamma_1$ with $\n{a_{\gamma}-a}_{\a}<\delta$
is $\asymp_{\Gamma_{1}}q^{l(a)}$. Therefore, the following
definition is optimal.
\begin{defn}
The sequence $(\Gamma_N) $ has the \emph{Optimal Lifting Property} if for every $\epsilon>0$, for every $a\in A_{+}$
with $l(a)\ge(1+\epsilon)\log_{q}([\Gamma_{1}:\Gamma_N])$,
\[
\#\left\{ (x,y)\in(\Gamma_N\backslash\Gamma_1)^{2}:\exists\gamma\in\Gamma_1\text{ s.t. }\pi_N(\gamma)x=y,\n{a_{\gamma}-a}_{\a}<\epsilon\n a_{\a}\right\} =(1-o_{\epsilon}(1))[\Gamma_{1}:\Gamma_N]^{2}.
\]
\end{defn}

Conjecturally, every sequence of congruence subgroups of an arithmetic
lattice in an almost-simple and simply connected Lie group satisfies
the Optimal Lifting Property. We refer to Conjecture~\ref{conj: weak injective radius conjetcure}
for a full statement.

The following two theorems show that our two main properties imply
the Optimal Lifting Property.
\begin{thm}
\label{thm:injective radius implies optimal lifting}Let $(\Gamma_N)$
be a sequence of lattices having a Spectral Gap and satisfying the
Weak Injective Radius Property. Then the sequence $(\Gamma_N)$
has the Optimal Lifting Property.
\end{thm}

\begin{thm}
\label{thm:density implies counting}Let $(\Gamma_N)$ be a sequence of cocompact lattices satisfying the Spherical Density Hypothesis. Then $(\Gamma_N)$ satisfies the Weak Injective Radius
Property. Therefore, assuming also Spectral Gap, the sequence also has the Optimal Lifting Property.
\end{thm}

The definition of Optimal Lifting is based on the main result in an
influential letter of Sarnak (\citep{sarnak2015lettermiller}), who
proved the Optimal Lifting Property for principal congruence subgroups
of $\SL_{2}(\Z)$, by utilizing a version of the Spherical
Density Hypothesis proved by Huxley (\citep{huxley1986exceptional}).

In the cocompact case, one may relate the Optimal Lifting Property
to almost-diameter of the quotient space as follows. If we give $X_N=\Gamma_N\backslash G/K$
the quotient metric $d$, the Optimal Lifting Property implies that
the distances between the points of $X_N$ are concentrated at the optimal location $\log_{q}(\mu(X_N))$, i.e. for every
$\epsilon>0$,
\[
\mu\left(\left\{ (x,y)\in X_N\times X_N:d(x,y)<(1+\epsilon)\log_{q}([\Gamma_{1}:\Gamma_N])\right\} \right)=(1-o_{\epsilon}(1))\mu^{2}(X_N).
\]

This concentration of distances phenomena was proven for Ramanujan
graphs by Sardari (\citep{sardari1diameter}) and Lubetzky-Peres (\citep{lubetzky2016cutoff}), who also related it to the cutoff phenomena. In higher dimensions, similar results for Ramanujan complexes appear in \citep{kamber2016lpcomplex,lubetzky2020random}.
Theorem~\ref{thm:injective radius implies optimal lifting} implies
that one may get results which are almost as strong, as long as we
assume only the far weaker Spherical Density Hypothesis or the Injective Radius Property.

The Weak Injective Radius Property is intended as the arithmetic,
or geometric, input to our approach, and we discuss it further in
Section~\ref{sec:Applications-and-Open}. There are a few cases where
it is known, most notably, following the work of Sarnak and Xue, for
principal congruence subgroups of arithmetic lattices in $\SL_{2}(\R)$
and $\SL_{2}(\C)$ (see Subsection~\ref{subsec:Sarnak-Xue work}).
In a companion paper by the second named author and Hagai Lavner,
the Weak Injective Radius Property is proved for some non-principal
congruence subgroups of $\SL_{3}(\Z)$, and this result is closely related to the works \citep{blomer2017applications,blomer2019density} (see Subsection~\ref{subsec:Gamma0} for a full discussion).
If we allow ourselves to relax the definition and add a parameter
$0<\alpha\le1$ to the definition of the Weak Injective Radius Property (see Subsection~\ref{subsec:Adding-a-Parameter}), then
it is quite straightforward to show that \emph{principal} congruence subgroups of arithmetic groups satisfy the Weak Injective Radius Property with some explicit parameter $\alpha>0$ (see Corollary~\ref{cor:principal congruence}).
As a matter of fact, recent results in \citep{abert2017growth,finis2018approximation}
show that every sequence of congruence subgroups satisfy the Weak Injective Radius Property with some explicit parameter $\alpha>0$
(see Theorem~\ref{thm:non-principal congruence subgroups}). However,
one must have $\alpha=1$ for the Optimal Lifting application.

\subsection{\label{subsec:intro part 2}The Relations Between the Different Properties}

We already stated that the Spherical Density Hypothesis implies the
Weak Injective Radius Property. For the deduction of spectral results
from the Weak Injective Radius Property, we have partial results in
the Archimedean case and full results in the $p$-adic case.

We believe that the following is true:
\begin{conjecture}
\label{conj:injective radius implies density archimedean}The Weak
Injective Radius Property implies both the General Density Hypothesis
and the Spherical Density Hypothesis.
\end{conjecture}

In the $p$-adic case, one can choose larger and larger sets covering $\Pi(G)$ as follows. For a compact open subgroup $K'$ of $G$, let $\Pi(G)_{K^{\prime}-\sph}$ be the set of isomorphism classes of irreducible unitary representations with a non-zero $K^{\prime}$-invariant vector. If we have a sequence $( K_{m}^{\prime})$ of arbitrarily small compact open subgroups (i.e., they generate the topology of $G$ near the identity), then we have 
\[
\Pi(G)=\bigcup_{m}\Pi(G)_{K_{m}^{\prime}-\sph}.
\]

We can now state:
\begin{thm}
\label{thm:p-adic injective radius implies density}Conjecture~\ref{conj:injective radius implies density archimedean} is true when $G$ is non-Archimedean. More precisely, there exists a sequence $( K_{m}^{\prime})$ consisting of arbitrarily
small compact open subgroups of $G$, such that if the Weak Injective
Radius Property holds for a sequence of cocompact lattices $(\Gamma_N)$,
then for every $N\ge1$, $m\ge1$, $p>2$, $\epsilon>0$, 
\[
M(\Pi(G)_{K_{m}^{\prime}-\sph},\Gamma_N,p)\ll_{K_{m}^{\prime},\epsilon}[\Gamma_{1}:\Gamma_N]^{2/p+\epsilon}.
\]
\end{thm}

In the Archimedean case, we do not know even whether the Weak Injective
Radius Property implies the Spherical Density Hypothesis. However,
for rank $1$ it was essentially proven in \citep{sarnak1991bounds}
(see the remark after the statement of Theorem 3 in \citep{sarnak1991bounds}):
\begin{thm}
\label{thm:rank one injective radius implies density}If $G$ is of
rank one and $(\Gamma_N) $ is a sequence of cocompact
lattices, then the Weak Injective Radius Property implies the Spherical
Density Hypothesis.
\end{thm}

For general representations in the Archimedean case, we have the following
theorem. For rank $1$ it was proven in \citep[Theorem 3]{sarnak1991bounds}.
\begin{thm}
\label{thm:Diophantine implies pointwise}If the sequence $(\Gamma_N) $
of cocompact lattices satisfies the Weak Injective Radius Property,
then $(\Gamma_N) $ satisfies the Pointwise Multiplicity
Hypothesis.
\end{thm}

Figure~\ref{fig:SX properties relation} summarizes the different
relations between our main properties for a sequence of cocompact
lattices.

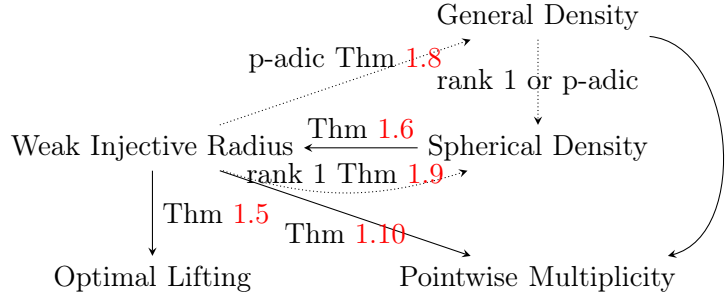
\begin{figure}[h]
\begin{tikzpicture}
\matrix (m) [matrix of math nodes, row sep=3em,
	column sep=3em]{
	 						   & \text{General Density} 	 &  \\
	\text{Weak Injective Radius}& \text{Spherical Density} 	&     \\
	\text{Optimal Lifting}	& \text{Pointwise Multiplicity} &	 \\};

\path[-stealth]     
	(m-2-1) edge node [right] {Thm \ref{thm:injective radius implies optimal lifting}}(m-3-1)
	(m-2-2) edge node [above]{Thm \ref{thm:density implies counting}} (m-2-1)
	(m-2-1) edge [densely dotted] node [above] {p-adic Thm \ref{thm:p-adic injective radius implies density}}(m-1-2)
	(m-1-2) edge [densely dotted] node  {rank 1 or p-adic}(m-2-2)
	(m-1-2) edge [bend left = 80](m-3-2)
	(m-2-1) edge node [below] {Thm \ref{thm:Diophantine implies pointwise}}(m-3-2)
	(m-2-1) edge [densely dotted,bend right=18] node [above] {rank 1 Thm \ref{thm:rank one injective radius implies density}}(m-2-2);
\end{tikzpicture}

\caption{\label{fig:SX properties relation}The relations between our main
properties for a sequence of cocompact lattices.}
\end{figure}
Let us end this introduction by stating some open problems this work
leads to.

The main open problem is to prove either the Spherical Density Hypothesis
or the Injective Radius Property for new cases, which would lead to
a proof of the Optimal Lifting Property. See Conjecture~\ref{conj: weak injective radius conjetcure}
for a general conjecture for the Archimedean case. Very few cases
of this conjecture are known for groups of rank greater than $1$.
We remark that it seems that the problem is harder for principal congruence
subgroups, and easier for group that are far from being normal, such
as $\Gamma_{0}(N)$ of $\SL_{m}(\Z)$ (see Subsection~\ref{subsec:Gamma0}).
See, for example, the \emph{Density Amplification} phenomena for
graphs in \citep{golubev2019cutoff2}.

In the more technical side, a main problem this work does not resolve
is Conjecture~\ref{conj:injective radius implies density archimedean},
which would show, in particular, that the Spherical Density Hypothesis
and the Injective Radius Property are indeed equivalent, also for
the Archimedean high rank case. It relates in particular to the understanding
of uniform \emph{lower bounds} on matrix coefficients, see, e.g., Conjecture~\ref{conj:mat-coeff-lower-bound}.

Finally, we only discuss multiplicities for cocompact lattices. We
strongly believe that the Weak Injective Radius Property has spectral
implications for non-uniform lattices as well, e.g., for bounds of
multiplicities of representations in the discrete spectrum. This problem
is related to the concentration of $L^{2}$-mass of non-tempered automorphic functions away from the cusp, in a uniform way. In hyperbolic spaces this problem is essentially solved, even for some discrete groups that are not lattices, thanks to the work of Gamburd on hyperbolic surfaces (\citep{gamburd2002spectral}) and the work of Magee on general hyperbolic spaces (\citep{magee2015quantitative}). Therefore, Theorem~\ref{thm:rank one injective radius implies density} can be generalized to such cases.

\subsection*{Structure of the Article}

In Section~\ref{sec:Applications-and-Open} we give some applications
of our results, and state some open problems.

In Section~\ref{sec:Main-Ideas} we present the main ideas behind
the proofs.

In Section~\ref{sec:Preliminaries} we collect various results, mainly
from representation theory. In particular, we discuss upper bounds
for matrix coefficients, which are well understood. Using those upper
bounds, we prove the essential Convolution Lemma~\ref{lem:convolution lemma}.

In Section~\ref{sec:Weak Injective Radius implies density} we discuss
the Weak Injective Radius Property and the spectral results
it implies. We prove Proposition~\ref{prop:good family proposition},
which reduces Theorem~\ref{thm:Diophantine implies pointwise}, Theorem~\ref{thm:rank one injective radius implies density},
and Theorem~\ref{thm:p-adic injective radius implies density} to
finding some explicit and strict lower bounds on matrix coefficients.

In Section~\ref{sec:Injective Radius implies Lifitng} we prove Theorem~\ref{thm:injective radius implies optimal lifting}.

In Section~\ref{sec:Density implies Stuff} we discuss the Spherical
Density Hypothesis and the results it implies. We prove Theorem~\ref{thm:density implies counting}
and prove Theorem~\ref{thm:spherical density implies optimal lifting},
which is a version of Theorem~\ref{thm:injective radius implies optimal lifting}
which assumes the Spherical Density Hypothesis and has stronger implications.

In Section~\ref{sec:Bernstein-Theory-of-NonBacktracking} we discuss
Bernstein description of the Hecke algebra in the non-Archimedean
case, and prove Theorem~\ref{thm:good family padic intro}, which
implies Theorem~\ref{thm:p-adic injective radius implies density}
together with Proposition~\ref{prop:good family proposition}

In Section~\ref{sec:pointwise lower bound} we discuss the theory
of leading coefficients and prove Theorem~\ref{thm: good family pointwise intro},
which implies Theorem~\ref{thm:Diophantine implies pointwise} together
with Proposition~\ref{prop:good family proposition}.

\subsection*{Acknowledgments}

The authors are thankful to Amos Nevo and Peter Sarnak for their support of this project.

During the work on this project, the first author was supported by the SNF grant 200020-169106 at ETH
Zurich. This work was part of the Ph.D. thesis of the second author at the Hebrew University of Jerusalem, under the guidance of Prof.
Alexander Lubotzky, and was supported by the ERC grant 692854.

\section{\label{sec:Applications-and-Open}Applications and Open Problems}

\subsection{Ramanujan Graphs and Complexes}

We shortly note that the results of this paper, and in particular
Theorem~\ref{thm:injective radius implies optimal lifting} (or the
stronger Theorem~\ref{thm:spherical density implies optimal lifting})
apply to Ramanujan Complexes, by which we mean here the situation
when $G$ is $p$-adic, $\Gamma_1$ is cocompact, and no non-tempered
and non-trivial spherical representation appears in the decomposition
of $L^{2}(\Gamma_{N}\backslash G)$. In this case, the
sequence $(\Gamma_N) $ obviously satisfies the Spherical
Density Hypothesis. The Ramanujan complexes themselves are the quotients
of the Bruhat-Tits buildings of $G$ by $\Gamma_{N}$. Such complexes were constructed (with the same definition) by Lubotzky, Samuels and Vishne (see \citep{lubotzky2005ramanujan,lubotzky2005explicit}).
Similar results to Theorem~\ref{thm:spherical density implies optimal lifting}
appear in \citep{kamber2016lpcomplex,lubetzky2020random,chapman2019cutoff}.

Note that the definition we use is not the same as in the more modern
approach to Ramanujan complexes, where one considers not only spherical
representation but also representations with a non-trivial vector
fixed by the Iwahori subgroup (See \citep[Subsection 2.3]{lubotzky2018high}
and the references within).

The standard way of proving that a complex is a Ramanujan complex
is to use results from the Langlands program (e.g.~\citep{lafforgue2002chtoucas}), and to eventually apply Deligne's proof of the Weil conjectures (\citep{deligne1974conjecture}).
This approach has obvious limitations, and in particular it seems
that one must assume that $G$ is $p$-adic for it to succeed. We
refer to \citep{evra2018ramanujan} for some recent work on this subject,
and to an ongoing and yet unpublished work of Shai Evra. 

We will avoid diving deeper into this subject, as it is based on spectral input, unlike our approach which is geometric.

\subsection{\label{subsec:Adding-a-Parameter}Adding a Parameter to the Properties}

It is useful to add a parameter $0<\alpha\le1$ to the properties,
with $\alpha=1$ being equivalent to the property without the parameter.
\begin{defn*}
We say that the sequence $(\Gamma_N) $ satisfies
the \emph{Weak Injective Radius Property with parameter $\alpha$,}
if for every $0\le d_{0}\le2\alpha\log_{q}([\Gamma_1:\Gamma_N])$,
$\epsilon>0$,
\[
\frac{1}{[\Gamma_1:\Gamma_N]}\sum_{y\in\Gamma_{N}\backslash\Gamma_{1}}\boldsymbol{N}(\Gamma_N,d_{0},y)\ll_{\epsilon}[\Gamma_1:\Gamma_N]^{\epsilon}q^{d_{0}/2}.
\]
\end{defn*}
\begin{defn*}
We say that the sequence $(\Gamma_N) $ satisfies
the General Density Hypothesis with parameter $\alpha$ if for every
precompact subset $A\subset\Pi(G)$ and every $\epsilon>0$,
$N$ and $p\ge2$,
\[
M(A,\Gamma,p)\ll_{\epsilon,A}[\Gamma_1:\Gamma_N]^{1-\alpha(1-2/p)+\epsilon}.
\]
\end{defn*}
One may similarly define the Spherical Density Hypothesis with parameter
$\alpha$ and the Pointwise Density Hypothesis with parameter $\alpha$,
by changing the exponent $2/p$ to $1-\alpha(1-2/p).$

Figure~\ref{fig:SX properties relation} remains true if we replace
the properties with their parameterized version, for the same parameter
$0<\alpha\le1$, with the exception that the derivation of the Optimal
Lifting Property from the Spherical Density Hypothesis requires $\alpha =1$.

In the work itself, we work with the parameterized version of the properties.

\subsection{\label{subsec:Density-from-inj-radius}Congruence Subgroups of Arithmetic Groups}

Let $\Gamma_1$ be an arithmetic lattice in a semisimple non-compact
Lie group $G$. Following \citep{sarnak1991bounds}, by arithmetic
we mean that $G$ is defined over $\Q$, $\Gamma_1\subset G(\Q)$,
there is a $\Q$-embedding $\rho\colon G\to \GL_{n}$ and $\Gamma_1$ is
commensurable with $\rho^{-1}(\GL_{n}(\Z))$.

In this case we may define a sequence of principle congruence subgroups
$(\Gamma_N) $ of $\Gamma_1$ by letting 
\[
\Gamma_{N}=\Gamma_1(N)=\Gamma_1\cap\rho^{-1}(\left\{ A\in \GL_{n}(\Z):A\equiv I\mod N\right\}).
\]

It is a well known fact that such subgroups have an injective radius which is logarithmic in the index (e.g., \citep[Lemma 1]{sarnak1991bounds}
or \citep[Proposition 16]{guth2014quantum}).

Let us shortly give the argument -- we assume by moving to a finite
index in $\Gamma_1$ that $\Gamma_1\subset\rho^{-1}(\GL_{n}(\Z))$.
In this case $\Gamma_{N}$ is normal in $\Gamma_1$ and it is left
to verify that for every $d_{0}\le2\alpha\ln([\Gamma_1:\Gamma_N])$
and $\epsilon>0$,
\[
\#\left\{ \gamma\in\Gamma_{N}:l(\gamma)\le d_{0}\right\} \ll_{\epsilon}e^{d_{0}(1/2+\epsilon)}.
\]
We note that there exists a constant $C$ such that for $g\in G$
outside a compact set the length $l(g)$ we defined satisfies
\[
C^{-1}\ln(\n{\rho(g)})\le l(g)
\]
where $\|\cdot\|\colon \GL_{n}(\R)\to\R_{\ge0}$ is the maximal
absolute value of an entry of the matrix. Since each element of $\rho(\Gamma_N)$ which is not the identity has an entry of size $N$, We deduce that for every $N$ large enough, 
\[
\left\{ \gamma\in\Gamma_{N}:l(\gamma)<C\ln(N)\right\} =\{ I\} .
\]

On the other hand, there is an injective map of $\Gamma_1/\Gamma_N$ into $\GL_n(\Z/N\Z)$, so
\[
[\Gamma_1:\Gamma_N]\ll N^{n^{2}}
\]
(and the exponent $n^2$ can usually be improved).

Combining the different estimates, for $d_{0}<2\frac{C}{2n^{2}}\ln([\Gamma_{1}:\Gamma_N])\le2\frac{C}{2n^{2}}N^{n^{2}})=C\ln(N)$,
it holds that for $N$ large enough,
\[
\#\left\{ \gamma\in\Gamma_{N}:l(\gamma)\le d_{0}\right\} =1.
\]

Therefore, the Weak Injective Radius Property is satisfied with parameter
$\alpha=\frac{C}{2n^{2}}$.
\begin{rem}
Pushing this argument further, we may take the parameter to be $\alpha=\frac{C}{d}$, where $d = \dim G$, at least when $G$ is split.
For example, for the principal congruence subgroups of $\SL_{2}(\Z)$
one gets this way the parameter $\alpha=\frac{2}{3}$, while $\alpha=1$
can be reached by a better analysis, see Subsection~\ref{subsec:Sarnak-Xue work} below.
\end{rem}

As a direct application of this fact we have:
\begin{cor}
\label{cor:principal congruence}Let $G$ be Archimedean, let $\Gamma_1$
be an arithmetic lattice and let $(\Gamma_N) $ be
the sequence of principle congruence subgroups of $\Gamma_1$. Then
the sequence $(\Gamma_N) $ satisfies the Weak Injective
Radius Property with parameter $\alpha=\alpha(\Gamma_1)$.
\end{cor}

As a matter of fact, this corollary can be easily be extended to principal
congruence subgroups of $S$-arithmetic groups, once those are properly
defined.

For arbitrary congruence subgroups, let us recall some of the results
of \citep{finis2018approximation} (the same result can be deduced
from \citep[Theorem 1.11 and Theorem 5.6]{abert2017growth}).

First, for $\gamma\in\Gamma_1$, we denote 
\[
c_{\Gamma_{N}}(\gamma)=\left|\left\{ y\in\Gamma_{N}\backslash\Gamma_{1}:y\gamma y^{-1}\in\Gamma_N\right\} \right|,
\]
which is the number of fixed points of the right action of $\gamma$
on $\Gamma_{N}\backslash\Gamma_1$. Then the Weak Injective Radius
Property with parameter $\alpha$ is equivalent to the fact that for
every $0\le d_{0}\le2\alpha\ln([\Gamma_1:\Gamma_N])$,
$\epsilon>0$,
\[
\frac{1}{[\Gamma_1:\Gamma_N]}\sum_{l(\gamma)\le d_{0}}c_{\Gamma_{N}}(\gamma)\ll_{\epsilon}[\Gamma_1:\Gamma_N]^{\epsilon}e^{d_{0}(\nicefrac{1}{2}+\epsilon)}.
\]

A congruence subgroup is a subgroup of one of the groups $\Gamma_1(N)$
as above.
\begin{thm}[{Following \citep[Corollary 5.9]{finis2018approximation}}]
\label{thm:non-principal congruence subgroups}Let $G$ be Archimedean,
semisimple, almost-simple and simply connected, let $\Gamma_1$
be an arithmetic lattice and let $(\Gamma_N) $ be
a sequence of congruence subgroups of $\Gamma_1$. Then there exist
constants $\mu>0,c>0$ depending only on $\Gamma_1$ such that for
every $d_{0}>0$, $N\ge1$ it holds that 
\begin{equation}
\frac{1}{[\Gamma_1:\Gamma_N]}\sum_{\gamma\in\Gamma_1,l(\gamma)\le d_{0}}c_{\Gamma_{N}}(\gamma)\ll1+e^{cd_{0}}[\Gamma_1:\Gamma_N]^{-\mu}.\label{eq:7samurai}
\end{equation}

In particular, the injective radius property holds with parameter
$\alpha=\frac{\mu}{c}$.
\end{thm}

\begin{proof}
By \citep[Corollary 5.9]{finis2018approximation}, there exists a constant $c'$ such that for every $\gamma\in\Gamma_1$ which does not belong to a proper normal subgroup of $G$ it holds that 
\[
\frac{1}{[\Gamma_1:\Gamma_N]}c_{\Gamma_{N}}(\gamma)\ll e^{c'l(\gamma)}[\Gamma_1:\Gamma_N]^{-\mu}.
\]
We remark that the dependence on $\gamma$ in \citep{finis2018approximation}
is different, but it is obvious that $e^{cl(\gamma)}$ for $c$ large enough is an upper bound on it. By our assumptions on $G$, the number of $\gamma\in\Gamma_1$ belonging to a proper normal subgroup is $\ll 1$.

By summing over all $\gamma\in\Gamma_1$ with $l(\gamma)\le d_{0}$
we get
\begin{align*}
\frac{1}{[\Gamma_1:\Gamma_N]}\sum_{\gamma\in\Gamma_1,l(\gamma)\le d_{0}}c_{\Gamma_{N}}(\gamma) & \ll\frac{[\Gamma_1:\Gamma_N]}{[\Gamma_1:\Gamma_N]}+e^{c'd_{0}}[\Gamma_1:\Gamma_N]^{-\mu}\sum_{\gamma\in\Gamma_1,l(\gamma)\le d_{0}}1\\
 & \ll1+e^{(1+c')d_{0}}[\Gamma_1:\Gamma_N]^{-\mu},
\end{align*}
which implies Equation~$\eqref{eq:7samurai}$ with $c=1+c'$.

When $d_{0}\le2\frac{\mu}{c}\ln([\Gamma_{1}:\Gamma_N])$,
or $[\Gamma_1:\Gamma_N]^{-\mu}\le e^{d_{0}c/2}$ we get
\[
\frac{1}{[\Gamma_1:\Gamma_N]}\sum_{\gamma\in\Gamma_1,l(\gamma)\le d_{0}}c_{\Gamma_{N}}(\gamma)\ll e^{d_{0}/2},
\]
which implies the Weak Injective Radius Property with parameter $\alpha=\frac{\mu}{c}$,
as needed.
\end{proof}
In the cocompact case, Theorem~\ref{thm:Diophantine implies pointwise}
implies the Pointwise Multiplicity Hypothesis with parameter $\alpha=\frac{\mu}{c}$,
namely 
\[
m(\pi,\Gamma_N)\ll_{\pi,\epsilon}[\Gamma_1:\Gamma_N]^{1-\frac{\mu}{c}(1-\nicefrac{2}{p(\pi)})+\epsilon}.
\]
It should be compared with \citep[Theorem 7.15]{abert2017growth},
which states
\[
m(\pi,\Gamma_N)\ll_{\pi}[\Gamma_1:\Gamma_N]^{1-\alpha(\pi)}.
\]

Finally, we with to make the following conjecture:
\begin{conjecture}
\label{conj: weak injective radius conjetcure}Let $G$ be Archimedean,
semisimple, almost-simple and simply connected, let $\Gamma_1$
be an arithmetic lattice and let $(\Gamma_N) $ be
a sequence of congruence subgroups of $\Gamma_1$. Then the sequence
$(\Gamma_N) $ satisfies the Weak Injective Radius
Property with parameter $\alpha=1$, and if $\Gamma_1$ is cocompact
then the sequence $(\Gamma_N) $ satisfies the General
Density Hypothesis with parameter $\alpha=1$.

As a corollary, the sequence $(\Gamma_N) $ has the
Optimal Lifting Property.
\end{conjecture}

The conjecture generalizes Conjecture~$1$ and Conjecture~$2$ from
the work of Sarnak and Xue (\citep{sarnak1991bounds}).
Finally, a similar conjecture should also hold in the $S$-arithmetic setting when $G$ is $p$-adic.

\subsection{\label{subsec:Sarnak-Xue work}The Work of Sarnak and Xue and its
Implications}

Sarnak and Xue proved the Weak Injective Radius Property with parameter
$\alpha=1$ for principal congruence subgroups of cocompact arithmetic lattices in $\SL_{2}(\R)$ and $\SL_{2}(\C)$,
and proved the Weak Injective Radius Property with parameter $\alpha=5/6$ for principal congruence subgroups of cocompact arithmetic lattices of $\SU(2,1)$.

We refer to \citep[Section 3]{sarnak1991bounds} for their calculations. Their argument actually works also for principal congruence subgroups of lattices in $\SL_2$ over $p$-adic fields coming from division algebras. See, e.g., \citep[Theorem 4.4.4]{davidoff2003elementary}.

As a simple example, let us prove here the Weak Injective Radius Property for the principal congruence subgroups of $\SL_{2}(\Z)$. In this case, we need to show that the number $\boldsymbol{N}(T,\Gamma(N))$,
of solutions to $ad-bc=1$, with $a\equiv d\equiv1\mod N$, $b\equiv c\equiv0\mod N$
and $\max\left\{ \left|a\right|,\left|b\right|,\left|c\right|,\left|d\right|\right\} \le T$,
for $T\le N^{3}$, is bounded by
\[
\boldsymbol{N}(T,\Gamma(N))\ll_{\epsilon}T^{1+\epsilon}.
\]

This is done as follows (see also \citep[Proposition 5.3]{gamburd2002spectral}).
From the congruence condition it follows that
\[
a+d-2=-(a-1)(d-1)+bc\equiv0\mod N^{2}.
\]
One may therefore choose $a+d$ in $(2T/N^{2}+1)$ ways,
and choose $(a,d)$ in $(2T/N^{2}+1)(2T/N+1)$
ways.

If $ad\ne1$, from $bc=1-ad$ and bounds on the divisor function,
there are $T^{\epsilon}$ ways of choosing $bc$. If $ad=1$ it is
also simple to bound the number of possibilities of $b,c$ by $4(T/N+1)$.
In total, we get for $T\le N^{3}$,
\[
\boldsymbol{N}(T,\Gamma(N))\ll_{\epsilon}T^{\epsilon}(T/N^{2}+1)(T/N+1)\ll_{\epsilon}T^{1+\epsilon}.
\]

\subsection{\label{subsec:Gamma0}On Some Congruence Subgroups of $\SL_3(\Z)$}

Consider the following subgroups of $\Gamma_{1}=\SL_{3}(\Z)$:
\begin{align*}
\Gamma_{0}(N)&=\left\{ \left(\begin{array}{ccc}
* & * & *\\*
* & * & *\\
a & b & *
\end{array}\right)\in \SL_{3}(\Z):a\equiv b\equiv0\mod N\right\} \subset \\
\Gamma_{2}(N)&=\left\{ \left(\begin{array}{ccc}
* & * & *\\*
c & * & *\\
a & b & *
\end{array}\right)\in \SL_{3}(\Z):a\equiv b\equiv c \equiv 0\mod N\right\} \subset\Gamma_{1}=\SL_{3}(\Z).
\end{align*}

In a companion paper by the second named author and Hagai Lavner (\citep{kamber2019optimal}), the following is proven:
\begin{thm}
The sequences of lattices $(\Gamma_{0}(N))$ and $(\Gamma_2(N))$, for $N$ prime,
have the Weak Injective Radius Property (with parameter $1$).

As a result, the sequences have the Optimal Lifting Property.
\end{thm}

We refer to \citep{kamber2019optimal} for an interpretation of this result in terms of the action of $\SL_{3}(\Z)$ on the projective plane over the field with $N$ elements, or on the corresponding flag space.

The work \citep{kamber2019optimal} is strongly influenced by a deep work of Blomer, Buttcane and Maga on $\Gamma_0(N)$ (\citep{blomer2017applications}), which is based on the Kuznetsov trace formula:
\begin{thm}[{\citep[Theorem 4]{blomer2017applications}}]
For $\pi\in\Pi(\SL_3(\R))_{\sph}$,
let $m_{\cusp}(\pi,\Gamma_{0}(N))$ be
the multiplicity of $\pi$ in the cuspidal part of $L^{2}(\Gamma_{0}(N)\backslash \SL_3(\R))$.
Then for every compact $A\subset\Pi(\SL_3(\R))_{\sph}$,
it holds that for every $N$ prime, $p>2$, $\epsilon>0$, 
\[
\sum_{\pi\in A,p(\pi)>0}m_{\cusp}(\pi,\Gamma_{0}(N))\ll_{A,\epsilon}[\Gamma_1:\Gamma_{0}(N)]^{1-2(1-2/p)+\epsilon}.
\]
\end{thm}
The result of the theorem is very similar to the Spherical Density Hypothesis with parameter $\alpha=2$. There are also a number of differences: first, $\SL_{3}(\Z)$ is not cocompact, so our discussion does not
apply to it. In particular, we have to deal with the continuous spectrum if we wish to deduce the Optimal Lifting Property. Secondly and less crucially, the dependence on the subset $A$ is not explicit, as needed in the definition of the Spherical Density Hypotheses, so it is hard to use it for geometric applications. 

This result was recently generalized by Blomer (\citep{blomer2019density}) to general $\SL_{M}(\Z)$, where the subgroup $\Gamma_{0}(N)$ is similarly defined to be the set of matrices with the entries in the last row, except for the $(M,M)$ entry, equal to $0$ modulo $N$. We remark that Blomer uses a slightly weaker way to measure temperedness, rather than $p(\pi)$ (for $\SL_3$ the two ways are equivalent).

In any case, it seems that the methods of \citep{blomer2017applications,blomer2019density} are not applicable to the subgroups of the form $\Gamma_2(N)$.

\subsection{\label{subsec:Weak-Injective-Radius-Principal-SLn}The Weak Injective
Radius of Principal Congruence Subgroups of $\SL_{n}(\protect\Z)$}

If the Weak Injective Radius Property would hold for the principal congruence subgroups of $\SL_{n}(\Z)$, then it will imply that
\[
\#\left\{ \gamma\in \SL_{n}(\Z):\gamma\equiv I\mod N,\n{\gamma}\le T\right\} \ll_\epsilon T^{\epsilon}N^{\epsilon}\left(\frac{T^{n^{2}-n}}{N^{n^{2}-1}}+T^{(n^{2}-n)/2}\right),
\]
where $\n{\cdot}$ is the maximal absolute value of an entry. One
may also try to improve this estimate, in particular the ``error''
part $T^{(n^{2}-n)/2}$.

One of the results of \citep{katznelson1993singular}, which is also
cited in \citep{sarnak1991bounds}, states that 
\[
\#\left\{ \gamma\in \SL_{n}(\Z):\gamma\equiv I\mod N,\n{\gamma}\le T\right\} \ll\frac{T^{n^{2}-n}}{N^{n^{2}-1}}\log T+1.
\]

As was pointed to us by Sarnak, the proof contains as error. As a
matter of fact, this naive estimate is actually \emph{false}, for simple
reasons. For example, in $\SL_{2}$, 
\[
\#\left\{ \left(\begin{array}{cc}
1 & *\\
0 & 1
\end{array}\right)\in \SL_{2}(\Z):\gamma\equiv I\mod N,\n{\gamma}\le T\right\} \asymp T/N+1.
\]
and this is larger than $\frac{T^{2}}{N^{3}}\log T$ in the range
$N^{1+\epsilon}<T<N^{2(1-\epsilon)}$. The same
argument works for every $n$, but there are $n(n-1)/2$
entries we can use. Therefore, we have the lower bound: 
\[
\#\left\{ \gamma\in \SL_{n}(\Z):\gamma\equiv I\mod N,\n{\gamma}\le T\right\} \gg\left(\frac{T^{n^{2}-n}}{N^{n^{2}-1}}+\left(\frac{T}{N}\right)^{(n^{2}-n)/2}+1\right).
\]

Up to $T^{\epsilon}$, this is also the upper bound for $n=2$. We
conjecture that up to $T^{\epsilon}$ this is also the upper bound
for larger values of $n$.

\section{\label{sec:Main-Ideas}Main Ideas of the Proofs}

In this section, we discuss some ideas from the proofs of the main theorems and the technical problems they lead to.

All the proofs are based on a reduction of the geometric properties into spectral ones. We first restate the Weak Injective Radius Property in terms of traces of operators.
\begin{defn}
\label{def:chi and psi}For $d_{0}\in\R_{\ge0}$, let $\chi_{d_0}\in C_{c}^{\infty}(K\backslash G/K)$
be a function that is equal to 1 for $l(g)\le d_{0}$, equal to
$0$ for $l(g)\ge d_{0}+1$ and for $d_{0}\le l(g)\le d_{0}+1$
it is defined between $0$ and 1 so that $\chi_{d_0}\in C_{c}^{\infty}(K\backslash G/K)$.

Let $\psi_{d_0}(g)\in C_{c}^{\infty}(K\backslash G/K)$ be $\psi_{d_0}(g)=q^{\nicefrac{(d_{0}-l(g))}{2}}\chi_{d_0}(g)$
for $l(g)\ge1$ and for $0\le l(g)\le1$ $\psi_{d_0}(g)$ is defined between $q^{(d_{0}-1)/2}$
and $q^{d_{0}/2}$ so that eventually $\psi_{d_0}(g)\in C_{c}^{\infty}(K\backslash G/K)$.
\end{defn}

As $\chi_{d_0},\psi_{d_0}\in C_{c}^{\infty}(K\backslash G/K)$,
they act naturally on $L^{2}(\Gamma_{N}\backslash G)$,
and moreover, are of trace class (\citep[Chapter 1]{gelfand1968representation},
see also Subsection~\ref{subsec:Traces-of-Operators}). The trivial eigenvalues of $\chi_{d_0},\psi_{d_0}$ satisfy $q^{d_0}\ll\intop_{G}\chi_{d_0}(g)dg\ll_{\epsilon}q^{d_{0}(1+\epsilon)}$ and $q^{d_0}\ll\intop_{G}\psi_{d_0}(g)dg\ll_{\epsilon}q^{d_{0}(1+\epsilon)}$.
Since $\chi_{d_0},\psi_{d_0}\in C_{c}(K\backslash G/K)$,
their action is actually on $L^{2}(\Gamma_{N}\backslash G/K)$.

The reason that we look at $\psi_{d_0}$, and one of the main reasons
we use the length $l(g)$, is the following Convolution
Lemma, which replaces the rank $1$ case in \citep[Lemma 3.1]{sarnak1991bounds}.
The lemma says that $\psi_{d_0}$ provides an approximated upper
bound for the convolution of $\chi_{\nicefrac{d_0}{2}}$ with itself:
\begin{lem}[See Lemma~\ref{lem:convolution lemma}]
\label{lem:convolution lemma intro}It holds that $c_{d_0}=\chi_{d_0}\ast\chi_{d_0}\in C_{c}^{\infty}(K\backslash G/K)$,
and for every $g\in G$ and $\epsilon>0$, we have
\[
c_{d_0}(g)\ll_{\epsilon}q^{d_{0}\epsilon}\psi_{2d_{0}}(g).
\]
\end{lem}

Using the pre-trace formula (see Subsection~\ref{subsec:The trace formula}),
we get:
\begin{prop}
\label{prop:Trace equivalent to Lattice point Counting}The following
are equivalent for a sequence $(\Gamma_N)$ of cocompact lattices:
\begin{enumerate}
\item The sequence $(\Gamma_N) $ satisfies the Weak Injective
Radius Property with parameter $\alpha$.
\item For every $0\le d_{0}\le2\alpha\log_{q}([\Gamma_1:\Gamma_N])$,
$\epsilon>0$, 
\begin{equation}
\tr\chi_{d_0}|_{L^{2}(\Gamma_{N}\backslash G)}\ll_{\epsilon}[\Gamma_1:\Gamma_N]^{1+\epsilon}q^{d_{0}(\nicefrac{1}{2}+\epsilon)}.\label{eq:trace formula to prove}
\end{equation}
\item For every $\epsilon>0$, for every $h\in C_c(G)$ satisfying
$h(g)\ll_{\epsilon}[\Gamma_1:\Gamma_N]^{\epsilon}\psi_{2\alpha\log_{q}([\Gamma_1:\Gamma_N])}(g)$,
it holds that
\begin{equation}
\tr h|_{L^{2}(\Gamma_{N}\backslash G)}\ll_{\epsilon}[\Gamma_1:\Gamma_N]^{1+\alpha+\epsilon}.\label{eq:trace formula to use}
\end{equation}
\end{enumerate}
\end{prop}

The proposition allows one to prove the Weak Injective Radius Property
using the trace formula and Equation~\eqref{eq:trace formula to prove}.
It also allows proving various multiplicity results using
Equation~\eqref{eq:trace formula to use}.

By combining the Convolution Lemma~\ref{lem:convolution lemma intro},
Spectral Gap, and a version of Equation~\eqref{eq:trace formula to use},
we prove Theorem~\ref{thm:injective radius implies optimal lifting}.
See Section~\ref{sec:Injective Radius implies Lifitng}.

The direction ``spectral to geometric'' uses carefully
Harish-Chandra's upper bounds on spherical functions, and its generalization to arbitrary matrix coefficients of representations. Our analysis is closely related to the work of Ghosh, Gorodnik and Nevo on Diophantine
exponents (\citep{ghosh2013diophantine,ghosh2014best} and the reference
therein). Very generally, the main difference between our analysis
and theirs is that we assume Density, while they assume Spectral Gap.
The new idea is to note that one applies the spectral estimates to characteristic functions of ``small sets'', and if there are few ``bad eigenvectors'' they correlate poorly with such functions.

Let us state some concrete results. Let $\Xi(g)=\int\delta(gk)^{-\nicefrac{1}{2}}dk$
be Harish-Chandra's function of $G$. This function satisfies $\left|\Xi(g)\right|\ll_{\epsilon}q^{-l(g)(\nicefrac{1}{2}+\epsilon)}$
(which is another motivation for the definition of $l$). Now, combining the two theorems of \citep{cowling1988almost}, we get:
\begin{thm}[\citep{cowling1988almost}]
Let $(\pi,V)$ be a unitary irreducible representation of $G$ with $p(\pi)\le2$, and let $v_1,v_{2}\in V$ be two $K$-finite vectors, such that $\dim\spann Kv_{1}=d_{1}$, $\dim\spann Kv_{2}=d_{2}$. Then 
\[
\left|\left\langle v_1,gv_{2}\right\rangle \right|\le\sqrt{d_{1}d_{2}}\n{v_{1}}\n{v_{2}}\Xi(g).
\]
\end{thm}

However, this theorem is not sufficient for our uses, since we wish to consider $p(\pi)$ arbitrary. We remark that \citep{cowling1988almost}
does give some results for arbitrary $\pi$ and $p$, but they are not precise enough for us.

Let us state here a general theorem which generalizes the above, which gives a bound of the form that we need. For
$2\le p\le\infty$, let $\Xi_{p}(g)=\int\delta(gk)^{-\nicefrac{1}{p}}dk$
be the $p$-th version of Harish-Chandra's function of $G$. This function satisfies $\Xi_{p}(g)\ll\Xi(g)q^{-l(g)(\nicefrac{1}{p}-\nicefrac{1}{2})}\ll_{\epsilon}q^{-l(g)(\nicefrac{1}{p}+\epsilon)}$.
Then we have:
\begin{thm}[See Theorem~\ref{thm:CHH-Spherical Lp}]
\label{thm:Upper bounds intro}Let $(\pi,V)$ be a unitary
irreducible representation of $G$ with $p(\pi)\le p$,
$p\ge2$, and let $v_1,v_{2}\in V$ be two $K$-finite vectors,
such that $\dim\spann Kv_{1}=d_{1}$, $\dim\spann Kv_{2}=d_{2}$.
Then
\[
\left|\left\langle v_1,gv_{2}\right\rangle \right|\le\sqrt{d_{1}d_{2}}\n{v_{1}}\n{v_{2}}\Xi_{p}(g).
\]
\end{thm}

We discuss such bounds further in Subsection~\ref{subsec:Harish-Chandra bounds}.
Using Theorem~\ref{thm:Upper bounds intro} one may deduce various \emph{upper bounds} on matrix coefficients, norms of operators, and traces of operators. In particular, a useful bound is:
\begin{thm}[See Corollary~\ref{cor:Hecke Bounds}]
\label{thm:Upper Bounds Xi Intro}Let $(\pi,V)$ be a
unitary irreducible representation of $G$ with $p(\pi)=p$,
$p\ge2$. Then for every $d_{0}>0$ and $\epsilon>0$ 
\[
\n{\pi(\chi_{d_0})}\ll_{\epsilon}q^{d_{0}(1-\nicefrac{1}{p}+\epsilon)}.
\]
\end{thm}

Applying Theorem~\ref{thm:Upper Bounds Xi Intro} carefully allows us to deduce Theorem~\ref{thm:density implies counting}.

To prove Theorem~\ref{thm:Diophantine implies pointwise}, Theorem
\ref{thm:rank one injective radius implies density} and Theorem~\ref{thm:p-adic injective radius implies density},
which deduce multiplicity bounds from the Weak Injective Radius Property,
we use Equation~\eqref{eq:trace formula to use} of Proposition~\ref{prop:Trace equivalent to Lattice point Counting}.
To deduce from it upper bounds on multiplicity, one needs \emph{lower bounds} on traces on representations of functions $h\in C_{c}^{\infty}(G)$.
The following definition and proposition capture the situation:
\begin{defn}
\label{def:Good family}Let $A\subset\Pi(G)$ be a precompact
subset. We say that a family of functions $\left\{ f_{d_0}\right\} \subset C_{c}^{\infty}(G)$,
$d_{0}\in\R$, $d_{0}\ge D$ is \emph{good} for $A$ if it holds that:
\begin{enumerate}
\item For every $\pi\in A$, and $\epsilon>0$,
\[
q^{d_{0}(1-\nicefrac{1}{p(\pi)}-\epsilon)}\ll_{A,\epsilon}\tr(\pi(f_{d_0})).
\]
\item It holds that for every $g\in G$, $\epsilon>0$, $f_{d_0}(g)\ll_{\epsilon}q^{d_{0}\epsilon}\psi_{d_0}(g)$,
where $\psi_{d_0}(g)$ is from Definition~\ref{def:chi and psi}.
\item For every representation $\pi'\in\Pi(G)$, it holds that
$0\le\tr (\pi'(f_{d_0}))$.
\end{enumerate}
\end{defn}

\begin{rem}
If we replace $(2)$ by the slightly stronger condition $\left|f_{d_0}(g)\right|\ll_{\epsilon}q^{d_{0}\epsilon}\psi_{d_0}(g)$
and further assumes that $f_{d_0}$ is left and right $K$-finite, then
one actually has by Theorem~\ref{thm:Upper bounds intro} 
\[
\tr (\pi(f_{d_0}))\ll_{\pi,\epsilon}q^{d_{0}(1-\nicefrac{1}{p(\pi)}+\epsilon)},
\]
so the lower bound is rather tight.
\end{rem}

\begin{prop}
\label{prop:good family proposition}Let $A\subset\Pi(G)$
be a precompact subset, and assume that it has a good family of functions.
Under this condition, if the sequence $(\Gamma_N) $
of cocompact lattices satisfies the Weak Injective Radius Property
with parameter $\alpha$ then for every $N\ge1$, $p>2$, $\epsilon>0$,
\[
M(A,\Gamma_N,p)\ll_{A,\epsilon}[\Gamma_1:\Gamma_N]^{1-\alpha(1-\nicefrac{2}{p})+\epsilon}.
\]
\end{prop}

Finding general lower bounds on traces (uniformly for a family of
representations) is not well studied. Two special cases, which appear (somewhat implicitly) in the work of Sarnak and Xue, correspond to
Theorem~\ref{thm:rank one injective radius implies density} and
Theorem~\ref{thm:Diophantine implies pointwise}:
\begin{enumerate}
\item In rank $1$, one has a simple classification of spherical irreducible
unitary representations. In the Archimedean case, for each $2<p\le\infty$
there is at most a single spherical irreducible unitary representation
$(\pi,V)$ with $p(\pi)=p$ (with a corresponding
spherical function $\Xi_{p}(g)$), and one can easily deduce lower
bounds on the trace of $f_{d_0}=\chi_{d_{0}/2}*\chi_{d_{0}/2}$,
and deduce Theorem~\ref{thm:rank one injective radius implies density}.
In the non-Archimedean case, Theorem~\ref{thm:rank one injective radius implies density}
reduces to some statement on graphs, see Subsection
\ref{subsec:Lower-Bounds-rank1}.
\item If one is interested in a single representation, one has the following, from which we deduce Theorem~\ref{thm:Diophantine implies pointwise}.
In the Archimedean case, it follows from the asymptotic behavior of
leading exponents (\citep[Chapter VIII]{casselman1982asymptotic,knapp2016representation}).
The non-Archimedean case is easier, and in any case follows from Theorem~\ref{thm:good family padic intro}
below.
\end{enumerate}
\begin{thm}[See Section~\ref{sec:pointwise lower bound}]
\label{thm: good family pointwise intro}Let $(\pi,V)\in\Pi(G)$.
Then the set $A=\left\{ \pi\right\} $ has a good family of functions.
\end{thm}

Finally, we provide the following theorem, which implies Theorem~\ref{thm:p-adic injective radius implies density}
together with Proposition~\ref{prop:good family proposition}.
\begin{thm}[See Theorem~\ref{thm:Good family p-adic}]
\label{thm:good family padic intro}Let $G$ be non-Archimedean.
Then there exists a set $\left\{ K'\right\} $ of arbitrarily small
open-compact subgroups of $G$, such that for every $K'$, $\Pi(G)_{K'-sph}$
has a good family.
\end{thm}

The proof of Theorem~\ref{thm:good family padic intro} is based on two sources. The first is the connection between the Ihara graph Zeta
function and expansion (see \citep{hashimoto1989zeta}), and the second
is Bernstein's description of the Hecke algebra $C_{c}(K'\backslash G/K')$
(see \citep{bernshtein1974all}). A precise connection for $(q+1)$-regular
graphs between $p(\pi)$ for spherical function and poles
of the Ihara zeta function may be found in \citep{kamber2019p}. In recent years, there were various generalizations of the graph zeta function to higher dimensional buildings (see, e.g., \citep{kang2010zeta} and the references within). In \citep{kamber2016lpcomplex} the second
named author generalized the connection between $p(\pi)$
for representations $\pi\in\Pi(G)$ with Iwahori-fixed
vector and the poles of some generalized zeta function. By a slight
variant of those ideas one may prove the special case of Theorem~\ref{thm:good family padic intro}
when $K'$ is the Iwahori subgroup. For more general $K'$ we follow
the same ideas, by using Bernstein's description of the Hecke algebra
$C_{c}(K'\backslash G/K')$. See Section~\ref{sec:Bernstein-Theory-of-NonBacktracking}
for details.

If we consider only the spherical case, it would be useful if the
functions $f_{d_0}$ in the definition of a good family will be
left and right $K$-invariant, i.e., $f_{d_0}\in C_{c}^{\infty}(K\backslash G/K)$.
Recently, Matz and Templier proved a similar theorem for $G=\PGL_{n}$
using the Satake isomorphism (\citep{matz2021sato}). However, their
results are less precise than we desire -- they find a spherical function $f_{d_0}\in C_{c}^{\infty}(K\backslash G/K)$ which satisfies $f_{d_0}(g)\ll_{\epsilon}q^{d_{0}\epsilon}\psi_{d_0}(g)$,
with a lower bound $q^{\beta d_{0}(1-\nicefrac{1}{p(\pi)})}\ll\tr (\pi(f_{d_0}))$
for some $\beta<1$, instead of the optimal bound $q^{d_{0}(1-\nicefrac{1}{p(\pi)})}\ll\tr (\pi(f_{d_0}))$.

Let us finish this discussion with the following conjecture, which concerns only spherical functions. For $g\in G$, let $S(g)$ be the $K$-bi-invariant function such that $\intop_{G}f(g)dg=\intop_{K}\intop_{K}\intop_{A_{+}}f(kak')S(a)dadkdk'$
. It holds that for $g$ ``far from the walls'' $S(g)\approx q^{l(g)}$
(see Subsection~\ref{subsec:Distnaces and Length of Elements}).
\begin{conjecture}
\label{conj:mat-coeff-lower-bound}There exist $D>0$, $L>0$ such
that for every $\epsilon>0$ and every $(\pi,V)\in\Pi(G)_{\sph}$
(i.e. a unitary irreducible spherical representation) with $p(\pi)>2$,
if $v\in V$, $\n v=1$ is $K$-fixed, then:
\begin{enumerate}
\item In the non-Archimedean case, for $d_{0}>D$
\[
\intop_{l(g)\le d_{0}}S(g)\left|\left\langle v,\pi(g)v\right\rangle \right|^{2}dg\gg_{\epsilon}q^{2d_{0}(1-1/p(\pi)-\epsilon)}.
\]
\item In the Archimedean case, for $d_{0}>D$,
\[
\intop_{l(g)\le d_{0}}S(g)\left|\left\langle v,\pi(g)v\right\rangle \right|^{2}dg\gg_{\epsilon}(\lambda(\pi)+1)^{-L}q^{2d_{0}(1-1/p(\pi)-\epsilon)}.
\]
\end{enumerate}
\end{conjecture}

The exponents in the conjecture are tight, as the corresponding upper
bounds can be deduced from Theorem~\ref{thm:Upper bounds intro}.

\section{\label{sec:Preliminaries}Preliminaries}

\subsection{\label{subsec:Distnaces and Length of Elements}Distances and Length
of Elements}

Besides our definition of length, the following is standard, see e.g.,
\citep[Section 3]{ghosh2013diophantine}. We mainly follow \citep{knapp2016representation}
when $G$ is Archimedean and \citep{cartier1979representations} when $G$ is non-Archimedean.

Let $k$ be $\R$ or a $p$-adic field, and $\left|\cdot\right|_{k}\colon k\to\R_{+}$
its standard non-trivial valuation. Let $G$ be a semisimple non-compact
algebraic group over $k$, of $k$-rank $r$. Let $T\cong G_{m}^{r}\subset G$
be a maximal $k$-split torus. The choice of $T$ determines the set
of weights $X^{\ast}(T)$, i.e., of rational characters
of $T$. Let $\Phi(T,G)\subset X^{\ast}(T)$
be the set of roots of $G$ with respect to $T$.

In the Archimedean case, if $T_{0}\cong\left\{ \pm1\right\} ^{r}$ is the
maximal compact subgroup of $T$, the connected component $A\cong T/T_{0}$ of the identity of $T$ is the Lie group of a real Cartan subalgebra
$\a$ of $\mathfrak{g}$, and we define $\nu\colon T\to A\to\a\cong\R^{r}$
by the logarithm map.

In the non-Archimedean case, let $T_{0}$ be a maximal compact subgroup of $T$. Then $T/T_{0}\cong\Z^{r}$, this identification defines $\nu\colon T\to\Z^{r}\subset\R^{r}$,
and we identify $\R^{r}$ with $\a$.

Let $X^{\ast}(T)_{\R}\cong X^{\ast}(T)\otimes\R$
be the weight space. For an element $\alpha\in X^{\ast}(T)$,
we let $\chi_{\alpha}\in\a^{\ast}$ be the linear functional
defined such that for $t\in T$ $\left|\alpha(t)\right|_{k}=q^{\chi_{\alpha}(\nu(t))}$,
where $q=e$ in the Archimedean case and otherwise the size of the quotient field of $k$. For $\alpha\in X^{\ast}(T)_{\R}$ we define $\chi_{\alpha}\in\a^{\ast}$ by extension of the action above. This isomorphism (as linear spaces) between $X^{\ast}(T)_{\R}$
and $\a^{*}$ defines an isomorphism between the coweight
space $(X^{*}(T)_{\R})^{*}$ and $\a$.

Choose an ordering on the root system which defines the positive roots
$\Phi_{+}\subset\Phi(T,G)$ and let $\Delta=\left\{ \alpha_1,...,\alpha_{r}\right\} \subset X^{*}(T)_{\R}\cong\a^{*}$
be the simple roots of $\Phi$ with respect to this ordering. Let $\left\{ \omega_1,...,\omega_{r}\right\} \subset\a$
be the set of simple coweights, i.e. the dual basis
to $\Delta$. The set $\left\{ \sum_{i=1}^{r}x_{i}\omega_{i}:x_{i}\ge0\right\} \subset\a$
is called the dominant sector, or the positive Weyl chamber. It is
isomorphic to $\a/W$, where $W$ is the Weyl group of the
root system.

Let $\Phi_{+}^{\vee}\subset\a$ be the corresponding coroot
system, let $\Delta^{\vee}=\left\{ \alpha_{1}^{\vee},...,\alpha_{r}^{\vee}\right\} $
be the dual basis of simple coroots and let $\left\{ \omega_{1}^{\vee},...,\omega_{r}^{\vee}\right\} \subset\a^{*}$
be the set of simple weights. We define a partial ordering
on $\a$ by $\alpha\ge_{\a}\alpha'$ if and only
if $\omega_{i}^{\vee}(\alpha)\ge\omega_{i}^{\vee}(\alpha)$
for every simple weight, or alternatively $\alpha-\alpha'$
is a non-negative sum of elements of $\Delta^{\vee}$.

Let $P$ be the Borel subgroup with respect to the set of positive
roots. It holds that $P=MN$, where $M$ is the centralizer of $T$
in $G$ and $N$ is the unipotent radical of $P$ (\citep[p. 134]{cartier1979representations}).
Let $K$ be a maximal special compact open subgroup (i.e., in the
$p$-adic case we choose it to be ``good'' in the sense of Bruhat
and Tits \citep{bruhat1972groupes}). The Iwasawa decomposition $G=KP$
holds (\citep[Proposition 1.2]{knapp2016representation},\citep[p. 140]{cartier1979representations}).
It holds that $M=(M\cap K)\cdot T$, and we extend $\nu\colon M\to\R^{r}$
by $\nu(k)=1$ for $k\in M\cap K$.

Let us recall the Cartan decomposition $G=KA_{+}K$:
\begin{itemize}
\item In the real case, following \citep[Theorem 5.20]{knapp2016representation},
$A\cong T/T_{0}$ is the Lie group of the Cartan subalgebra $\a$
of $\mathfrak{g}$. $A_{+}$ is the exponent of the closure of the
dominant sector in $\a$, i.e. the set of elements $A_{+}=\left\{ t\in A:\forall\alpha\in\Phi_{+},\alpha(t)\ge1\right\} $.
It is isomorphic (as a set, using the exponential map) to the dominant
sector in $\a$.
\item In the $p$-adic case, following \citep[p. 140]{cartier1979representations},
let $\Lambda\cong M/M^{0}$, where $M^{0}$ is the maximal compact subgroup of $M$. We identify elements of $\Lambda$ with elements of $M\subset P$. Elements of the weight space are unramified characters of $M$, and therefore are characters of $\Lambda$ as well. We let $A_{+}=\left\{ \lambda\in\Lambda:\forall\alpha\in\Phi_{+},\alpha(\lambda)\ge1\right\}$.
The action $\nu\colon T/T_{0}\to\a$ extends to $\nu\colon M/M^{0}\to\a$.
Then $A_{+}$ is isomorphic as a set with the intersection of $\nu(\Lambda)$
with the dominant sector in $\a$ . It is also isomorphic to a subset of the special vertices in the dominant sector in an apartment in the Bruhat-Tits building of $G$.
\end{itemize}
In both cases, there exists a map $\nu\colon A_{+}\to\a$. It
is also useful to extend it to a map $H\colon G\to\a$, using the Iwasawa decomposition $G=KMN$, by $H(kmn)=\nu(m)$.

We will need the following fundamental technical lemma:
\begin{lem}
\label{lem:HFunctionTechnicalLemma}For $a\in A_{+}$ and $k\in K$
\[
H(ak)\le_{\a}H(a).
\]
\end{lem}

\begin{proof}
For the non-Archimedean case see \citep[Proposition 4.4.4(i)]{bruhat1972groupes}.
For the Archimedean case see \citep[Corollary 3.5.3]{gangolli1988harmonic}.
\end{proof}
\begin{cor}
\label{Cor:KaKaK technical Lemma}Let $a,a',a''\in A_{+}$. If $KaKa'K\cap Ka''K\ne\phi$
then $\nu(a'')\le_{\a}\nu(a)+\nu(a')$.
\end{cor}

\begin{proof}
We first notice that the following property of the $H$-function:
for $k\in K$, $g\in G$, $m\in M$, $n\in N$ 
\begin{align*}
H(kgmn) & =H(g)+H(m)
\end{align*}

Now, if $KaKa'K\cap Ka''K\ne\phi$ then $a''=k_{0}ak_{1}a'k_{2}$.
Applying the Iwasawa decomposition to $a'k_{2}$ we have $a'k_{2}=k_{2}^{\prime}mn$
with $H(m)\le_{\a}H(a')$ by Lemma~\ref{lem:HFunctionTechnicalLemma}.
Applying Lemma~\ref{lem:HFunctionTechnicalLemma} again, we have 
\begin{align*}
\nu(a'') & =H(k_{0}ak_{1}k_{2}^{\prime}mn)=H(k_{0}ak_{1}k_{2}^{\prime})+H(m)\\
 & \le_{\a}H(a)+H(a')=\nu(a)+\nu(a').
\end{align*}
\end{proof}
Corollary~\ref{Cor:KaKaK technical Lemma} in the non-Archimedean
case is \citep[Proposition 4.4.4(iii)]{bruhat1972groupes}, and is
deduced from Lemma~\ref{lem:HFunctionTechnicalLemma} in the same
way.

Let $\delta(\tilde{p})$ be the left modular character of $P$, i.e.
if $d\tilde{p}$ is a left Haar measure on $P$ then $\delta(\tilde{p})d\tilde{p}$
is a right Haar measure. Normalize the measures so that for $f\in C_c(G)$,
$\int_{G}f(g)dg=\int_{K}\int_{P}f(k\tilde{p})\delta(\tilde{p})d\tilde{p}dk=\int_{K}\int_{P}f(\tilde{p}k)d\tilde{p}dk$,
$\int_{K}dk=1$.

It can also be defined as follows: $M$ acts by conjugation on the
Lie algebra $\mathfrak{n}$ of $N$. Then for $m\in M$, $\delta(m)=\left|\Det\Ad_{\mathfrak{n}}(m)\right|_{k}$
(\citep[p. 135]{cartier1979representations}, \citep[Proposition 5.25]{knapp2016representation}).
Unwinding the definitions, $\delta(m)=q^{2\rho(\nu(m))}$,
where $\rho=\frac{1}{2}\sum_{\alpha\in\Phi_{+}}(\dim\mathfrak{g}_{\alpha})\chi_{\alpha}$.
Here $\mathfrak{g}_{\alpha}$ is the root space of $\alpha$ in the
Lie algebra $\mathfrak{g}$ of $G$. Also recall that $q=e$ for $k=\R$
and otherwise is the size of the quotient field of $k$. As an example,
for $G=\SL_{n}(\R)$, and the matrix $a=\diag(a_{0},\dots,a_{n-1})$,
$\delta(a)=\prod a_{i}^{n-1-2i}$.

We associate with each element $a\in A_{+}\subset G$ a length $l\colon A_{+}\to\R_{\ge0}$
by $l(a)=\log_{q}\delta(a)=2\rho(a)$. We extend $l\colon G\to\R_{\ge0}$
by $l(kak')=l(a)$. By definition, $l$ is left and right
$K$-invariant.

For $a\in T$, we may identify $l(a)$ by the entropy (taken
with logarithm in base $q$) of the dynamical system of translation
of $\Gamma\backslash G$ by $a$, with respect to the Haar measure
($\Gamma$ here is an arbitrary lattice, see \citep[Theorem 7.9]{einsiedler2010diagonal}).
Using this fact, we have for $a\in A$, $l(a)=l(a^{-1})$
and therefore for every $g\in G$, $l(g)=l(g^{-1})$.
The same fact can be proven directly.
\begin{prop}
For $g_1,g_{2}\in G$, it holds that $l(g_{1}g_{2})\le l(g_{1})+l(g_{2})$.
\end{prop}

\begin{proof}
The proposition actually states that if $KaKa'K\cap Ka''K\ne\phi$
for $a,a',a''\in A_{+}$, then $l(a'')\le l(a)+l(a')=l(aa')$.
Since $l(a)=2\rho(a)$ it follows from Corollary~\ref{Cor:KaKaK technical Lemma}.
\end{proof}
For $a\in A_{+}$, define $S(a)=\int\delta(ak)dk$. For $f\in C_c(G)$,
we have 
\[
\int f(g)dg=\int_{K}\int_{K}\int_{A_{+}}f(kgk')S(a)dkdk'da.
\]
We interpret $S(a)$ as the measure of the ``circle'' $KaK$. In
the non-Archimedean case $S(a)\asymp\delta(a)$ (\citep[p. 141]{cartier1979representations}).
In the Archimedean case, by \citep[Proposition 5.28]{knapp2016representation},
\[
S(a)=\prod_{\alpha\in\Phi_{+}}(\sinh(\chi_{\alpha}(a)))^{\dim\mathfrak{g}_{\alpha}}.
\]
Since $\sinh(x)\asymp_{\beta}e^{x}$ for $x>\beta$, for $a\in A_{+}$
``far from the walls'', i.e. with $\chi_{\alpha}(a)>\beta$
for every $\alpha\in\Delta$ (and therefore $\chi_{\alpha}(a)>\beta$
for every $\alpha\in\Phi_{+}$), we have $S(a)\asymp_\beta\delta(a)=e^{l(a)}$.
Near the walls where $\sinh x\approx x$ this approximation fails, but we still have $S(a)\ll\delta(a)$. In any case, if we choose some norm $\n{\cdot}_{\a}$ on $A_{+}$ then for every $\tau>0$, $a\in A_{+}$, $\mu(\left\{ g:\n{a_{g}-a}_{\a}\le\tau\right\} )\asymp_{\tau}q^{l(a)}$,
since the set $\left\{ a'\in A_{+}:\n{a'-a}_{\a}\le\tau\right\} $
contains elements which are far enough (with respect to $\tau$) from the walls.

We deduce that the size of balls for $l\gg1$, 
\begin{equation}
q^{l}\ll\mu(\cup_{a':l(a')\le l}Ka'K)=\intop_{a\in A_{+}:l(a)\le l}S(a)da\ll p(l)q^{l}\ll_{\epsilon}q^{l(1+\epsilon)},\label{eq:size of ball-1}
\end{equation}
 for some polynomial $p$.
\begin{rem}
The literature has two popular choices of ``distance'' or ``length''
on $G/K$ or $G$.
\begin{enumerate}
\item For $G\subset \SL_{n}$, where $K=G\cap K'$ for $K'$ maximal compact
in $\SL_{n}$, one defines $\tilde{l}(g)=\log\n g$, where
$\n{\cdot}$ is some matrix norm on $\GL_{n}$. Recall that we are
mainly interested in distances as $g\to\infty$, so the specific choice
of matrix norm does not matter. Such choice (without calling it a
distance) is studied in \citep{sarnak1991bounds,duke1993density}.
\item For $G$ Archimedean, let $\mathfrak{g},\mathfrak{t}$ be the Lie
algebras of $G$ and $K$, and let $B\colon \mathfrak{g}\times\mathfrak{g}\to\C$
the Killing form of $G$. Let $\mathfrak{p}=\left\{ X\in\mathfrak{g}:B(X,Y)=0\,\forall Y\in\mathfrak{t}\right\} $.
Then $B|_{\mathfrak{p}\mathfrak{\times}\mathfrak{p}}$ is positive
definite. $\mathfrak{p}$ can be identified with the tangent space of $G/K$ at the identity, and it defines a natural Riemannian structure on $G/K$, with length $\hat{l}(g)=\hat{d}(g,1)$. See, e.g., \citep[Section 2]{degeorge1978limit}, \citep{sarnak1991bounds}.
A similar distance is used in the $p$-adic case in \citep[2.3]{tits1979reductive}.
\end{enumerate}
Since we mostly care about far distances in the group, and since $K$ is compact, by the Cartan decomposition it suffices to compare $l$ to other distances on $A_{+}$. In general, $l$ and $\tilde{l}$ look like an $L^{1}$-norm on $A_{+}$, and $\hat{l}$ looks like an $L^{2}$-norm.

Let us concentrate on $G=\SL_{n}(\R)$ and $\tilde{l}(g)=\ln(\n g_{2})$,
$\n g_{2}^{2}=\tr (gg^{t})$. For $G=\SL_{2}(\R)$,
its symmetric space is the hyperbolic plane with the standard metric
of curvature $-1$, and $l$ we defined above coincides with the hyperbolic
metric. For example, consider the matrix $g=\left(\begin{array}{cc}
e^{\frac{t}{2}} & 0\\
0 & e^{-\frac{t}{2}}
\end{array}\right)$. Then $l(g)=t$, and $\tilde{l}(g)=\frac{1}{2}\ln(e^{t}+e^{-t})\approx\frac{t}{2}$
. In fact, for every $g\in \SL_{2}(\R)$, $l(g)-2\tilde{l}(g)=O(1)$,
and $l$ and $\tilde{l}$ are equal up to an additive constant.

For $G=\SL_{3}(\R)$ it is no longer true. For \[
g_{1}=\left(\begin{array}{ccc}
e^{t/3}\\
 & e^{t/3}\\
 &  & e^{-2t/3}
\end{array}\right), g_{2}=\left(\begin{array}{ccc}
e^{2t/3}\\
 & e^{-t/3}\\
 &  & e^{-t/3}
\end{array}\right),
\]
it holds that $l(g_{1})=l(g_{2})=2t$, while $\tilde{l}(g_{1})\approx t/3$,
$\tilde{l}(g_{2})\approx2t/3$. So the two distances $l,\tilde{l}$
are not equal up to an additive constant, but are only Lipschitz-equivalent, with $\frac{3}{2}\tilde{l}\le l+O(1)\le3\tilde{l}$.
However, if we chose $\tilde{\tilde{l}}(g)=2(\tilde{l}(g)+\tilde{l}(g^{-1}))$, then $l(g)-2\tilde{\tilde{l}}(g)=O(1)$.
This solution no longer works for $\SL_{4}(\R)$.
\end{rem}

\subsection{\label{subsec:Harish-Chandra bounds}Growth of Matrix Coefficients
and Harish-Chandra's Bounds}

For $2\le p\le\infty$, let $\Xi_{p}(g)=\int\delta(gk)^{-\nicefrac{1}{p}}dk$ be the $p$-th version of Harish-Chandra's function. Let $\Xi(g)=\Xi_{2}(g)$ be the standard Harish-Chandra's function. Note that since $\Xi(g)$
is left and right $K$-invariant it only depends on $a\in A_{+}$ from the Cartan decomposition of $g$.

An explicit upper bound on Harish-Chandra's function is given by the following Theorem:
\begin{thm}
\label{thm: Explicit Bounds}There is a constant $C$ such that for every $g\in G$ and $\epsilon>0$,
\[
\Xi_{p}(g)\le\Xi^{\nicefrac{2}{p}}(g)\le q^{(\nicefrac{1}{2}-\nicefrac{1}{p})l(g)}\Xi(g)\ll (l(g)+1)^{C}q^{-\nicefrac{l(g)}{p}}\ll_{\epsilon}q^{-l(g)(\nicefrac{1}{p}-\epsilon)}.
\]
\end{thm}

\begin{proof}
By the Cartan decomposition, it suffices to verify the theorem for $a\in A_{+}$, where $q^{l(a)}=\delta(a)$. The first inequality follows from convexity, the second inequality follows from $\Xi(g)\ge q^{-l(g)/2}$ (see Equation~\eqref{eq:lower bound spherical}), and the fourth inequality is trivial. We are left with the third inequality.

For the Archimedean case, see \citep[Theorem 3]{harish1958spherical}
or \citep[Proposition 7.15]{knapp2016representation}. For the non-Archimedean case the standard reference is \citep[4.2.1]{silberger2015introduction}, (where there is an assumption that $\charr k=0$), but the theorem is well known to experts. In any case, the general non-Archimedean
result can be deduced from the results of \citep{kamber2016lpcomplex} for arbitrary affine buildings.
\end{proof}

A representation $(\pi,V)$ of $G$ is called tempered if $p(\pi)\le2$. The following theorem is the standard reference for upper bounds on matrix coefficients:
\begin{thm}[\citep{cowling1988almost}]
\label{thm:CHH-Thm2}Let $(\pi,V)$ be a unitary irreducible tempered representation of $G$ and let $v_1,v_{2}\in V$ be two $K$-finite vectors, such that $\dim \spann Kv_{1}=d_{1}$, $\dim\spann Kv_{2}=d_{2}$.
Then $\left|\left\langle v_1,gv_{2}\right\rangle \right|\le\sqrt{d_{1}d_{2}}\n{v_{1}}\n{v_{2}}\Xi(g)$.
\end{thm}

The work \citep{cowling1988almost} also provides a bound when $p(\pi)>2$: 
\begin{thm}[\citep{cowling1988almost}]
\label{thm:CHH-corollary}Let $(\pi,V)$ be a unitary
irreducible representation of $G$ with $p(\pi)\le2k$,
$k\in\N$ and let $v_1,v_{2}\in V$ be two $K$-finite vectors,
such that $\dim\spann Kv_{1}=d_{1}$, $\dim\spann K v_{2}=d_{2}$.
Then $\left|\left\langle v_1,gv_{2}\right\rangle \right|\le\sqrt{d_{1}d_{2}}\n{v_{1}}\n{v_{2}}\Xi^{\nicefrac{1}{k}}(g)$.
\end{thm}

This theorem is not sufficient for this work, since we need more precise bounds when $p(\pi)\notin2\N$. The following theorem contains a general upper bounds that is good enough for all the applications of this paper.
\begin{thm}
\label{thm:CHH Lp genralization}Let $(\pi,V)$ be a unitary
irreducible representation of $G$ with $p(\pi)\le p$,
$p\ge2$, and let $v_1,v_{2}\in V$ be two $K$-finite vectors,
such that $\dim\spann Kv_{1}=d_{1}$, $\dim\spann Kv_{2}=d_{2}$.
Then
\[
\left|\left\langle v_1,gv_{2}\right\rangle \right|\le\sqrt{d_{1}d_{2}}\n{v_{1}}\n{v_{2}}\Xi_{p}(g).
\]
\end{thm}

Let us discuss older similar results, slightly weaker than Theorem~\ref{thm:CHH Lp genralization}.

We first consider bounds on matrix coefficients of a single representation $\pi$. In the Archimedean case, we may consider the theory of leading exponents (\citep[Chapter VIII]{knapp2016representation}),
which we describe in Subsection~\ref{subsec:Leading-Exponents}. In the non-Archimedean case, analogous results hold (\citep[Section 4]{casselman1974introduction}). 
\begin{thm}[See Theorems~\ref{thm:lead_coeff_upper_bound},\ref{thm:lead_coeff_Lp}]
\label{thm:CHH Lp leading coeff}Let $(\pi,V)$ be a unitary irreducible representation of $G$ with $p(\pi)\le p$,
$p\ge2$, and let $v_1,v_{2}\in V$ be two $K$-finite vectors.
Then for every $\epsilon>0$,
\[
\left|\left\langle v_1,gv_{2}\right\rangle \right|\ll_{v_1,v_{2},\pi,\epsilon}q^{-l(g)(\nicefrac{1}{p}-\epsilon)}.
\]
\end{thm}

Next, we consider bounds on matrix coefficients when $v_1,v_{2}$ are $K$-fixed. In such case (if $v_1,v_{2}\ne0$) the representation
is spherical, and well understood, at least in the Archimedean case (see Subsection~\ref{subsec:Lower-Bounds-rank1}
below). From those bounds, we have:
\begin{thm}[{See \citep[Section 3]{ghosh2013diophantine}}]
\label{thm:CHH-Spherical Lp}Let $G$ be Archimedean, and let $(\pi,V)$ be a unitary irreducible representation of $G$ with $p(\pi)\le p$,
and let $v_1,v_{2}\in V$ be two $K$-fixed vectors. Then for every
$\epsilon>0$, 
\[
\left|\left\langle v_1,\pi(g)v_{2}\right\rangle \right|\le\n{v_{1}}\n{v_{2}}\Xi(g)\delta^{\nicefrac{1}{2}-\nicefrac{1}{p}}\ll_{\epsilon}\n{v_{1}}\n{v_{2}}q^{-l(g)(\nicefrac{1}{p}-\epsilon)}.
\]
\end{thm}

Let us provide a proof of Theorem~\ref{thm:CHH Lp genralization} when $v_1,v_{2}$ are $K$-fixed. A bit more work can give bounds for $K$-finite vectors as well, using the ideas of \citep{cowling1988almost}.
Let us first prove the following lemma, which is based on the proof of \citep[Theorem 2]{cowling1988almost}. For
$f\in L^{p}(G)$ and $g\in G$ we let $gf\in L^{p}(G)$
be $gf(g')=f(g^{-1}g')$.
\begin{lem}
\label{lem:Bound on L^p}If $f_{1}\in L^{\nicefrac{p}{(p-1)}}(K\backslash G)$,
$f_{2}\in L^{p}(K\backslash G)$ are two $K$-fixed functions, then $\left|\left\langle f_1,gf_{2}\right\rangle \right|\le\n{f_{1}}_{\nicefrac{p}{(p-1)}}\n{f_{2}}_{p}\Xi_{p}(g)$.
\end{lem}

\begin{proof}
To avoid integrability questions, we assume that $f_1,f_2\in C_c(G)$
and deduce the theorem by density. Denote $p'=\frac{p}{p-1}$.
For $f\in C_c(G)$ it holds that
\[
\intop_{G}f(x)dx=\intop_{K}\intop_{P}f(k\tilde{p})\delta(\tilde{p})d\tilde{p}dk,
\]
so 
\begin{align*}
\left|\left\langle f_1,gf_{2}\right\rangle \right| & \le\intop_{K}\intop_{P}\left|f_{1}(k\tilde{p})\right|\left|f_{2}(g^{-1}k\tilde{p})\right|\delta(\tilde{p})d\tilde{p}dk=\\
 & =\intop_{K}\intop_{P}\left|f_{1}(k\tilde{p})\right|\delta(\tilde{p})^{\nicefrac{1}{p'}}\left|f_{2}(g^{-1}k\tilde{p})\right|\delta(\tilde{p})^{\nicefrac{1}{p}}d\tilde{p}dk\\
 & \le\intop_{K}\left(\intop_{P}\left|f_{1}(k\tilde{p})\right|^{p'}\delta(\tilde{p})d\tilde{p}\right)^{\nicefrac{1}{p'}}\left(\intop_{P}\left|f_{2}(g^{-1}k\tilde{p})\right|^{p}\delta(\tilde{p})d\tilde{p}\right)^{\nicefrac{1}{p}}dk.
\end{align*}
Since $f_{1}$ is $K$-fixed, we have for every
$k\in K$,
\[
\left(\intop_{P}\left|f_{1}(k\tilde{p})\right|^{p'}\delta(\tilde{p})d\tilde{p}\right)^{\nicefrac{1}{p'}}=\n{f_{1}}_{p'}.
\]
write $g^{-1}k=k_{0}\tilde{p}_{0}$. Then 
\begin{align*}
\left(\intop_{P}\left|f_{2}(g^{-1}k\tilde{p})\right|^{p}\delta(\tilde{p})d\tilde{p}\right)^{\nicefrac{1}{p}} & =\left(\intop_{P}\left|f_{2}(k_{0}\tilde{p}_{0}\tilde{p})\right|^{p}\delta(\tilde{p})d\tilde{p}\right)^{\nicefrac{1}{p}}\\
 & =\left(\intop_{P}\left|f_{2}(\tilde{p}_{0}\tilde{p})\right|^{p}\delta(\tilde{p}_{0}\tilde{p})d\tilde{p}\right)^{\nicefrac{1}{p}}\delta(\tilde{p}_{0})^{-\nicefrac{1}{p}}\\
 & =\left(\intop_{P}\left|f_{2}(\tilde{p})\right|^{p}\delta(\tilde{p})d\tilde{p}\right)^{\nicefrac{1}{p}}\delta(\tilde{p}_{0})^{-\nicefrac{1}{p}}\\
 & =\n{f_{2}}_{p}\delta(\tilde{p}_{0})^{-\nicefrac{1}{p}}.
\end{align*}
Since $\delta$ is left $K$-fixed and $\tilde{p}_{0}=k_{0}^{-1}g^{-1}k$
\begin{align*}
\intop_{K_{0}}\left(\intop_{P}\left|f_{2}(g^{-1}k\tilde{p})\right|^{p}\delta(\tilde{p})d\tilde{p}\right)^{1/p}dk & =\n{f_{2}}_{p}\intop\delta(g^{-1}k)^{-\nicefrac{1}{p}}dk=\\
 & =\n{f_{2}}_{p}\Xi_{p}(g).
\end{align*}
\end{proof}
The lemma has a nice corollary: for $a\in A_{+}$ let $A_{a}$ be
the operator $A_{a}\colon C(K\backslash G)\to C(K\backslash G)$
defined by $A_{a}f(g)=\intop_{K}f(a^{-1}k^{-1}g)dk=\intop_{K}\intop_{K}f(k^{\prime-1}a^{-1}k^{-1}g)dkdk'$.
We may define $A_{g}$ for $g\in G$, but it only depends on the $A_{+}$ component of $g$ from the Cartan decomposition.
Since $A_{a}$ is a translation followed by an average, its $L^{p}$-norm
is bounded by $1$ on $L^{p}(G)\cap C(K\backslash G)$,
and therefore it defines an operator $A_{a}\colon L^{p}(K\backslash G)\to L^{p}(K\backslash G)$.
\begin{cor}
\label{cor:L^p norm}The norm of $A_{a}\colon L^{p}(K\backslash G)\to L^{p}(K\backslash G)$
is bounded by $\Xi_{p}(a)$.
\end{cor}

\begin{proof}
Let $p'=\nicefrac{p}{(p-1)}$. Let $f\in L^{p}(K\backslash G)$.
Let $f_{1}\in L^{p'}(K\backslash G)$ with $\n{f_{1}}_{p'}=1$
be such that $\left\langle f_1,A_{a}f\right\rangle =\n{A_{a}f}_{p}$.
Then since $f_{1}$ is left $K$-invariant $\left\langle f_1,A_{a}f\right\rangle =\left\langle f_1,af\right\rangle $.
Applying Lemma~\ref{lem:Bound on L^p}, we have 
\[
\n{A_{a}f}_{p}\le\Xi_{p}(a)\n f_{p},
\]
as needed.
\end{proof}
Note that if $(\pi,V)$ is a unitary representation of
$G$, we may define $\pi(A_{a})\colon V^{K}\to V^{K}$ by the
same arguments as above, as 
\begin{equation}
\pi(A_{a})v=\intop_{K}\pi(ka)v\,dk=\intop_{K}\intop_{K}\pi(kak')v\,dkdk',\label{eq:Aa definition}
\end{equation}
and $\n{A_{a}}_{V}\le1$. By a standard argument, this operation commutes with taking matrix coefficients -- denote for $v_1,v_{2}\in V^{K}$,
$\varphi_{v_1,v_{2}}(g)=\left\langle v_1,\pi(g)v_{2}\right\rangle $,
$\varphi_{v_1,v_{2}}(g)\in L^{\infty}(G)$,
and then $A_{a}\varphi_{v_1,v_{2}}(g)=\varphi_{A_{a}v_1,v_{2}}(g)$.

\begin{cor}
\label{cor:pi(Aa) bound}Let $(\pi,V)$ be a unitary irreducible
representation of $G$ with $p(\pi)\le p$, and let $v_1,v_{2}\in V$
be two $K$-fixed vectors. Then 
\[
\left|\left\langle \pi(g)v_1,v_{2}\right\rangle \right|\le\n{v_{1}}\n{v_{2}}\Xi_{p}(g).
\]
\end{cor}

\begin{proof}
We may assume that $v_1,v_{2}$ are of norm $1$. Since both of
them are left $K$-invariant $\left\langle \pi(g)v_1,v_{2}\right\rangle =\left\langle \pi(A_{a_{g}})v_1,v_{2}\right\rangle $,
where $a_{g}\in A_{+}$ is the $A_{+}$ component of the Cartan decomposition
of $g$. If $\pi$ is irreducible, it is well known that the subspace
$V^{K}$ of $K$-fixed vectors is one dimensional. Therefore, $v_{1}$ is an eigenvector of $\pi(A_{a_{g}})$ on $V^{K}$. Therefore,
$c_{v_1,v_{2}}(g)$ is an eigenvector of $A_{a_{g}}$
on $L^{p+\epsilon}(g)$ for every $\epsilon>0$, with eigenvalue
$\left\langle \pi(g)v_1,v_{2}\right\rangle $. As the
norm of $A_{a_{g}}$ on $L^{p+\epsilon}(K\backslash G)$
is bounded by $\Xi_{p+\epsilon}(g)$, each of its eigenvalues is bounded
by $\Xi_{p+\epsilon}(g)$ as well. Therefore, $\left|\left\langle \pi(g)v_1,v_{2}\right\rangle \right|\le\Xi_{p+\epsilon}(g)$
for every $\epsilon>0$. By taking $\epsilon\to0$ we deduce $\left|\left\langle v_1,\pi(A_{g})v_{2}\right\rangle \right|\le\Xi_{p}(g)$,
as required.
\end{proof}

\subsection{Upper Bounds on Operators}

Let $(\pi,V)$ be a unitary representation of $G$. For
$a\in A_{+}$ let $\pi(A_{a})\colon V^{K}\to V^{K}$ be
from Equation~\eqref{eq:Aa definition}. Similarly, for $d_{0}\in\R_{\ge0}$,
let $\chi_{d_0},\psi_{d_0}\in C_{c}^{\infty}(G)$ be
as in Definition~\ref{def:chi and psi}. Recall that
$\chi_{d_0}$ is a smooth approximation for the characteristic
function of $\left\{ g:l(g)\le d_{0}\right\} $ and $\psi_{d_0}$
is a smooth approximation for $q^{(d_{0}-l(g))/2}\chi_{d_0}$.

For $h\in C_c(G)$ we define as usual $\pi(h)v=\intop_{G}h(g)\pi(g)v\,dg$.
\begin{cor}
\label{cor:Hecke Bounds}Let $(\pi,V)$ be a unitary irreducible
representation of $G$ with $p(\pi)\le p$, and let $v\in V$
be a $K$-fixed vector. Then 
\begin{align*}
\n{\pi(A_{a})v} & \ll_{\epsilon}q^{-l(a)(\nicefrac{1}{p}-\epsilon)}\n v\\
\n{\pi(\chi_{d_0})v} & \ll_{\epsilon}q^{d_{0}(1-\nicefrac{1}{p}+\epsilon)}\n v\\
\n{\pi(\psi_{d_0})v} & \ll_{\epsilon}q^{d_{0}(1-\nicefrac{1}{p}+\epsilon)}\n v.
\end{align*}
\end{cor}

\begin{proof}
The bound on $\pi(A_{a})$ follows from Corollary~\ref{cor:pi(Aa) bound},
using the explicit bounds of Theorem~\ref{thm: Explicit Bounds}.

We will only prove the estimate for $\psi_{d_0}$, since the proof
for $\chi_{d_0}$ is similar and a little easier.
\begin{align*}
\n{\pi(\psi_{d_0})v} & \le\intop_{l(g)\le d_{0}+1}q^{\nicefrac{(d_{0}-l(g))}{2}}\n{\pi(g)v}dg\\
 & \ll_{\epsilon}\intop_{l(g)\le d_{0}+1}q^{\nicefrac{d_0}{2}-l(g)(\nicefrac{1}{2}+\nicefrac{1}{p}-\epsilon)}\n vdg\\
 & =\intop_{K}\intop_{K}\intop_{l(a)\le d_{0}+1}S(a)q^{\nicefrac{d_0}{2}-l(k_{1}ak_{2})(\nicefrac{1}{2}+\nicefrac{1}{p}-\epsilon)}\n vdk_{1}\,da\,dk_{1}\\
 & \ll\intop_{l(a)\le d_{0}+1}q^{l(a)}q^{\nicefrac{d_0}{2}-l(a)(\nicefrac{1}{2}+\nicefrac{1}{p}-\epsilon)}\n vda\\
 & \ll_{\epsilon}\intop_{l(a)\le d_{0}+1}q^{\nicefrac{d_0}{2}+l(a)(\nicefrac{1}{2}-\nicefrac{1}{p}+\epsilon)}\n vda\\
 & \le\intop_{l(a)\le d_{0}+1}q^{\nicefrac{d_0}{2}+d_{0}(\nicefrac{1}{2}-\nicefrac{1}{p}+\epsilon)}\n vda\\
 & \ll_{\epsilon}q^{d_{0}(1-\nicefrac{1}{p}+\epsilon)}\n v.
\end{align*}
In the last line we used the fact that $\intop_{l(a)\le d_{0}+1}da\ll_{\epsilon}q^{d_{0}\epsilon}$.
\end{proof}

\subsection{\label{subsec:Lower-Bounds-rank1}Spherical Functions}

Let us recall the definition of spherical functions. For $\lambda\in \a_{\C}^{\ast}=\a^{\ast}\otimes\C$
dominant, we let $\lambda_{P}\colon P\to\C^{\times}$ be $\lambda_{P}(mn)=q^{\lambda(\nu(m))}\delta^{-1/2}(m)=q^{(\lambda-\rho)(\nu(m))}$.
Extend it to $\tilde{\lambda}\colon G\to\C$ by $\tilde{\lambda}(kp)=\lambda_{P}(p)$.
Finally, define the spherical function 
\[
\varphi_{\lambda}(g)=\intop_{K}\tilde{\lambda}(gk)dk.
\]

Note that $\varphi_{0}=\Xi$ and $\varphi_{(1-2/p)\rho}=\Xi_{p}$
for $2\le p\le\infty$.

The theory of spherical functions was developed by Harish-Chandra
in the Archimedean case and by Satake in the non-Archimedean case
(see \citep[Subsection 3.2]{ghosh2013diophantine} and the reference
therein). We will need some very basic properties of spherical functions.

Let $(\pi,V)\in\Pi(G)_{\sph}$ be
a unitary spherical representation, i.e. a unitary irreducible representation
with a non-trivial $K$-fixed vector. If $v_1,v_{2}$ are $K$-fixed,
then there exists a dominant $\lambda\in \a_{\C}^{\ast}$,
such that
\[
\left\langle v_1,\pi(g)v_{2}\right\rangle =\left\langle v_1,v_{2}\right\rangle \varphi_{\lambda}(g).
\]

Upper bounds on spherical functions are given as follows. By \citep[Lemma 3.3]{ghosh2013diophantine},
it holds that for $a\in A_{+}$, $k,k'\in K$ and $\epsilon>0$,
\[
\varphi_{\lambda}(kak')=\varphi_{\lambda}(a)\ll q^{\re\lambda(\nu(a))}\Xi(a)\ll_{\epsilon}q^{\re\lambda(\nu(a))}q^{-l(a)(1/2-\epsilon)}.
\]

Let $\omega_1,...,\omega_{r}\in\a$ be the fundamental coweights. Since $\lambda$ is dominant $\re\lambda(\omega_{i})\ge0$
for $1\le i\le r$. Using the above bounds, we deduce that
\begin{equation}
\re\lambda(\omega_{i})\le(1-\nicefrac{2}{p})\rho(\omega_{i}) \ \ \forall 1\le i\le r \implies p(\pi)\le p.
\label{eq:criteria for boundness} 
\end{equation}
See also \citep[Lemma 3.2]{ghosh2013diophantine}, which has a small misprint.
 
In the Archimedean case the other direction of the implication of Equation~\eqref{eq:criteria for boundness} is also true (\cite[Theorem 8.48]{knapp2016representation}), which gives a proof of Theorem~\ref{thm:CHH-Spherical Lp} in this case.  

However, the other direction is no longer true for the general non-Archimedean case. The simplest example is when the Bruhat-Tits building of $G$ is a bipartite $(d_0,d_1)$-regular tree, $d_0\ne d_1$, where for one of the choices of a maximal compact subgroup, there is a spherical function that is square integrable but $\re\lambda(\omega)>0$ (this result was first observed by \cite{matsumoto69}). In general, the other direction is true in the standard case according to MacDonald (\cite[Chapter V]{macdonald1971spherical}), which excludes the case above. In particular, every group has a standard maximal compact subgroup for which Equation~\eqref{eq:criteria for boundness} is an equivalence.

Lower bounds on Spherical functions can be given in general if $\re\lambda=\lambda$.
By Lemma~\ref{lem:HFunctionTechnicalLemma}, if $\re\lambda=\lambda$
and $\re\lambda(\omega_{i})\le\rho(\omega_{i})$
then for $a\in A_{+}$, $k,k'\in K$
\begin{equation}
\varphi_{\lambda}(kak')=\varphi_{\lambda}(a)=\intop_{K}\tilde{\lambda}(ak)dk\ge\intop_{K}\tilde{\lambda}(a)dk=\lambda_{P}(a).\label{eq:lower bound spherical}
\end{equation}

In the case when $G$ is Archimedean, it is well known that the dominant
$\lambda\in \a_{\C}^{\ast}$ which occur this way for unitary spherical representations must satisfy $-\bar{\lambda}=w\lambda$
for some $w\in W$ (\citep{kostant1969existence}).

If we moreover assume that $G$ is of rank $1$, then $\a_{\C}^{*}\cong\C$
and $W=\left\{ 1,s\right\} $ acts by $s\lambda=-\lambda$. Therefore, the only dominant $\lambda\in \a_{\C}^{*}$ which may occur satisfy
either $\re\lambda=0$ or $\re\lambda=\lambda$. In the
case $\re\lambda=0$ the corresponding unitary representation
satisfies $p(\pi)=2$. If $\re\lambda=\lambda$ then
$\lambda=\alpha\rho$ with $\alpha\ge0$. By Equation~\eqref{eq:criteria for boundness},
$\alpha\le1$. Write $\alpha=1-1/p$,
$2\le p\le\infty$. Using Equation~\eqref{eq:criteria for boundness} and the
definition of $\Sigma_{p}(g)$ we conclude that for $p>2$,
there is at most a single unitary irreducible representation $(\pi,V)$
with a non-trivial $K$-fixed vector and $p(\pi)=p$. The
corresponding spherical function is $\Xi_{p}(g)$. We
conclude:
\begin{prop}
\label{prop:Rank1 lower bound- Archimedean}Let $G$ be Archimedean
of rank $1$, $p>2$ and $(\pi,V)$ a unitary irreducible
representation of $G$ with $p(\pi)=p$, having a non-trivial
$K$-fixed vector $v$. Then $\pi(A_{a})v=\lambda_{1}v$,
$\pi(\chi_{d_0})v=\lambda_{2}v$, with
\begin{align*}
e^{-l(a)/p} & \le\lambda_{1}\\
e^{d_{0}(1-\nicefrac{1}{p})} & \ll\lambda_{2}\,\,\,\text{for }d_{0}\ge1.
\end{align*}

Moreover, if $h\in C_{c}(K\backslash G/K)$ satisfies $h(g)\ge\chi_{d_0}(g)$
for every $g\in G$, then $\pi(h)v=\lambda_{h}v$, with
$\lambda_{h}\ge\lambda_{2}$.
\end{prop}

\begin{proof}
It is well known that in this case the set of all $K$-fixed vectors
in $V$ is one dimensional and equals $\spann\left\{ v\right\} $.
Since $\pi(A_{a})v$, $\pi(\chi_{d_0})v$,$\pi(h)v$
are $K$-fixed we get that $\pi(A_{a})v=\lambda_{1}v$, $\pi(\chi_{d_0})v=\lambda_{2}v$
and $\pi(h)v=\lambda_{h}v$. Applying matrix coefficients,
we see that 
\begin{align*}
\lambda_{1} & =\Xi_{p}(a)\\
\lambda_{2} & =\intop_{G}\chi_{d_0}(g)\Xi_{p}(g)dg\\
\lambda_{h} & =\intop_{G}h(g)\Xi_{p}(g)dg\ge\intop_{G}\chi_{d_0}(g)\Xi_{p}(g)dg=\lambda_{2}.
\end{align*}
Using Equation~\eqref{eq:lower bound spherical} we get 
\[
\Xi_{p}(g)\ge e^{-\frac{2}{p}\rho(\nu(a))}=e^{-l(a)/p}.
\]
And therefore $\lambda_{1}\ge e^{-l(a)/p}$ and for $d_{0}\ge1$,
\begin{align*}
\lambda_{2} & =\intop_{G}\chi_{d_0}(g)\Xi_{p}(g)dg\ge\\
 & \ge\intop_{0}^{d_0}\sinh(t)e^{-t/p}dt\gg e^{(1-1/p)d_{0}}.
\end{align*}
\end{proof}

A similar analysis can be done for the non-Archimedean rank $1$ case, whenever the Bruhat-Tits tree of $G$ is regular. However, when the Bruhat-Tits Tree of $G$ is not regular, the description of the spherical unitary dual is more complicated, and in particular it is different for the two non-conjugate maximal compact subgroups. 
The analysis of this can be done by analyzing the slightly more general case of regular and bi-regular trees. See \citep{hashimoto1989zeta}
or \citep{kamber2019p} for an analysis of this case.

Since the calculations are a bit long and are not the main focus of this work, we skip them and state the final result. 
Recall that we have a map $\nu\colon M/M^{0}\to\a$ with a discrete image. We identify its image with $\Z$.

\begin{prop}
\label{prop:Rank1 lower bound- padic}Let $G$ be non-Archimedean
of rank $1$, $p>2$ and $(\pi,V)$ a unitary irreducible representation of $G$ with $p(\pi)=p$, having a non-trivial
$K$-fixed vector $v$. Then for $a\in A_{+}$ with $\nu(a)\in\Z$
even, we have $\pi(A_{a})v=\lambda_{1}v$, with
\begin{align*}
q^{-l(a)/p} & \ll\lambda_{1}.
\end{align*}
\end{prop}

We can finally conclude:
\begin{thm}
\label{thm:rank 1 has good family}Let $G$ be of rank $1$. Then
the set $\Pi(G)_{\sph,\nt}$ of spherical non-tempered
unitary representations has a good family of functions.
\end{thm}

\begin{proof}
For $G$ Archimedean, we choose $f_{d_0}=\chi_{d_{0}/2}*\chi_{d_{0}/2}$.
For $G$ non-Archimedean, we choose $a\in A_{+}$ with $\nu(a)\in\Z$
even, and with $l(a)=d_{0}/2\pm O(1)$, and
choose $f_{d_0}=A_{a}*A_{a}^{*}$.

The first property of a good family follows from Proposition~\ref{prop:Rank1 lower bound- Archimedean},
Proposition~\ref{prop:Rank1 lower bound- padic} and the Convolution
Lemma~\ref{lem:convolution lemma} below.

The second and third follows from simple properties of convolutions.
\end{proof}

\subsection{\label{subsec:Convolution-of-Operators}Convolution of Operators}

In this section, we analyze the function $\chi_{d_0}\ast\chi_{d_0}(g)$.
Similar analysis can be found for rank $1$ in \citep[Lemma 3.1]{sarnak1991bounds}.
Our analysis is less accurate, but is more abstract and works for
every rank.
\begin{lem}[Convolution Lemma]
\label{lem:convolution lemma}It holds that $c_{d_0}=\chi_{d_0}\ast\chi_{d_0}\in C_{c}^\infty(K\backslash G/K)$,
and satisfies the inequality 
\[
c_{d_0}(g)\ll_{\epsilon}q^{d_{0}\epsilon}\psi_{2d_{0}+2}(g).
\]
The same bound holds for the convolution of every $f_1,f_{2}\in C_{c}(K\backslash G/K)$
such that $f_{1}(g),f_{2}(g)\ll_{\epsilon}q^{d_{0}\epsilon}\chi_{d_0}(g)$.
\end{lem}

\begin{proof}
The second statement is a consequence of the first, so we only need
to prove the first statement.

The idea is to look at the action (by right convolution) of $\chi_{d_0}$
on $L^{2}(G)$. By Lemma~\ref{lem:Bound on L^p}, and
the same arguments as in Corollary~\ref{cor:L^p norm}, the norm
of $\chi_{d_0}$ on $L^{2}(G)$ is bounded by $\ll_{\epsilon}q^{d_{0}(\nicefrac{1}{2}+\epsilon)}$.
Therefore, the norm of $c_{d_0}$ on $L^{2}(G)$ is bounded
by $\ll_{\epsilon}q^{d_{0}(1+\epsilon)}$. Now we need
to use some continuity arguments to deduce pointwise bounds.

Notice that since $\chi_{d_0}\in C_{c}^\infty(K\backslash G/K)$,
the same is true for $c_{d_0}$. In the non-Archimedean case the
arguments are simpler: we look at the action of $c_{d_0}$ on the
characteristic function $\boldsymbol{1}_{K}$ of $K$. Then 
\[
\n{\boldsymbol{1}_{K}\ast c_{d_0}}_{L^{2}(G)}^{2}\ll_{\epsilon}q^{2d_{0}(1+\epsilon)}\n{\boldsymbol{1}_{K}}_{L^{2}(G)}=q^{2d_{0}(1+\epsilon)}.
\]
But if $c_{d_0}(g)=R$ then $\boldsymbol{1}_{K}\ast c_{d_0}(g)=R$,
so 
\[
\n{\boldsymbol{1}_{K}\ast c_{d_0}}_{L^{2}(G)}^{2}\ge\mu(KgK)R^{2}\gg q^{l(g)}R^{2}.
\]
Therefore, $c_{d_0}(g)=R\ll_{\epsilon}q^{d_{0}(1+\epsilon)-\nicefrac{l(g)}{2}}$
as needed.

In the Archimedean case, assume that $c_{d_0}(g)=R$.
Then $c_{d_{0}+2}(g')\ge R$ for every $g'\in G$ with
$\left|l(g)-l(g^{\prime})\right|\le1$. We consider
$\boldsymbol{1}_{B_{1}}$ where $B_{1}$ is the ball of radius $1$
around the identity. It holds that 
\[
\n{\boldsymbol{1}_{B_{1}}\ast c_{d_{0}+2}}_{L^{2}(G)}^{2}\ll_{\epsilon}q^{2d_{0}(1+\epsilon)}\n{\boldsymbol{1}_{B_{1}}}_{L^{2}(G)}\ll q^{2d_{0}(1+\epsilon)}.
\]
It also holds that $\boldsymbol{1}_{B_{1}}\ast c_{d_{0}+2}(g)\gg R$,
for $g'$ satisfying $\left|l(g)-l(g^{\prime})\right|\le1$,
so 
\[
\n{\boldsymbol{1}_{K}\ast c_{d_{0}+2}}_{L^{2}(G)}^{2}\gg\mu(KB_{1}gK)R^{2}\gg q^{l(g)}R^{2},
\]
and $c_{d_0}(g)=R\ll_{\epsilon}q^{d_{0}(1+\epsilon)-\nicefrac{l(g)}{2}}$
as needed.
\end{proof}

\subsection{\label{subsec:Traces-of-Operators}Traces of Operators on Irreducible
Unitary Representations}

Our goal in this section is to relate the lower and upper bounds
of the previous sections to lower and upper bounds on traces.

Let us recall how to define traces of an operator $h\in C_c(G)$
on a unitary irreducible representation $(\pi,V)$. We
will only consider the case when $h$ is left and right $K$-finite,
which simplifies the theory. In such case, there is an orthogonal projection $e_{h}\colon C(K)\to C(K)$ such that $h=e_{h}*h*e_{h}$
(see e.g. \citep[proof of Theorem 2]{cowling1988almost}). The orthogonal
projection has a finite dimensional image, so it is supported on a
finite number of $K$-types.

Since $\pi$ is admissible, the image of $\pi(h)$ is finite
dimensional, and therefore have a trace (\citep[Chapter X]{knapp2016representation}),
defined by
\[
\tr \pi(h)=\sum_{i}\left\langle u_{i},\pi(h)u_{i}\right\rangle ,
\]
where $\left\{ u_{i}\right\} \subset V$ is an orthonormal basis.

By uniform admissibility (\citep[Theorem 10.2]{knapp2016representation},
\citep{bernshtein1974all}) and the fact the image of $\pi(h)$
is supported on a finite number of $K$-types, the image of $\pi(h)$
is of bounded dimension, depending only on the projection $e_{h}$.

Finally, recall that the norm of a finite dimensional operator is
larger than the largest absolute value of an eigenvalue. We conclude:
\begin{prop}
Assume that $h\in C_c(G)$ is left and right $K$-finite
and $(\pi,V)\in\Pi(G)$. Then 
\[
\left|\tr \pi(h)\right|\ll_{e_{h}}\n{\pi(h)},
\]
the bound depending only on the projection $e_{h}$ such that $e_{h}*h*e_{h}$.
\end{prop}

As an example, if $h\in C_{c}(K\backslash G/K)$ is left
and right $K$-invariant, the image of $\pi(h)$ is of
dimension $1$ or $0$. If the dimension is 0, obviously $\tr \pi(h)=\pi(h)=0$.
If the dimension is $1$, $V$ has a $K$-invariant vector $v\in V$,
$\n v=1$, and 
\[
\left|\tr \pi(h)\right|=\n{\pi(h)}=\left|\left\langle v,\pi(h)v\right\rangle \right|.
\]

One may also deduce lower bounds on traces of non-negative self-adjoint
operators. It follows from the same considerations as in the finite
dimensional case.
\begin{prop}
Assume that $h\in C_c(G)$ is left and right $K$-finite
and $(\pi,V)\in\Pi(G)$. Moreover, assume that
$\pi(h)$ is self-adjoint and non-negative. Then 
\[
\n{\pi(h)}=\sup_{v:\n v=1}\left\langle v,\pi(h)v\right\rangle \le\tr \pi(h).
\]
\end{prop}

\subsection{\label{subsec:The trace formula}The Pre-Trace formula}

Let us recall the pre-trace formula (\citep[Chapter 1]{gelfand1968representation}). Let $h\in C_c(G)$
and let $\Gamma\subset G$ be a cocompact lattice. Denote by $\hat{h}\in C_c(G)$ the function $\hat{h}(g)=h(g^{-1})$. Then we
have an operator $h\colon L^{2}(\Gamma\backslash G)\to L^{2}(\Gamma\backslash G)$,
acting on $f\in L^{2}(\Gamma\backslash G)$ by
\begin{align*}
(hf)(x)=f\ast\hat{h}(x) & =\intop_{G}f(xg)h(g)dg=\intop_{G}f(y)h(x^{-1}y)dy\\
 & =\intop_{\Gamma\backslash G}f(y)(\sum_{\gamma\in\Gamma}h(x^{-1}\gamma y))dy=\intop_{\Gamma\backslash G}K(x,y)f(y)dy,
\end{align*}
for $K(x,y)=\sum_{\gamma\in\Gamma}h(x^{-1}\gamma y)$.
By \citep[Chapter 1]{gelfand1968representation}, if $h$ is also
self-adjoint then it has a trace on $L^{2}(\Gamma\backslash G)$
and

\[
\tr h|_{L^{2}(\Gamma\backslash G)}=\intop_{\Gamma\backslash G}\sum_{\gamma\in\Gamma}h(x^{-1}\gamma x)dx.
\]

Moreover, if $h\in C_{c}^{\infty}(G)$ and $L^{2}(\Gamma\backslash G)\cong\oplus_{\pi\in\Pi(G)}m(\pi,\Gamma)$
is the decomposition into irreducible representations, then we have
the pre-trace formula: 
\begin{align*}
\tr h|_{L^{2}(\Gamma\backslash G)} & =\intop_{\Gamma\backslash G}\sum_{\gamma\in\Gamma}h(x^{-1}\gamma x)dx\\
 & =\sum_{\pi\in\Pi(G)}m(\pi,\Gamma)\tr \pi(h),
\end{align*}
where $\tr (\pi(h))$ is the usual trace
on the representation space $(\pi,V)$.

The following lemma is immediate from the pre-trace formula but essential
for our work.
\begin{lem}
\label{lem:positivity of trace}If $h_1,h_{2}\in C_c(G)$
satisfy that $h_{1}(g)\ge h_{2}(g)$ for every
$g\in G$ then $\tr h_{1}|_{L^{2}(\Gamma\backslash G)}\ge\tr h_{2}|_{L^{2}(\Gamma\backslash G)}$.
In particular, if $h_{1}(g)\ge0$ then $\tr h_{1}|_{L^{2}(\Gamma\backslash G)}\ge0$.
\end{lem}

\subsection{\label{subsec:Spectral-Decomposition-of-small-ball}Spectral Decomposition
of a Characteristic Function of a Small Ball}

For $x\in\Gamma_{N}\backslash\Gamma_1\subset X_{N}=\Gamma_{N}\backslash G/K$,
let $b_{x,\delta}\in L^{2}(X_N)$ be defined as follows:
\begin{itemize}
\item In the non-Archimedean case choose $h_{\delta}\in C_{c}(K\backslash G/K)$
to be the characteristic function of $K$.
\item In the Archimedean case, choose $0<\delta< 1/4$ such that $l(\gamma)>\delta$
for every $\gamma\in\Gamma_1$ with $l(\gamma)>0$.
Choose a function $h_{\delta}\in C_{c}^{\infty}(K\backslash G/K)$
such that:
\begin{itemize}
\item $0\le h_{\delta}(g)\le\frac{2}{\mu(B_{\delta}(e))}$
for all $g\in G$, where $B_{\delta}(e)$ is the ball of radius $\delta$ around the identity $e\in G$.
\item $h_{\delta}(g)$ is supported on $\{g\in G:l(g)\le\delta\}$.
\item $\intop_{G}h_{\delta}(g)dg=1$
\item $h(g)=h(g^{-1})$, i.e., $h=\hat{h}$.
\end{itemize}
\end{itemize}
Finally, let $b_{x,\delta}\in L^{2}(X_N)$ be 
\[
b_{x,\delta}(y)=\sum_{\gamma\in\Gamma_{N}}h_{\delta}(x^{-1}\gamma y).
\]

We fix $\delta>0$ once and for all (depending on $\Gamma_1$) and suppress the dependence on it from now on. 
We notice that 
\[
\intop_{X_N}b_{x,\delta}(y)dy=1.
\]

By the properties of $\delta$, the sum defining it is over at most $\left|\Gamma_1\cap K\right|$ elements, and therefore (recall that we suppress the dependence on $\Gamma_1$ from our notations), 
\[
\n{b_{x,\delta}}_{\infty}\ll 1, \ \  \n{b_{x,\delta}}_{2}\ll 1.
\]

Let us remark that for our uses in the Archimedean rank $1$ case
one can simply choose instead
\[
h_{\delta}(g)=\begin{cases}
\frac{1}{\mu(B_{\delta}(e))} & l(g)\le\delta\\
0 & \text{else}
\end{cases},
\]
and for higher rank we make this choice so that one can apply the
Paley-Wiener theorem for spherical functions due to Harish-Chandra,
which is used as follows:
\begin{lem}
\label{lem:Covering Lemma 2}Let $f\in L^{2}(X_N)$.
Then, 
\[
\sum_{x\in\Gamma_{N}\backslash\Gamma_{1}}\left|\left\langle b_{x,\delta},f\right\rangle \right|^{2}\ll\n f_{2}^{2}.
\]

Moreover, in the Archimedean case, if $f\in L^{2}(X_N)$
is the $K$-fixed function of some irreducible representation $\pi\subset L^{2}(\Gamma_{N}\backslash G)$
with $\lambda(\pi)=\lambda$, then for every $L'>0$, 
\[
\sum_{x\in\Gamma_{N}\backslash\Gamma_{1}}\left|\left\langle b_{x,\delta},f\right\rangle \right|^{2}\ll_{L'}(1+\lambda)^{-L'}\n f_{2}^{2}.
\]

(The last result will only be used in the Archimedean rank $\ge2$
case).
\end{lem}

\begin{proof}
By our assumption on $\delta$, the balls $B_{\delta}(x)$
for $x\in\Gamma_{N}\backslash\Gamma_1$ are all either equal (with
multiplicity at most $\left|\Gamma_1\cap K\right|$) or distinct,
so by Cauchy-Schwartz
\begin{align*}
\sum_{x\in\Gamma_{N}\backslash\Gamma_{1}}\left|\left\langle f,b_{x,\delta}\right\rangle \right|^{2} & \le\sum_{x\in\Gamma_{N}\backslash\Gamma_{1}}\left\Vert f|_{B_{\delta}(x)}\right\Vert _{2}^{2}\left\Vert b_{x,\delta}\right\Vert _{2}^{2}\le\left|\Gamma_1\cap K\right|\n f_{2}^{2}\max_{x\in\Gamma_{N}\backslash\Gamma_{1}}\n{b_{x}}_{2}^{2}\\
 & \ll_{\delta}\n f_{2}^{2}.
\end{align*}

For the moreover part, note that 
\[
\left\langle f,b_{x,\delta}\right\rangle =f(x)\tr \pi(h_{\delta}).
\]
It is well known (see \citep{seeger1989bounds} for an exact statement)
that there is a constant $M>0$ such that 
\[
\left|f(x)\right|\ll(1+\lambda)^{M}\left\Vert f|_{B_{\delta}(x)}\right\Vert _{2}^{2}.
\]
By the Paley-Wiener theorem for spherical functions (\citep[Subsection 3.4]{duistermaat1979spectra}),
\[
\tr \pi(h_{\delta})\ll_{L'}(1+\lambda)^{-L'-M}.
\]
Combining both estimates we get the required inequality.
\end{proof}

\section{\label{sec:Weak Injective Radius implies density}The Weak Injective
Radius Property}

In this section, we will study the weak injective radius property, and deduce spectral results from it when the lattices are cocompact.

Consider $\chi_{d_0}\in C_{c}^{\infty}(G)$ from Section~\ref{sec:Main-Ideas}. It is self-adjoint since $l(g)=l(g^{-1})$. Since $\chi_{d_0}$
is left and right $K$-invariant, it acts on $L^{2}(\Gamma\backslash G/K)$.

Note that by the definition of $\boldsymbol{N}(\Gamma,d_{0},y)$,
it holds that for every $x\in \Gamma_n\backslash \Gamma_1$, 
\begin{equation}
\boldsymbol{N}(\Gamma,d_{0},x)\le\sum_{\gamma\in\Gamma}\chi_{d_0}(x^{-1}\gamma x)\le\boldsymbol{N}(\Gamma,d_{0}+1,x).\label{eq:relation between hd and N}
\end{equation}
If $\chi_d$ had been the characteristic function of $\{g\in G: l(g)\le d\}$, then the first inequality in Equation~\ref{eq:relation between hd and N} would simply be an equality.

Our initial observation is:
\begin{lem}\label{lem: counting from operator}
For every $x_0\in \Gamma_N\backslash \Gamma$
\[
\left\langle b_{x_0,\delta},\chi_{d_0-2}b_{x_0,\delta}\right\rangle \ll\boldsymbol{N}(\Gamma_{N},d_{0},x_0)\ll\left\langle b_{x,\delta},\chi_{d_{0}+2}b_{x_0,\delta}\right\rangle.
\]
\end{lem}
\begin{proof}
By unfolding, we get that
\begin{align*}
\langle b_{x,\delta},\chi_{d_0}b_{x_0,\delta} \rangle &= 
\intop_{\Gamma_N\backslash G} \intop_{\Gamma_N \backslash G} \sum_{\gamma\in \Gamma_N} b_{x_0,\delta}(x)\chi_d(x^{-1}\gamma y)  b_{x_0,\delta}(y) dxdy \\
& = \intop_{\Gamma_N \backslash G} \intop_{\Gamma_N \backslash G} \sum_{\gamma_1\in \Gamma_N}h_{\delta}(x_0^{-1}\gamma_1 x)\sum_{\gamma\in \Gamma_N}\chi_d(x^{-1}\gamma y)  \sum_{\gamma_2\in \Gamma_N}h_{\delta}(x_0^{-1}\gamma_2 y) dxdy \\
& = \intop_{\Gamma_N \backslash G} \intop_{\Gamma_N \backslash G} \sum_{\gamma_1\in \Gamma_N}h_{\delta}(x_0^{-1}\gamma_1 x)\sum_{\gamma\in \Gamma_N}\chi_d(x^{-1}\gamma y)  \sum_{\gamma_2\in \Gamma_N}h_{\delta}(y^{-1}\gamma_2 x_0) dxdy \\
&= \intop_{G} \intop_{ G} h_{\delta}(x_0^{-1} x)\sum_{\gamma\in \Gamma_N}\chi_d(x^{-1}\gamma y)  h_{\delta}(y^{-1} x_0) dxdy \\
&= \sum_{\gamma\in \Gamma_N} (h_\delta*\chi_d*h_\delta)(x_0^{-1}\gamma x_0).
\end{align*}
The lemma follows from Equation~\eqref{eq:relation between hd and N} and the simple pointwise estimates
\[
\chi_{d-2}(g) \ll h_\delta*\chi_d*h_\delta(g) \ll \chi_{d+2}(g).
\]
\end{proof}

We can now prove that some bounds on $\boldsymbol{N}(\Gamma_N,d_0,x)$ imply bounds on $\boldsymbol{N}(\Gamma_N,d_0+d,x)$.
\begin{lem}\label{lem:Varrying d}
Assume that for some $0\le d_{0}$, $x\in \Gamma_N\backslash \Gamma_1$ it holds for some $M$ that for every $d\le d_0$
\[
\boldsymbol{N}(\Gamma_N,d,x)\ll M q^{d/2}.
\]
Then for every $d_1\ge0,\epsilon>0$ it holds that 
\[
\boldsymbol{N}(\Gamma_N,d_{0}+d_1,x)\ll_{\epsilon}M q^{(d_0+d_1)\epsilon} q^{d_{0}/2+d_1}.
\]
\end{lem}

\begin{proof}
We choose $\tilde{d}=(d_0-4)/2$ and calculate $\n{\chi_{\tilde{d}}b_{x,\delta}}_{2}^{2}$ using Lemma~\ref{lem:convolution lemma} and Lemma~\ref{lem: counting from operator}:

\begin{align*}
\n{\chi_{\tilde{d}}b_{x,\delta}}_{2}^{2} 
& =\left\langle \chi_{\tilde{d}}b_{x,\delta},\chi_{\tilde{d}}b_{x,\delta}
\right\rangle \\
 & =\left\langle \chi_{\tilde{d}}*\chi_{\tilde{d}}b_{x,\delta},b_{x,\delta}\right\rangle \\
 & \ll_{\epsilon}q^{d_0\epsilon} \left\langle \psi_{2\tilde{d}+2}b_{x,\delta},b_{x,\delta}
 \right\rangle \\
 & \ll q^{d_0\epsilon} \intop_{0}^{d_0-4}
 q^{(d_0-d)/2}\left\langle \chi_{d+2}b_{x,\delta},b_{x,\delta}\right\rangle dd\\
 & \ll q^{d_0\epsilon} \intop_{0}^{d_0-4} q^{(d_0-d)/2}\boldsymbol{N}(\Gamma_N,d+4,x)dd\\
 & \ll q^{d_0\epsilon}M\intop_{0}^{d_0-4}q^{(d_0-d)/2}q^{d/2}dd\\
 & \ll M q^{d_0\epsilon}q^{\tilde{d}} \le M q^{d_0\epsilon}q^{d_0/2}.
\end{align*}

Now, for $d_1\ge0$,
we use the inequality 
\[
\chi_{d_0+d_1+2}(g)\ll(\chi_{\tilde{d}}*\chi_{d_1+4}*\chi_{\tilde{d}})(g),
\]
to deduce
\begin{align*}
\boldsymbol{N}(\Gamma,d_{0}+d_1,x) & \ll \left\langle \chi_{d_0+d_1+2}b_{x,\delta},b_{x,\delta}\right\rangle \\
 & \ll \left\langle \chi_{\tilde{d}}*\chi_{d_1+4}*\chi_{\tilde{d}}b_{x,\delta},b_{x,\delta}\right\rangle \\
 & =\left\langle \chi_{d_1+4}*\chi_{\tilde{d}}b_{x,\delta},\chi_{\tilde{d}}b_{x,\delta}\right\rangle \\
 & \le \n{\pi(\chi_{d_{1}+4})}\n{\chi_{\tilde{d}}b_{x,\delta}}_{2}^{2}\\
 & \ll_{\epsilon}q^{d_1(1+\epsilon)}\n{\chi_{\tilde{d}}b_{x,\delta}}_{2}^{2}\\
 & \ll q^{(d_{0}+d_1)\epsilon}M q^{d_0/2+d_1}.
\end{align*}
\end{proof}

Lemma~\ref{lem:Varrying d} allows us to slightly modify the definition of the Weak Injective Radius Property. In particular, the following two claims are equivalent to each other and the Weak Injective Radius Property with parameter $\alpha$, for $C$ some fixed constant (say $C=100$):
\begin{itemize} 
    \item For
every $d_{0}\le2\alpha\log_{q}([\Gamma_1:\Gamma_N])-C$,
$\epsilon>0$,
\[
\frac{1}{[\Gamma_1:\Gamma_N]}\sum_{x\in \Gamma_N \backslash \Gamma}\boldsymbol{N}(\Gamma_N,d_{0},x)dx\ll_{\epsilon}[\Gamma_1:\Gamma_N]^{\epsilon}q^{d_{0}/2}.
\]
\item For
every $d_{0}\le2\alpha\log_{q}([\Gamma_1:\Gamma_N])+C$,
$\epsilon>0$,
\[
\frac{1}{[\Gamma_1:\Gamma_N]}\sum_{x\in \Gamma_N \backslash \Gamma}\boldsymbol{N}(\Gamma_N,d_{0},x)dx\ll_{\epsilon}[\Gamma_1:\Gamma_N]^{\epsilon}q^{d_{0}/2}.
\]
\end{itemize}

It also implies the following proposition, which should be compared with \citep[Conjecture 2]{sarnak1991bounds}:
\begin{prop}
\label{prop:Counting for large}Let $(\Gamma_N) $
be a sequence of lattices. Assume that the Weak Injective Radius Property
holds with parameter $\alpha=1$, i.e. for every $0\le d_{0}\le2\log_{q}([\Gamma_1:\Gamma_N])$,
$\epsilon>0$,
\[
\frac{1}{[\Gamma_1:\Gamma_N]}\sum_{y\in\Gamma_{N}\backslash\Gamma_{1}}\boldsymbol{N}(\Gamma_N,d_{0},y)\ll_{\epsilon}[\Gamma_1:\Gamma_N]^{\epsilon}q^{d_{0}(\nicefrac{1}{2}+\epsilon)}.
\]
Then for every $d_{0}\ge0,\epsilon>0$ it holds that 
\[
\frac{1}{[\Gamma_1:\Gamma_N]}\sum_{y\in\Gamma_{N}\backslash\Gamma_{1}}\boldsymbol{N}(\Gamma_N,d_{0},y)\ll_{\epsilon}[\Gamma_1:\Gamma_N]^{\epsilon}q^{d_{0}\epsilon}(\frac{q^{d_0}}{[\Gamma_1:\Gamma_N]}+q^{d_{0}/2}).
\]
\end{prop}
\begin{proof}
It is sufficient to prove that for $d_1\ge 0$, $d_0=2\log_{q}([\Gamma_1:\Gamma_N])$, it holds that 
\[
\frac{1}{[\Gamma_1:\Gamma_N]}\sum_{y\in\Gamma_{N}\backslash\Gamma_{1}}\boldsymbol{N}(\Gamma_N,d_{0}+d_1,y)\ll_{\epsilon}q^{(d_{0}+d)\epsilon}q^{d_0/2+d_1}.
\]
This follows from Lemma~\ref{lem:Varrying d}.
\end{proof}

We can now present the proof of Proposition~\ref{prop:Trace equivalent to Lattice point Counting}, in the following slightly more general claim.
\begin{cor}
\label{cor:trace equivalence}Let $(\Gamma_N)$ be a sequence of cocompact lattices, and $0<\alpha\le1$. The following are equivalent:
\begin{enumerate}
\item For every $d_{0}\le2\alpha\log_{q}([\Gamma_1:\Gamma_N])$,
$\epsilon>0$,
\[
\tr \chi_{d_0}|_{L^{2}(X_N)}=\intop_{\Gamma_{N}\backslash G}\sum_{\gamma\in\Gamma_N}\chi_{d_0}(x^{-1}\gamma x)dx\ll_{\epsilon}[\Gamma_1:\Gamma_N]^{1+\epsilon}q^{d_{0}(\nicefrac{1}{2}+\epsilon)}.
\]
\item The Weak Injective Radius Property with parameter $\alpha$ -- for
every $d_{0}\le2\alpha\log_{q}([\Gamma_1:\Gamma_N])$,
$\epsilon>0$,
\[
\frac{1}{[\Gamma_1:\Gamma_N]}\sum_{x\in \Gamma_N \backslash \Gamma}\boldsymbol{N}(\Gamma_N,d_{0},x)dx\ll_{\epsilon}[\Gamma_1:\Gamma_N]^{\epsilon}q^{d_{0}(\nicefrac{1}{2}+\epsilon)}.
\]
\item For every $h\in C_c(G)$ self-adjoint and satisfying
$h(g)\ll_{\epsilon}[\Gamma_1:\Gamma_N]^{\epsilon}\psi_{\tilde{d}}(g)$
for $\tilde{d}=2\alpha\log_{q}([\Gamma_1:\Gamma_N])$,
it holds that 
\begin{equation}
\tr h|_{L^{2}(X_N)}\ll_{\epsilon}[\Gamma_1:\Gamma_N]^{1+\epsilon}q^{\tilde{d}(\nicefrac{1}{2}+\epsilon)}\asymp[\Gamma_1:\Gamma_N]^{1+\alpha+\epsilon}.\label{eq:Trace equivalence3}
\end{equation}
\item For $\tilde{d}=2\alpha\log_{q}([\Gamma_1:\Gamma_N])$,
it holds that 
\[
\tr \psi_{\tilde{d}}|_{L^{2}(X_N)}\ll_{\epsilon}[\Gamma_1:\Gamma_N]^{1+\epsilon}q^{\tilde{d}(\nicefrac{1}{2}+\epsilon)}\asymp[\Gamma_1:\Gamma_N]^{1+\alpha+\epsilon}.
\]
\end{enumerate}
\end{cor}

\begin{proof}
$(3)$ obviously implies $(4)$, and the fact that $(4)$ implies $(3)$ a result of Lemma~\ref{lem:positivity of trace}.

Since $\Gamma_1$ is cocompact, it has a finite diameter $D$. It implies that the balls of radius $D$ around points in $\Gamma_N\backslash \Gamma_1$ cover the entire space $\Gamma_N \backslash G$. Moreover, if two points $x,y$ are of distance $d$ apart, then 
\[
\sum_{\gamma\in\Gamma_N}\chi_{d_0-2d-1}(x^{-1}\gamma x)
\le \sum_{\gamma\in\Gamma_N}\chi_{d_0}(y^{-1}\gamma y)
\le \sum_{\gamma\in\Gamma_N}\chi_{d_0+2d+1}(x^{-1}\gamma x).
\]

Using Equation~\eqref{eq:relation between hd and N}, we deduce that if $B_{x_0,D}$ is the ball of radius $D$ around $x_0\in \Gamma_N \backslash \Gamma_1$, then for every $d_0$, 
\[
\intop_{B_{x_0,D}} \sum_{\gamma\in\Gamma_N}\chi_{d_0-2D-2}(x^{-1}\gamma x)dx \le \boldsymbol{N}(\Gamma_N,d_{0},x) \le 
\intop_{B_{x_0,D}} \sum_{\gamma\in\Gamma_N}\chi_{d_0+2D+2}(x^{-1}\gamma x)dx.
\]

Summing over $x_0 \in \Gamma_N\backslash \Gamma_1$ and using the discussion after Lemma~\ref{lem:Varrying d} to change $d_0$ by a constant allows us to deduce the equivalence between $(1)$ and $(2)$.

To show that $(4)$ implies $(1)$, note that
for $d_{0}\le\tilde{d}=2\alpha\log_{q}([\Gamma_1:\Gamma_N])$,
it holds that $\chi_{d_0}\ll q^{(d_{0}-\tilde{d})/2}\psi_{\tilde{d}}(g)$.
Then if $(4)$ holds then 
\begin{align*}
\tr \chi_{d_0}|_{L^{2}(X_N)} & \ll q^{(d_{0}-\tilde{d})/2}\tr \psi_{\tilde{d}}|_{L^{2}(X_N)}\\
 & \ll q^{(d_{0}-\tilde{d})/2}[\Gamma_1:\Gamma_N]^{1+\epsilon}q^{\tilde{d}(\nicefrac{1}{2}+\epsilon)}\\
 & \ll[\Gamma_1:\Gamma_N]^{1+\epsilon}q^{d_{0}(\nicefrac{1}{2}+\epsilon)}.
\end{align*}

Finally, we prove that $(1)$ implies $(4)$.
Note that for every $g\in G$,
\[
\psi_{\tilde{d}}(g)\ll\intop_{0}^{\tilde{d}}q^{(\tilde{d}-d_{0})/2}\chi_{d_0}(g)dd_{0}.
\]
Then if $(1)$ holds for every $d_{0}\le\tilde{d}$ we have
\begin{align*}
\tr \psi_{\tilde{d}} & \ll\intop_{0}^{\tilde{d}}q^{(\tilde{d}-d_{0})/2}\tr \chi_{d_0}|_{L^{2}(X_N)}dd_{0}\\
 & \ll_{\epsilon}[\Gamma_1:\Gamma_N]^{1+\epsilon}\intop_{0}^{\tilde{d}}q^{(\tilde{d}-d_{0})/2}q^{d_{0}(\nicefrac{1}{2}+\epsilon)}dd_{0}\\
 & =[\Gamma_1:\Gamma_N]^{1+\epsilon}\tilde{d}q^{\tilde{d}(\nicefrac{1}{2}+\epsilon)}\\
 & \ll_{\epsilon}[\Gamma_1:\Gamma_N]^{1+\epsilon}q^{\tilde{d}(\nicefrac{1}{2}+\epsilon)}.
\end{align*}
\end{proof}

Another simple property of the Weak Injective Radius is:
\begin{prop}
\label{prop:A le 1 for injective radius}Let $(\Gamma_N) $
be a sequence of lattices. If $\alpha$ is the Weak Injective Radius
parameter of the sequence, then $\alpha\le1$.
\end{prop}

\begin{proof}
For simplicity, we assume that the lattices $(\Gamma_N) $
are cocompact, so we can use Equation~\eqref{eq:Trace equivalence3}.
The arguments can change to deal with the non-uniform case as well.

Note that Corollary~\ref{cor:trace equivalence} did not assume that
$\alpha\le1$. It is therefore enough to prove that for every $\alpha>1$
Equation~\eqref{eq:Trace equivalence3} from Corollary~\ref{cor:trace equivalence}
does not hold for $\alpha$.

Let $d_{0}=2\alpha\log_{q}([\Gamma_1:\Gamma_N])$,
$d_{1}=\nicefrac{d_0}{2}-1$, and $c_{d_{1}}=\chi_{d_{1}}\ast\chi_{d_{1}}$.
By Lemma~\ref{lem:convolution lemma}, $c_{d_{1}}\ll_{\epsilon}[\Gamma_1:\Gamma_N]^{\epsilon}\psi_{d_0}(g)$,
so the condition before Equation~\eqref{eq:Trace equivalence3} holds,
and 
\[
\tr c_{d_{1}}|_{L^{2}(X_N)}\ll_{\epsilon}[\Gamma_1:\Gamma_N]^{1+\alpha+\epsilon}.
\]

Since $\chi_{d_{1}}$ is self-adjoint, $\tr \pi(c_{d_{1}})\ge0$
for every $\pi\in\Pi(G)$, so 
\[
\tr c_{d_{1}}|_{L^{2}(X_N)}\ge\tr (\pi_{\triv}(c_{d_0})),
\]

where $\pi_{\triv}$ is the trivial representation. On the other
hand, it holds that 
\[
\tr (\pi_{\triv}(c_{d_0}))=\intop_{G}c_{d_{1}}(g)dg=
\left(\intop_{G}\chi_{d_{1}}dg\right)^{2}\gg q^{2d_{1}}\gg[\Gamma_1:\Gamma_N]^{2\alpha}.
\]
As $\alpha>1$, we get a contradiction for $[\Gamma_1:\Gamma_N]$ big enough.
\end{proof}

We can now prove Proposition~\ref{prop:good family proposition}: 
\begin{proof}
Recall that we assume that the sequence $(\Gamma_N) $
of cocompact lattices satisfies the Weak Injective Radius Property
with parameter $\alpha$, let $d_{0}=\alpha\log_{q}([\Gamma_1:\Gamma_N])$
and let $f_{d_0}$ from the definition of a good family. 

By Corollary~\ref{cor:trace equivalence},
the trace of $f_{d_0}$ on $L^{2}(\Gamma_{N}\backslash G)$
satisfies
\[
\tr f_{d_0}|_{L^{2}(\Gamma_{N}\backslash G)}\ll_{\epsilon}[\Gamma_1:\Gamma_N]^{1+\alpha+\epsilon}.
\]

Let us calculate the spectral side of the trace. From the second and
first properties of a good family,
\begin{align*}
\tr f_{d_0}|_{L^{2}(\Gamma_{N}\backslash G)} & \ge\sum_{\pi\in A}m(\pi,\Gamma_N)\tr \pi(f_{d_0})\\
 & \gg_{\epsilon,A}\sum_{\pi\in A}m(\pi,\Gamma_N)q^{d_0(1-\nicefrac{1}{p(\pi)}-\epsilon)}\\
 & =\sum_{\pi\in A}m(\pi,\Gamma_N)[\Gamma_1:\Gamma_N]^{2\alpha(1-\nicefrac{1}{p(\pi)}-\epsilon)}\\
 & \ge M(A,\Gamma_N,p)[\Gamma_1:\Gamma_N]^{2\alpha(1-\nicefrac{1}{p}-\epsilon)}.
\end{align*}

We deduce that for every $N$, $p>2$, $\epsilon>0$, 
\[
M(A,\Gamma_N,p)\ll_{A,\epsilon}[\Gamma_1:\Gamma_N]^{1-\alpha(1-\nicefrac{2}{p})+\epsilon},
\]
as needed.
\end{proof}

\section{\label{sec:Injective Radius implies Lifitng}The Weak Injective Radius Property implies the Optimal Lifting Property}

In this section, we prove Theorem~\ref{thm:injective radius implies optimal lifting}.

\subsection{Reduction to a Spectral Argument}

Recall that assuming the Weak Injective Radius Property and Spectral
Gap, we should prove that for every $\epsilon>0$, for every $a\in A_{+}$
with $l(a)\ge(1+\epsilon)\log_{q}(\mu(X_N))$,
\begin{equation}
\#\left\{ (x,y)\in(\Gamma_1/\Gamma_N)^{2}:\exists\gamma\in\Gamma_1\text{ s.t. }\pi_{N}(\gamma)x=y,\n{a_{\gamma}-a}_{\a}<\epsilon\n a_{\a}\right\} =(1-o_{\epsilon}(1))[\Gamma_1:\Gamma_N]^{2}.\label{eq:Need to prove}
\end{equation}

For $(x,y)\in(\Gamma_1/\Gamma_N)^{2}$,
$a\in A_{+}$ and $\epsilon>0$, we say that $\gamma\in\Gamma_1$ is
\emph{good} for $(x,y,a,\epsilon)$ if $\pi_{N}(\gamma)x=y$
and $\n{a_{\gamma}-a}_{\a}<\epsilon\n a_{\a}$.
\begin{lem}
\label{lem: Optimal lifting to zeros}Let $(x,y)\in(\Gamma_1/\Gamma_N)^{2}$,
and assume that there is no good $\gamma\in\Gamma_1$ for $(x,y,a,\epsilon)$.
Identify $x,y$ with elements $x,y\in X_{N}=\Gamma_{N}\backslash G/K$.
Let $f_{a}\in L^{1}(G)$ be a function supported on the
set $\left\{ g\in G:\n{a_{g}-a}_{\a}<\epsilon/2\n a_{\a}\right\} $,
and for $\delta$ small enough with respect to $\epsilon\n{a}_\a$ let $b_{x,\delta}$
as in Subsection~\ref{subsec:Spectral-Decomposition-of-small-ball}.
Then 
\[
f_{a}b_{x,\delta}(y)=0.
\]
Moreover, for every $y'\in B_{\delta}(y)$ it holds that 
\[
f_{a}b_{x,\delta}(y')=0.
\]
\end{lem}

\begin{proof}
We think of $f_{a}b_{x,\delta}$ as a left $\Gamma_{N}$-invariant function
on $G$, and identify $x,y$ with some lifts of them in $\Gamma_1 \subset G$. The support of $f_{a}b_{x,\delta}= b_{x,\delta}*\hat{f}_a$ is contained in the set 
\[
\left\{ y'\in G:\exists\gamma\in\Gamma_N,x'\in G,\,f_{a}(y^{\prime-1}\gamma x')>0,\,d(x,x')\le\delta\right\}.
\]
Assume by contradiction that $y'\in B_{\delta}(y)$ is
in the support of $f_{a}b_{x,\delta}$, and let $\gamma\in\Gamma_{N}$,
$x'\in G$ be such that $f_{a}(y^{\prime-1}\gamma x')>0,\,d(x,x')\le\delta$.
By the assumption on the support of $f_{a}$, 
\[
\n{a_{y^{\prime-1}\gamma x'}-a}_{\a}<\epsilon/2\n{a}_{\a}.
\]

Look at
\[
\gamma'=y^{-1}\gamma x\in\Gamma_1.
\]

Then $a_{\gamma'}=a_{y^{-1}\gamma x}$, and if $\delta$ is small
enough with respect to $\epsilon\n{a}_\a$, then $\n{a_{y^{-1}\gamma x}-a_{y^{\prime-1}\gamma x'}}_{\a}<\epsilon/2\n a_{\a}$.
Therefore, 
\[
\n{a_{\gamma'}-a}_{\a}<\epsilon\n a_{\a},
\]
and 
\[
\pi_N(\gamma')x = x\gamma^{\prime^-1} = \gamma y.
\]
Since $\gamma\in \Gamma_N$, it says that $\gamma'$ sends $x\in\Gamma_{N}\backslash\Gamma_1$
to $y\in\Gamma_{N}\backslash\Gamma_1$, as needed.
\end{proof}

Let $\pi\in L^{2}(X_N)$, be the uniform probability
distribution, i.e., $\pi(x)=\frac{1}{\mu(x_N)}$.
\begin{lem}
\label{lem:L^2 bound-1}Equation~\eqref{eq:Need to prove} holds if for every $\epsilon>0$, for every $a\in A_{+}$ with $l(a)>(1+\epsilon)\log_{q}([\Gamma_1:\Gamma_N])$,
there exists a probability function $f_{a}\in C_{c}^{\infty}(G)$
supported on $\left\{ g\in G:\n{a_{g}-a}_{\a}<\epsilon\n a_{\a}\right\} $,
such that 
\begin{equation}
\sum_{x\in\Gamma_1/\Gamma_{N}}\n{f_{a}b_{x,\delta}-\pi}_{2}^{2}=o_{\epsilon}(1).\label{eq:L^2 bound-1}
\end{equation}
\end{lem}

\begin{proof}
We assume that Equation~\eqref{eq:L^2 bound-1} holds and want to prove that Equation~\eqref{eq:Need to prove} holds.

Let $\epsilon>0$. For $x,y\in(\Gamma_1/\Gamma_N)^{2}$,
by Lemma~\ref{lem: Optimal lifting to zeros}, if there is no good
$\gamma$ for $(x,y,a,\epsilon/2)$, then for $y'$ in
the $\delta$-neighborhood of $y$ it holds that $\left|(f_{a}b_{x,\delta}-\pi)(y')\right|=\pi(y')=\frac{1}{\mu(X_N)}\asymp[\Gamma_1:\Gamma_N]^{-1}$.
Therefore, for a fixed $x\in\Gamma_1/\Gamma_{N}$ each $y$ without
good $\gamma$ contributes $\gg_{\delta}[\Gamma_1:\Gamma_N]^{-2}$
to $\n{f_{a}b_{x,\delta}-\pi}_{2}^{2}$. Moreover, the contributions are distinct for $y,y'$ whose image in $X_N =\Gamma_N \backslash G / K$ is different.

Therefore,
\begin{align*}
 & \#\left\{ (x,y)\in(\Gamma_1/\Gamma_N)^{2}:\text{There is no good }\gamma\text{ for }(x,y,a,\epsilon/2)\right\} \\
 & \ll\sum_{x\in\Gamma_1/\Gamma_{N}}\n{f_{a}b_{x,\delta}-\pi}_{2}^{2}[\Gamma_1:\Gamma_N]^{2}\\
 & =[\Gamma_1:\Gamma_N]^{2}\sum_{x\in\Gamma_1/\Gamma_{N}}\n{f_{a}b_{x,\delta}-\pi}_{2}^{2}\\
 & =o([\Gamma_1:\Gamma_N]^{2}),
\end{align*}
where we used Equation~\eqref{eq:L^2 bound-1} in the last step. This implies Equation~\eqref{eq:Need to prove} for $\epsilon/2$.
\end{proof}
The following lemma explains where Spectral Gap is used. 
\begin{lem}
\label{lem:L^2 bound-2}Equation~\eqref{eq:Need to prove} follows
from the following two conditions:
\begin{enumerate}
\item Spectral Gap holds for $(\Gamma_N) $.
\item For every $\epsilon>0$, for some $\delta>0$, for every $a\in A_{+}$ with $l(a)\ge(1+\epsilon)\log_{q}([\Gamma_1:\Gamma_N])$,
there exists a probability function $f_{a}\in C_{c}^{\infty}(G)$
supported on $\left\{ g\in G:\n{a_{g}-a}_{\a}<\epsilon\n a_{\a}\right\} $,
such that  for every $\epsilon_{1}>0$,
\begin{equation}
\sum_{x\in\Gamma_1/\Gamma_{N}}\n{f_{a}b_{x,\delta}}_{2}^{2}\ll_{\epsilon,\epsilon_{1}}q^{\epsilon_{1}l(a)}.\label{eq:L^2 bound-1-1}
\end{equation}
\end{enumerate}
\end{lem}

\begin{proof}
First, we note that $\n{\pi}_{2}^{2}=\intop_{X_{N}}\mu(X_N)^{-2}dx=\mu(X_N)^{-1}\asymp[\Gamma_1:\Gamma_N]^{-1}$.
Therefore, if Equation~\eqref{eq:L^2 bound-1-1} holds then also 
\[
\sum_{x\in\Gamma_1/\Gamma_{N}}\n{f_{a}b_{x,\delta}-\pi}_{2}^{2}\le\sum_{x\in\Gamma_1/\Gamma_{N}}\n{f_{a}b_{x,\delta}}_{2}^{2}+\sum_{x\in\Gamma_1/\Gamma_{N}}\n{\pi}_{2}^{2}\ll_{\epsilon,\epsilon_{1}}q^{\epsilon_{1}l(a)},
\]
which is similar to Equation~\eqref{eq:L^2 bound-1}, but $o(1)$ is
replaced with $O_{\epsilon,\epsilon_{1}}(q^{\epsilon_{1}l(a)})$.

Let $\epsilon'>0$ and let $a\in A_+$ be such that $l(a')=\epsilon'l(a)$. Let $f_{a}^{\prime}=A_{a'}*f_{a}$.
Assuming $\epsilon'$ is small enough, $f_{a}^{\prime}$ is supported
on 
\[
\left\{ g\in G:g\in G:\n{a_{g}-a}_{\a}<2\epsilon\n a_{\a}\right\} .
\]
We will show that Equation~\eqref{eq:L^2 bound-1} holds for $f_{a}^{\prime}$.

Notice that $A_{a'}\pi=\pi$
and $f_{a}b_{x,\delta}-\pi\perp\pi$. By the Spectral Gap assumption
and Corollary~\ref{cor:Hecke Bounds}, for some $p'<\infty$,
\begin{align*}
\n{f_{a}^{\prime}b_{x,\delta}-\pi}_{2} & =\n{A_{a'}(f_{a}b_{x,\delta}-\pi)}_{2}\\
 & \ll q^{-\epsilon'l(a)/p'}\n{f_{a}b_{x,\delta}-\pi}_{2}.
\end{align*}
Therefore, 
\begin{align*}
\sum_{x\in\Gamma_1/\Gamma_{N}}\n{f_{a'}b_{x,\delta}-\pi}_{2}^{2} & \ll q^{-\epsilon'l(a)/p'}
\sum_{x\in\Gamma_1/\Gamma_{N}}\n{f_{a}b_{x,\delta}-\pi}_{2}^{2}\\
 & \ll_{\epsilon,\epsilon_{1}}q^{-\epsilon'l(a)/p'}q^{\epsilon_1l(a)}.
\end{align*}
If we choose $\epsilon_{1}$ small enough, this is $o(1)$ and Equation~\eqref{eq:L^2 bound-1} holds. Applying Lemma~\ref{lem:L^2 bound-1}, we get that Equation~\eqref{eq:Need to prove} holds as well.
\end{proof}

\subsection{Completing the Proof of Theorem~\ref{thm:injective radius implies optimal lifting}}

Recall that Theorem~\ref{thm:injective radius implies optimal lifting}
states that Spectral Gap and the Weak Injective Radius Property imply
the Optimal Lifting Property. In Lemma~\ref{lem:L^2 bound-2} we reduced it to some spectral statement, Equation~\eqref{eq:L^2 bound-1-1}.
We now claim:
\begin{lem}
\label{lem:Weak-Injective-Radius-Implies-Spectral}Assume that the Weak Injective Radius Property holds for a sequence $(\Gamma_n)$. Then for every $\epsilon>0$ sufficiently small, for some $\delta>0$,
for every $a\in A_{+}$ with $l(a)\ge\log_{q}([\Gamma_1:\Gamma_N])$,
there is a probability function $f_{a}\in C_{c}^{\infty}(G)$
supported on $\left\{ g\in G:\n{a_{g}-a}_{\a}<\epsilon\right\}$, such that 
\[
\sum_{x\in\Gamma_1/\Gamma_{N}}\n{f_{a}b_{x,\delta}}_{2}^{2}\ll_{\epsilon,\epsilon_1}q^{\epsilon_{1}l(a)}.
\]
\end{lem}

Notice the function $f_a$ in Lemma~\ref{lem:Weak-Injective-Radius-Implies-Spectral} satisfies slightly stronger conditions than required by Lemma~\ref{lem:L^2 bound-2}, so together they imply Theorem~\ref{thm:injective radius implies optimal lifting}.
\begin{proof}
Let $\epsilon>0$ be sufficiently small. 
By a standard argument, there is a smooth probability function $f_{a}\in C_{c}^{\infty}(G)$ supported on $\left\{ g\in G:\n{a_{g}-a}_{\a}<\epsilon\right\}$, and satisfies $f_{a}(x)\ll q^{-l(a)}\chi_{l(a)+1}(x)$.

Therefore, by the Convolution Lemma~\ref{lem:convolution lemma} and Lemma~\ref{lem: counting from operator},
\begin{align*}
\n{f_{a}b_{x,\delta}}_{2}^{2} & \ll_{\epsilon}q^{-2l(a)}\n{\chi_{l(a)+1}b_{x,\delta}}_{2}^{2}=q^{-2l(a)}\left\langle \chi_{l(a) +1 }b_{x,\delta},\chi_{l(a)+1}b_{x,\delta}\right\rangle \\
 & =q^{-2l(a)}\left\langle \chi_{l(a)+1}*\chi_{l(a)+1}b_{x,\delta},b_{x,\delta}\right\rangle \\
 & \ll_{\epsilon_{1}}q^{-2(1-\epsilon_{1})l(a)}\left\langle \psi_{2(l(a)+2)}b_{x,\delta},b_{x,\delta}\right\rangle \\
 & \ll q^{-2(1-\epsilon_{1})l(a)}\intop_{0}^{2(l(a)+2)}q^{(2l(a)-d_{0})/2}\left\langle \chi_{d_0}b_{x,\delta},b_{x,\delta}\right\rangle dd_{0}.\\
 & \ll q^{-2(1-\epsilon_{1})l(a)}\intop_{0}^{2(l(a)+2)}q^{(2l(a)-d_{0})/2}N(x,\Gamma_N,d_{0}+2)dd_{0}.
\end{align*}

Therefore, applying Proposition~\ref{prop:Counting for large}, we
get 
\begin{align*}
\sum_{x\in\Gamma_1/\Gamma_{N}}\n{f_{a}b_{x,\delta}}_{2}^{2} & \ll_{\epsilon_{1}}q^{-2(1-\epsilon_{1})l(a)}\intop_{0}^{2(l(a)+2)}q^{(2l(a)-d_{0})/2}(q^{d_0}+q^{d_{0}/2}[\Gamma_1:\Gamma_N])dd_{0}\\
 & \ll_{\epsilon_{1}}q^{l(a)\epsilon_{1}}(1+q^{-l(a)}[\Gamma_1:\Gamma_N]).
\end{align*}
Since $l(a)\ge\log_{q}([\Gamma_1:\Gamma_N])$, this is $O(q^{l(a)\epsilon_{1}})$ as needed.
\end{proof}

\section{\label{sec:Density implies Stuff}The Spectral to Geometric Direction}

\subsection{Some Technical Calculations\label{subsec:Sums involving density}}

For an ease of reference, we give here a couple of technical
bounds.
\begin{lem}
\label{lem:intimidating sums lemma}Let $G$ be non-Archimedean or
rank $1$. The following are equivalent for a sequence $(\Gamma_N)$:
\begin{enumerate}
\item The Spherical Density Hypothesis with parameter $\alpha$: for every
$\epsilon>0$ and $p>2$
\[
M(\Pi(G)_{\sph},\Gamma_N,p)\ll_{\epsilon}[\Gamma_1:\Gamma_N]^{1-\alpha(1-\nicefrac{2}{p})+\epsilon}.
\]
\item For every $\epsilon>0$
\[
\sum_{\pi\in\Pi(G)_{\sph},p(\pi)>2}[\Gamma_1:\Gamma_N]^{-1+\alpha(1-2/p(\pi))}m(\pi,\Gamma_N)\ll_{\epsilon}[\Gamma_1:\Gamma_N]^{\epsilon}.
\]
\end{enumerate}
\end{lem}

\begin{proof}
The fact that (2) implies (1) is simple and is left to the reader.

The fact that (1) implies (2) follows from a standard trick of integration
by parts (\citep[Theorem 421]{hardy1979introduction}): 
\begin{align*}
 & \sum_{\pi\in\Pi(G)_{\sph},p(\pi)>2}[\Gamma_1:\Gamma_N]^{-1+\alpha(1-2/p(\pi))}m(\pi,\Gamma_N)\\
 & =\lim_{p\to2,p>2}M(\Pi(G)_{\sph},\Gamma_N,p)[\Gamma_1:\Gamma_N]^{-1+\alpha(1-\nicefrac{2}{p})}\\
 & \,\,\,+\intop_{2}^{\infty}M(\Pi(G)_{\sph},\Gamma_N,p)\frac{\partial}{\partial p}([\Gamma_1:\Gamma_N]^{-1+\alpha(1-\nicefrac{2}{p})})dp\\
 & =\lim_{p\to2,p>2}M(\Pi(G)_{\sph},\Gamma_N,p)[\Gamma_1:\Gamma_N]^{-1}\\
 & \,\,\,+\intop_{2}^{\infty}M(\Pi(G)_{\sph},\Gamma_N,p)2p^{-2}\alpha\ln([\Gamma_1:\Gamma_N])[\Gamma_1:\Gamma_N]^{-1+\alpha(1-\nicefrac{2}{p})}dp\\
 & \ll_{\epsilon}[\Gamma_1:\Gamma_N]^{\epsilon}(1+\intop_{2}^{\infty}p^{-2}dp)\asymp[\Gamma_1:\Gamma_N]^{\epsilon}.
\end{align*}
\end{proof}
For the higher rank Archimedean case we have:
\begin{lem}
\label{lem:intimidating sums lemma high rank}Let $G$ be Archimedean.
The following are equivalent for a sequence $(\Gamma_N) $:
\begin{enumerate}
\item The Spherical Density Hypothesis with parameter $\alpha$: for some
$L>0$ large enough and for every $\lambda\ge0$, $N\ge1$, $p>2$,
$\epsilon>0$, 
\[
M(\Pi(G)_{\sph},\Gamma_N,p,\lambda)\ll_{\epsilon}(1+\lambda)^{L}[\Gamma_1:\Gamma_N]^{1-\alpha(1-\nicefrac{2}{p})+\epsilon}
\]
\item For some $L'>0$ large enough and every $\epsilon>0$,
\begin{equation}\label{eq:intimidating sum high rank}
\sum_{\pi\in\Pi(G)_{\sph},p(\pi)>2}[\Gamma_1:\Gamma_N]^{-1+\alpha(1-\nicefrac{2}{p(\pi)})}(1+\lambda(\pi))^{-L'}m(\pi,\Gamma_N)\ll_{\epsilon}[\Gamma_1:\Gamma_N]^{\epsilon}.    
\end{equation}
\end{enumerate}
\end{lem}

\begin{proof}
The fact that (2) implies (1) is again simple and is left to the reader.

The fact that (1) implies (2) is again done by integration by parts,
with two variables. Let us state it formally. If $((p_{i},\lambda_{i}))_{i=1}^{\infty}\subset(2,\infty)\times(0,\infty)$
is a sequence of values without limit points, $f(p,\lambda)$
is a non-negative smooth function, and $M'(p,\lambda)=\#\left\{ i:p_{i}\ge p,\lambda_{i}\le\lambda\right\}$
then whenever everything absolutely converges,
\begin{align*}
\sum_{i}f(p_{i},\lambda_{i}) & =\lim_{p\to2}\lim_{\lambda\to\infty}M'(p,\lambda)f(p,\lambda)-\lim_{p\to2}\intop_{0}^{\infty}M'(p,\lambda)\frac{\partial}{\partial\lambda}f(p,\lambda)d\lambda\\
 & \,\,\,+\lim_{\lambda\to\infty}\intop_{2}^{\infty}M'(p,\lambda)\frac{\partial}{\partial p}f(p,\lambda)dp\\
 & \,\,\,-\intop_{0}^{\infty}\intop_{2}^{\infty}M'(p,\lambda)\frac{\partial}{\partial\lambda}\frac{\partial}{\partial p}f(p,\lambda)dpd\lambda.
\end{align*}

Applying the integration by parts formula to the left-hand side
of Equation~\eqref{eq:intimidating sum high rank}, with $M'(p,\lambda)=M(\Pi(G)_{\sph},\Gamma_N,p,\lambda)$,
$f(p,\lambda)=[\Gamma_1:\Gamma_N]^{-1+\alpha(1-\nicefrac{2}{p})}(1+\lambda)^{-L'}$,
we get 
\begin{align}
L.H.S & =\lim_{p\to2}\lim_{\lambda\to\infty}M(\Pi(G)_{\sph},\Gamma_N,p,\lambda)[\Gamma_1:\Gamma_N]^{-1+\alpha(1-\nicefrac{2}{p})}(1+\lambda)^{-L'}\nonumber \\
 & +\lim_{p\to2}\intop_{0}^{\infty}M(\Pi(G)_{\sph},\Gamma_N,p,\lambda)[\Gamma_1:\Gamma_N]^{-1+\alpha(1-\nicefrac{2}{p})}L'(1+\lambda)^{-L'-1}d\lambda\nonumber \\
 & +\lim_{\lambda\to\infty}\intop_{2}^{\infty}M(\Pi(G)_{\sph},\Gamma_N,p,\lambda)2\alpha p^{-2}\ln([\Gamma_1:\Gamma_N])[\Gamma_1:\Gamma_N]^{-1+\alpha(1-\nicefrac{2}{p})}(1+\lambda)^{-L'}dp\label{eq:IntegrationX2}\\
 & +\intop_{0}^{\infty}\intop_{2}^{\infty}M(\Pi(G)_{\sph},\Gamma_N,p,\lambda)2\alpha p^{-2}\ln([\Gamma_1:\Gamma_N])[\Gamma_1:\Gamma_N]^{-1+\alpha(1-\nicefrac{2}{p})}L'(1+\lambda)^{-L'-1}dpd\lambda\nonumber 
\end{align}

Note that by the Spherical Density Hypothesis, 
\[
\lim_{p\to 2}M(\Pi(G)_{\sph},\Gamma_N,p,\lambda)\ll_{\epsilon}[\Gamma_1:\Gamma_N]^{1+\epsilon}(1+\lambda)^{L}.
\]

A similar and more precise bound may be derived directly from Weyl's law (\citep{duistermaat1979spectra})). Note also that we may assume that $L'>L+10$.

Therefore, the first summand in Equation~\eqref{eq:IntegrationX2} is
bounded by
\[
\ll_{\epsilon}\lim_{\lambda\to\infty}[\Gamma_1:\Gamma_N]^{1+\epsilon}[\Gamma_1:\Gamma_N](1+\lambda)^{L}(1+\lambda)^{-L'}=0.
\]

The second summand is bounded similarly by 
\[
\ll_{\epsilon}\intop_{0}^{\infty}[\Gamma_1:\Gamma_N]^{1+\epsilon}[\Gamma_1:\Gamma_N]^{-1}(1+\lambda)^{L}L'(1+\lambda)^{-L'-1}d\lambda\ll[\Gamma_1:\Gamma_N]^{\epsilon},
\]

The third summand is bounded by 
\begin{align*}
\ll_{\epsilon} & \lim_{\lambda\to\infty}\intop_{2}^{\infty}(1+\lambda)^{L}[\Gamma_1:\Gamma_N]^{1-\alpha(1-\nicefrac{2}{p})+\epsilon}2\alpha p^{-2}\ln([\Gamma_1:\Gamma_N])[\Gamma_1:\Gamma_N]^{-1+\alpha(1-\nicefrac{2}{p})}(1+\lambda)^{-L'}dp\\
\ll & \lim_{\lambda\to\infty}(1+\lambda)^{L-L'}\intop_{2}^{\infty}[\Gamma_1:\Gamma_N]^{\epsilon}p^{-2}dp=0,
\end{align*}

The final summand is bounded by 
\begin{align*}
\ll_{\epsilon} & \intop_{0}^{\infty}\intop_{2}^{\infty}(1+\lambda)^{L}[\Gamma_1:\Gamma_N]^{1-\alpha(1-\nicefrac{2}{p})+\epsilon}2\alpha p^{-2}\ln([\Gamma_1:\Gamma_N])[\Gamma_1:\Gamma_N]^{-1+\alpha(1-\nicefrac{2}{p})}L'(1+\lambda)^{-L'-1}dpd\lambda\\
\ll & [\Gamma_1:\Gamma_N]^{\epsilon}\intop_{0}^{\infty}L'(1+\lambda)^{L-L'-1}d\lambda\intop_{2}^{\infty}p^{-2}dp\ll[\Gamma_1:\Gamma_N]^{\epsilon}.
\end{align*}
By combining all the bounds we get Equation~\eqref{eq:intimidating sum high rank}.
\end{proof}

\subsection{\label{subsec:Proof-of-density implies injective radius}Proof of
Theorem~\ref{thm:density implies counting}}

In this subsection we prove Theorem~\ref{thm:density implies counting}, or more explicitly we prove that the Spherical Density Hypothesis with parameter $\alpha$ implies the Weak Injective Radius Property with parameter $\alpha$. The most natural proof of the claim is to analyze the spectral side of the pre-trace formula for the function $\chi_{d_0}$. We will
instead discretize and prove directly the Weak Injective Radius Property.
\begin{proof}
Recall that we should prove that for every $d_{0}\le2\alpha\log_{q}([\Gamma_1:\Gamma_N])$,
$\epsilon>0$,
\[
\frac{1}{[\Gamma_1:\Gamma_N]}\sum_{x\in\Gamma_1/\Gamma_{N}}\boldsymbol{N}(\Gamma_N,d_{0},x)dx\ll_{\epsilon}[\Gamma_1:\Gamma_N]^{\epsilon}q^{d_{0}(\nicefrac{1}{2}+\epsilon)}.
\]

For $x\in\Gamma_{N}\backslash\Gamma_1$, let $b_{x,\delta}\in L^{2}(X_N)$
as in Subsection~\ref{subsec:Spectral-Decomposition-of-small-ball}.

Recall from Lemma~\ref{lem: counting from operator} that
\[
\boldsymbol{N}(X_{N},d_{0},x)\ll\left\langle b_{x,\delta},\chi_{d_{0}+2}b_{x,\delta}\right\rangle .
\]

Let $\{ \pi_{i}\} _{i=1}^{T}$ be an orthogonal basis of irreducible
subrepresentations of $L^{2}(\Gamma_{N}\backslash G)$ with
$K$-fixed vectors and $p(\pi_{i})>2$ ($T$ is finite in the $p$-adic or real rank $1$ case, otherwise $T$ may be $\infty$).
Recall that the set of $K$-invariant vectors of each irreducible representation $\pi_{i}$ is a one dimensional vector space. Let $u_{i}\in L^{2}(X_N)$ be a $K$-invariant vector in $\pi_{i}$ with $\n{u_{i}}=1$. Let
$p_{0}=2$ and let $V_{0}$ the orthogonal complement of $\spann\left\{ \pi\right\} \oplus(\oplus_{i}\spann\left\{ u_{i}\right\} )$
in $L^{2}(X_N)$ ($\pi$ here is the uniform probability
function on $X_{N}$). Note that the $G$-representation generated
by $V_{0}$ is ($2$-)tempered.

Decompose $b_{x,\delta}=\pi+v_{0,x}+v_{1,x}+...$, according to the
decomposition $L^{2}(X_N)=\spann\left\{ \pi\right\} \oplus V_{0}\oplus \spann{u_1}\oplus...$,
i.e., for $i=1,2,...,T$, $v_{i,x}=\langle u_{i},b_{x,\delta}\rangle u_{i}$.
Then,
\begin{align}
\sum_{x\in\Gamma_1/\Gamma_{N}}\left\langle b_{x,\delta},\chi_{d_{0}+1}b_{x,\delta}\right\rangle  & =\sum_{x\in\Gamma_1/\Gamma_{N}}\left\langle \pi,\chi_{d_{0}+1}\pi\right\rangle +\label{eq:density implies diophatine}\\
 & \sum_{x\in\Gamma_1/\Gamma_{N}}\left\langle v_{0,x},\chi_{d_{0}+1}v_{0,x}\right\rangle +\sum_{x\in\Gamma_1/\Gamma_{N}}\sum_{i=1}^{T}\left\langle v_{i,x},\chi_{d_{0}+1}v_{i,x}\right\rangle .\nonumber 
\end{align}

The first summand in Equation~\eqref{eq:density implies diophatine}
equals
\begin{align*}
\lambda_{\triv}(\chi_{d_{0}+C})[\Gamma_1:\Gamma_N]\n{\pi}_{2}^{2} & \ll_{\epsilon}q^{d_{0}(1+\epsilon)}[\Gamma_1:\Gamma_N]\mu^{-1}(X_N)\ll q^{d_{0}(1+\epsilon)}\\
 & \ll q^{d_{0}(\nicefrac{1}{2}+\epsilon)}[\Gamma_{1}:\Gamma_N],
\end{align*}
where $\lambda_{\triv}(\chi_{d_{0}+C})$ is the trivial
eigenvalue of $\chi_{d_{0}+C}$, and we used the fact that $d_{0}\le2\alpha\log_{q}([\Gamma_1:\Gamma_N])$.

Since $V_{0}$ spans a tempered representation, by Corollary~\ref{cor:Hecke Bounds} the second summand in Equation~\eqref{eq:density implies diophatine}
is bounded by
\[
\ll_{\epsilon}q^{d_{0}(\nicefrac{1}{2}+\epsilon)}\sum_{x\in\Gamma_1/\Gamma_N}\n{v_{0,x}}_{2}^{2}\ll q^{d_{0}(\nicefrac{1}{2}+\epsilon)}[\Gamma_{1}:\Gamma_N]\n{b_{x,\delta}}_{2}^{2}\ll q^{d_{0}(\nicefrac{1}{2}+\epsilon)}[\Gamma_{1}:\Gamma_N].
\]

To analyze the final summand in Equation~\eqref{eq:density implies diophatine},
we first assume that $G$ is $p$-adic or rank $1$. By Lemma~\ref{lem:Covering Lemma 2},
$\sum_{x\in\Gamma_1/\Gamma_N}\n{v_{i,x}}_{2}^{2}\ll_{\delta}1$. 
Write $d_{0}=2\alpha'\log_{q}([\Gamma_{1}:\Gamma_N])$,
for $\alpha'\le\alpha$. Then using Corollary~\ref{cor:Hecke Bounds},
\begin{align*}
\sum_{x\in\Gamma_1/\Gamma_N}\sum_{i=1}^{T}\left\langle v_{i,x},\chi_{d_{0}+1}v_{i,x}\right\rangle  & \ll_{\epsilon}\sum_{x\in\Gamma_1/\Gamma_N}\sum_{i=1}^{T}q^{d_{0}(1-1/p(\pi_{i})+\epsilon)}\n{v_{i,x}}_{2}^{2}\\
 & \ll \sum_{i=1}^{T}q^{d_{0}(1-1/p(\pi_{i})+\epsilon)}\\
 & \ll\sum_{\pi\in\Pi(G)_{\sph},p(\pi)>2} m(\pi,\Gamma_N)q^{d_{0}(1-1/p(\pi)+\epsilon)}\\
 & =\sum_{\pi\in\Pi(G)_{\sph},p(\pi)>2} m(\pi,\Gamma_N)[\Gamma_{1}:\Gamma_N]^{2\alpha'(1-1/p(\pi)+\epsilon)}.
\end{align*}
Applying Lemma~\ref{lem:intimidating sums lemma} (for $\alpha'\le\alpha$)
and arranging, we get
\[
\ll_{\epsilon}[\Gamma_{1}:\Gamma_N]^{1+\alpha'+\epsilon}=q^{d_{0}(\nicefrac{1}{2}+\epsilon)}[\Gamma_{1}:\Gamma_N].
\]
This finishes the proof of the non-Archimedean and the rank $1$ case.

For the Archimedean high-rank case, by Lemma~\ref{lem:Covering Lemma 2},
for $L$ large enough 
\[
\sum_{x\in\Gamma_1/\Gamma_N}\n{v_{i,x}}_{2}^{2}\ll_{L}(1+\lambda(\pi_{i}))^{-L}.
\]

The rest of the argument is as above, but using Lemma~\ref{lem:intimidating sums lemma high rank}
instead of Lemma~\ref{lem:intimidating sums lemma}.
\end{proof}
\begin{rem}
The proof of Theorem~\ref{thm:density implies counting} for hyperbolic
spaces actually works for non-compact quotients as well. The reason
is that the entire continuous spectrum is tempered, i.e., contained in $V_{0}$ (\citep{lax1982asymptotic}).
For more general results about hyperbolic surfaces, see \citep{golubev2019cutoff}.
\end{rem}

\subsection{\label{subsec:Density implies optimal lifting}A strong version of
Theorem~\ref{thm:injective radius implies optimal lifting}, assuming
the Spherical Density Hypothesis}

The goal of this Subsection is to prove the following theorem:
\begin{thm}
\label{thm:spherical density implies optimal lifting}Let $(\Gamma_N) $
be a sequence satisfying the Spherical Density Hypothesis (with parameter
$\alpha=1$) and Spectral Gap. Then for every $\epsilon>0$,
for every $a\in A_{+}$ with $l(a)\ge(1+\epsilon)\log_{q}(\mu(X_N))$,
\[
\#\left\{ (x,y)\in(\Gamma_1/\Gamma_N)^{2}:\exists\gamma\in\Gamma_1\text{ s.t. }\pi_{N}(\gamma)x=y,\n{a_{\gamma}-a}_{\a}<1 \right\} =(1-o_{\epsilon}(1))[\Gamma_{1}:\Gamma_N]^{2}.
\]
\end{thm}

The result of Theorem~\ref{thm:spherical density implies optimal lifting}
is stronger than in Theorem~\ref{thm:injective radius implies optimal lifting}
-- here we determine the $A_{+}$-component of the Cartan decomposition
of $\gamma$ in a far greater precision. In the non-Archimedean case
it says that we may choose $a\in A_{+}$ precisely.

\begin{proof}
As before, let $\pi\in L^{2}(X_N)$, be the uniform probability
distribution, i.e., $\pi(x)=\frac{1}{\mu(x_N)}$.

Using the same arguments as in Lemma~\ref{lem:L^2 bound-1}, to prove Theorem~\ref{thm:spherical density implies optimal lifting} it suffices to prove under its assumptions that for $a\in A_{+}$
with $l(a)>(1+\epsilon)\log_{q}([\Gamma_{1}:\Gamma_N])$
\begin{equation}
\sum_{x\in\Gamma_1/\Gamma_N}\n{A_{a}b_{x,\delta}-\pi}_{2}^{2}=o_{\epsilon,\delta}(1).\label{eq:L^2 bound}
\end{equation}

Let us first show that Equation~\eqref{eq:L^2 bound} is immediate if
$M(\Pi(G)_{\sph},\Gamma_N,p)=0$ for $p>2$ (the \emph{Ramanujan} case). If $l(a)>(1+\epsilon)\log_{q}([\Gamma_{1}:\Gamma_N])$
and then $q^{-l(a)}\ll[\Gamma_{1}:\Gamma_N]^{-(1+\epsilon)}.$
Then if we apply Corollary~\ref{cor:Hecke Bounds}, for $\epsilon'$
sufficiently small, 
\begin{align}
\n{A_{a}b_{x_0,\delta}-\pi}_{2}^{2} & =\n{A_{a}(b_{x_0,\delta}-\pi)}_{2}^{2}\ll_{\epsilon'}q^{-l(a)(1+\epsilon')}\n{b_{x_0,\delta}}_{2}^{2}\label{eq:Ramanujan proof}\\
 & \ll[\Gamma_{1}:\Gamma_N]^{-(1+\epsilon)(1-\epsilon')}=o([\Gamma_{1}:\Gamma_N]^{-1}).\nonumber 
\end{align}

Summing over $x\in\rho_N^{-1}(x_{0})$ we get the required
bound. As a matter of fact, we proved the stronger result that if $M(X_N,p)=0$ for $p>2$ then for every $x\in\Gamma_1/\Gamma_N$
\[
\#\left\{ y\in\Gamma_1/\Gamma_N:\exists\gamma\in\Gamma_1\text{ s.t. }\pi_{N}(\gamma)x=y,\n{a_{\gamma}-a}_{\a}<\delta\right\} =(1-o_{\epsilon,\delta}(1))[\Gamma_{1}:\Gamma_N].
\]

The proof in the general case is basically the same as the proof of
Theorem~\ref{thm:density implies counting} in the previous subsection.
Let us quickly give the differences in the proofs, while using the same notations.

The decomposition of $b_{x,\delta}$ is the same, but instead of bounding
\[
\sum_{x\in\Gamma_1/\Gamma_N}\left\langle b_{x,\delta},\chi_{d_{0}+1}b_{x,\delta}\right\rangle ,
\]
we bound for $l(a)>(1+\epsilon)\log_{q}([\Gamma_{1}:\Gamma_N])$
\[
\sum_{x\in\Gamma_1/\Gamma_N}\n{A_{a}(b_{x_0,\delta}-\pi)}_{2}^{2}.
\]

We get in the same way 
\begin{equation}
\sum_{x\in\Gamma_1/\Gamma_N}\n{A_{a}(b_{x_0,\delta}-\pi)}_{2}^{2}\ll\sum_{x\in\Gamma_1/\Gamma_N}\n{A_{a}v_{0,x}}_{2}^{2}+\sum_{x\in\Gamma_1/\Gamma_N}\sum_{i=1}^{T}\n{A_{a}v_{i,x}}_{2}^{2}.\label{eq:proof of density implies lifting}
\end{equation}

Instead of using the bound on $\chi_{d_0}$, we use the bound on
$A_{a}$, which is very similar.

For $\epsilon'$ small enough, the first summand of Equation~\eqref{eq:proof of density implies lifting} is bounded by 
\begin{align*}
\sum_{x\in\Gamma_1/\Gamma_N}\n{A_{a}v_{0,x}}_{2}^{2} & \ll_{\epsilon'}\sum_{x\in\Gamma_1/\Gamma_N}q^{-l(a)(1-\epsilon')}\n{v_{0,x}}_{2}^{2}\\
 & \ll [\Gamma_{1}:\Gamma_N]^{-(1+\epsilon)(1-\epsilon')}[\Gamma_{1}:\Gamma_N]\\
 & =o(1).
\end{align*}

The second summand of Equation~\eqref{eq:proof of density implies lifting} is
bounded in the $p$-adic or rank $1$ case by 
\begin{align*}
\sum_{x\in\Gamma_1/\Gamma_N}\sum_{i=1}^{T}\n{A_{a}v_{i,x}}_{2}^{2} & \ll_{\epsilon'}\sum_{x\in\Gamma_1/\Gamma_N}\sum_{i=1}^{T}q^{-l(a)(2/p(\pi_{i})-\epsilon')}\n{v_{i,x}}_{2}^{2}\\
 & \ll\sum_{i=1}^{T}[\Gamma_{1}:\Gamma_N]^{-(1+\epsilon)(2/p(\pi_{i})-\epsilon')}[\Gamma_{1}:\Gamma_N]\\
 & \ll[\Gamma_{1}:\Gamma_N]^{-\epsilon(2/p'-\epsilon')}\sum_{\pi\in\Pi(G)_{\sph},p(\pi)>2} m(\pi,\Gamma_N)[\Gamma_{1}:\Gamma_N]^{(1-2/p(\pi)+\epsilon')},
\end{align*}
where $p'$ satisfies $p(\pi_{i})\le p'$ for every $i$, by the Spectral
Gap assumption. Using Lemma~\ref{lem:intimidating sums lemma} we
get for $\epsilon'>0$ small enough relative to $\epsilon$,
\begin{align*}
\sum_{x\in\Gamma_1/\Gamma_N}\sum_{i=1}^{T}\n{A_{a}v_{i,x}}_{2}^{2} & \ll_{\epsilon'}[\Gamma_{1}:\Gamma_N]^{-\epsilon(2/p'-\epsilon')}[\Gamma_{1}:\Gamma_N]^{\epsilon'}=o(1),
\end{align*}
as needed.

The proof of the Archimedean high rank case is also similar.
\end{proof}

\section{\label{sec:Bernstein-Theory-of-NonBacktracking}The Bernstein Theory of non-Backtracking Operators}

The results of this section are the main technical contribution of
this work. We restrict to the case of a \emph{non-Archimedean}
field, and we base our work on the results of \citep{bernshtein1974all}.

Let $K'\subset K$ be a compact open subgroup. For $g\in G$
it hold that $\mu(KgK)[K:K']^{-2}\le\mu(K'gK')\le\mu(KgK),$
and therefore $\mu(K'gK')\asymp_{K'}\mu(KgK)\asymp q^{l(g)}$.

Consider the Hecke algebra $H_{K'}=\C_{c}(K'\backslash G/K')$.
For $g\in K'\backslash G/K'$, denote $h_{K',g}=\frac{1}{\mu(K')}K'gK'$,
and let $q_{K',g}=\mu(K'gK')\mu^{-1}(K')$ be the number of right (or left) $K'$ cosets in $K'gK'$. It holds that $q_{K',g}\asymp_{K'}q^{l(g)}$. 
By the representation theory of $p$-adic groups (\citep{cartier1979representations}), given a smooth representation $V$ of $G$, the Hecke algebra $H_{K'}$ acts on the $K'$-fixed vectors $V^{K'}$ of $V$. 

We first discuss the Iwahori-Hecke algebra. Let $I\subset K$ be the Iwahori-Hecke subgroup, i.e. the pointwise stabilizer of a chamber in the Bruhat-Tits building of $G$. Let $W$ be the affine Weyl group of the root system of $G$ (relative to the maximal $k$-torus $T$) and $\hat{W}$ the extended affine Weyl group. By the Iwahori decomposition we have $G=I\tilde{W}I$, where $W\subset\tilde{W}\subset\hat{W}$ is some intermediate subgroup. For $w\in\tilde{W}$, denote $q_{w}=\mu(IwI)/\mu(I)$ which is a natural number. Let $H=C_{c}(I\backslash G/I)$ be the Iwahori-Hecke algebra of $G$ and $h_{w}\in H$ be the element $\frac{1}{\mu(I)}IwI$.

Let $\beta_1,...,\beta_{r}\in\tilde{W}$ ($r=\rank G$) be
some fixed multiples of the simple coweights of the root system of
$W$ (the simple coweight themselves belong to $\hat{W}$, so we cannot
us them directly). Then $h_{\beta_{i}}$ satisfies that $h_{\beta_{i}}^{m}=h_{\beta_{i}^{m}}$,
i.e. $h_{\beta_{i}}$ is a \emph{non-backtracking operator} (when acting
on the building $\mathcal{B}$ of $G$, it is indeed a non-backtracking
operator, or ``collision free'' in the notions of \citep{lubetzky2020random}).
Then it holds that:
\begin{thm}[{See \citep[Theorem 22.1]{kamber2016lpcomplex}}]
\label{thm:Bernstein decomposition for Iwahori-Hecke}There exist
two finite sets $A,B\subset\tilde{W}$ such that each $w\in\tilde{W}$
can be written \emph{uniquely} as $w=a\beta_{1}^{m_{1}}\cdot....\cdot\beta_{r}^{m_{r}}b$
with $a\in A$, $b\in B$, $m_{i}\ge0$, and moreover $l_{\tilde{W}}(w)=l_{\tilde{W}}(a)+\sum_{i=1}^{r}m_{i}l_{\tilde{W}}(\beta_{i})+l_{\tilde{W}}(b)$,
where $l_{\tilde{W}}\colon \tilde{W}\to\N$ is the length function of the group $\tilde{W}$ as (an extended) Coxeter group.

As a corollary, it holds that in the Iwahori-Hecke algebra,
\[
h_{w}=h_{a}h_{\beta_{1}}^{m_{1}}\cdot....\cdot h_{\beta_{r}}^{m_{r}}h_{b}.
\]
\end{thm}

Let us now generalize this theorem to arbitrarily small compact open subgroups $K'\subset G$. The following theorem is based on the results of \citep{bernshtein1974all} (see also \citep[Chapter II, Section 2]{bernstein1992notes}).

\begin{thm}[Bernstein's Decomposition]
\label{thm:Bernstein Decomposition}There exist arbitrarily small
compact open subgroups $K'$ such that for each $K'$ there exist
two finite sets $A,B\subset K'\backslash G/K'$ and $\beta_1,...,\beta_{r}\in K'\backslash G/K'$,
$r=\rank_{k}G$, such that:
\begin{enumerate}
    \item For each $1\le i\le r$ and $m\ge0$, $h_{K',\beta_{i}}^{m}=h_{K',\beta_{i}^{m}}$.
    \item The operators $h_{K',\beta_{i}}$, $1\le i\le r$, commute.
    \item For each $g\in G$ there exist $b\in B$, $a\in A$ and $m_{i}\ge0$
such that 
\begin{align}
h_{K',g} & =h_{K',a}h_{K'\beta_{1}}^{m_{1}}\cdot....\cdot h_{K',\beta_{r}}^{m_{r}}h_{K',b}.\label{eq:Decomposition of an operator}
\end{align}
\end{enumerate}

\end{thm}

\begin{rem}
Equation~\eqref{eq:Decomposition of an operator} is equivalent to the
double coset decomposition 
\[
K'gK'=K'aK'\beta_{1}^{m_{1}}K'\cdot...K'\beta_{r}^{m_{r}}K'bK'=K'a\beta_{1}^{m_{1}}\cdot...\cdot\beta_{r}^{m_{r}}bK'.
\]
\end{rem}

Unlike in Theorem~\ref{thm:Bernstein Decomposition}, there is no
uniqueness in the claim.
\begin{proof}
We follow \citep[Chapter II, Section 2]{bernstein1992notes}. Start
from the Cartan decomposition $KA_{+}K$, where $A_{+}\subset P$
are the dominant elements in the lattice $M/(M\cap K)$
($A_{+}$ is denoted $\Lambda_{+}$ in \citep{bernstein1992notes}).
Let $\beta_1,...,\beta_{r}\in A_{+}$ be generators of a free semigroup
$\tilde{A}_{+}$ in $A_{+}$ (they may be chosen so that their lift
to $M$ commute, not just as elements in $M/(M\cap K)$).
Let $\mu_1,\dots,\mu_{l}\in A_{+}$ be elements such that $A_{+}=\cup_{i=1}^{l}\tilde{A}_+\mu_{i}$,
the union being disjoint.

By Bruhat's Theorem in \citep{bernstein1992notes}, there exist arbitrary
small compact open subgroups $K'\subset K$ such that:
\begin{itemize}
\item $K'$ is normal in $K$.
\item For every $a,b\in A_{+}$ it holds that $h_{K',a}h_{K',b}=h_{K',ab}$,
i.e. 
\begin{equation}
K'aK'bK'=K'abK'.\label{eq:double coset for A+}
\end{equation}
\end{itemize}
Let $x_1,...,x_{m}\in K$ be representatives of right cosets of
$K'$ in $K$. Since $K'$ is normal in $K$ they are also representatives of left cosets. Let $A=\left\{ x_1,...,x_{r}\right\} $ and let
$B=\left\{ \mu_{i}x_{j}:i=1,\dots,l,j=1,\dots,r\right\} $. By the Cartan decomposition for each $g\in G$ there exist $x\in A$, $\mu x'\in B$
and $m_1,...,m_{r}$ such that $K'gK'=K'x\beta_{1}^{m_{1}}\cdot...\cdot\beta_{r}^{m_{r}}\mu x'K'$.
It remains to show that 
\begin{equation}
K'xK'\beta_{1}^{m_{1}}K'\cdot...\cdot\beta_{r}^{m_{r}}K'\mu x'K'=K'x\beta_{1}^{m_{1}}\cdot...\cdot\beta_{r}^{m_{r}}\mu x'K.\label{eq:proof of bernstein decomposition}
\end{equation}
Since $\beta_1,...,\beta_{m},\mu\in A_{+}$, $K'\beta_{1}^{m_{1}}K'\cdot...\cdot\beta_{r}^{m_{r}}K'\mu K'=K'\beta_{1}^{m_{1}}\cdot...\cdot\beta_{r}^{m_{r}}\mu K'$
by Equation~\eqref{eq:double coset for A+}. Finally, since $K'$ is
normal in $K$ and $x,x'\in K$, $K'x=xK'$ and $K'x'=x'K'$. Applying
those equalities we get Equation~\eqref{eq:proof of bernstein decomposition}.
\end{proof}

\subsection{Non-Backtracking Operators and Temperedness}

We may now deduce:
\begin{thm}
\label{thm:Temperedness by Riemann}Let $(\pi,V)$ be a unitary irreducible representation of $G$, let $V^{K'}$ the $K'$-fixed vectors and assume that $V^{K'}\ne\left\{ 0\right\}$. Consider the action of $H_{K'}=C_{c}(K'\backslash G/K')$ on $V^{K'}$. Then $\pi$ is $p$-tempered if and only if, for every $1\le i\le r$, for every eigenvalue $\lambda$ of $\pi(h_{K',\beta_{i}})$ on $V^{K'}$
it holds that $\left|\lambda\right|\le q_{K',\beta_{i}}^{1-\nicefrac{1}{p}}$.
\end{thm}

\begin{proof}
First, note that $V$ is $p$-tempered if and only if for every $0\ne v_{0}\in V^{K'}$
and $p'>p$,
\begin{align*}
\intop_{G}\left|\left\langle v_{0},\pi(g)\cdot v_{0}\right\rangle \right|^{p'}dg & =\sum_{g\in[K'\backslash G/K']}\mu(K'gK')\left(\frac{\left|\left\langle v_{0},\pi(K'gK')v_{0}\right\rangle \right|}{\mu(K'gK')}\right)^{p'}\\
 & =\sum_{g\in[K'\backslash G/K']}(\mu(K'gK'))^{1-p'}(\mu(K')\left|\left\langle v_{0},\pi(h_{K',g})v_{0}\right\rangle \right|)^{p'}\\
 & \asymp\sum_{g\in[K'\backslash G/K']}(q_{K',g})^{1-p'}(\left|\left\langle v_{0},\pi(h_{K',g})v_{0}\right\rangle \right|)^{p'}<\infty.
\end{align*}

The only if part is easier and does not require Bernstein's decomposition.
For $p=2$ it is essentially the main result of \citep{lubetzky2020random}.

Assume that some eigenvalue $\lambda$ of $h_{K',\beta_{i}}$ satisfies
$\left|\lambda\right|>q_{K',\beta_{i}}^{1-\nicefrac{1}{p}}$. Let
$v_{0}\in V^{K}$ be an eigenvector of $h_{K',\beta_{i}}$ with eigenvalue
$\lambda$. Then, for $p'>p$ such that $\left|\lambda\right|\ge q_{K',\beta_{i}}^{1-\nicefrac{1}{p'}}$,
\begin{align*}
\intop_{G}\left|\left\langle v_{0},\pi(g)\cdot v_{0}\right\rangle \right|^{p'}dg & \gg \sum_{g\in [K'\backslash G/K']}(q_{K',g})^{1-p}(\left|\left\langle v_{0},\pi(h_{K',g})v_{0}\right\rangle \right|)^{p}\\
& \ge\sum_{m=0}^{\infty}(q_{K',\beta_{i}^{m}})^{1-p}(\left|\left\langle v_{0},\pi(h_{K',\beta_{i}^{m}})v_{0}\right\rangle \right|)^{p}\\
 & =\sum_{m=0}^{\infty}(q_{K',\beta_{i}})^{m(1-p)}(\left|\left\langle v_{0},\pi(h_{K',\beta_{i}})^{m}v_{0}\right\rangle \right|)^{p}\\
 & =\sum_{m=0}^{\infty}(q_{K',\beta_{i}})^{m(1-p)}(\left|\lambda\right|^{m}\n{v_{0}}^{2})^{p}\\
 & \ge\n{v_{0}}^{2p}\sum_{m=0}^{\infty}(q_{K',\beta_{i}})^{m((1-p)+p(1-\nicefrac{1}{p}))}=\n{v_{0}}^{2p}\sum_{m=0}^{\infty}1=\infty,
\end{align*}
and $V$ is not $p$-tempered.

We now prove the if part. One should prove that for $p'>p$, the matrix
coefficients are in $L^{p'}(G)$. By Bernstein decomposition,
\begin{align*}
\intop_{G}\left|\left\langle v_{0},\pi(g)\cdot v_{0}\right\rangle \right|^{p'}dg  \asymp&\sum_{g\in[K'\backslash G/K']}(q_{K',g})^{1-p'}(\left|\left\langle v_{0},\pi(h_{K',g})v_{0}\right\rangle \right|)^{p'}\\
 \le& \sum_{a\in A,b\in B}\sum_{m_{1}\ge0,...,m_{r}\ge0}(q_{K',a}q_{K',\beta_{1}}^{m_{1}}\cdot...\cdot q_{K',\beta_{r}}^{m_{r}}q_{K',b})^{1-p'}\\
 & \cdot(\left|\left\langle \pi(h_{K',a})^{\ast}v_{0},\pi(h_{K',\beta_{1}}^{m_{1}}\cdot....\cdot h_{K',\beta_{r}}^{m_{r}})\pi(h_{K',b})v_{0}\right\rangle \right|)^{p'}\\
 \ll_{K'}& \sum_{a\in A,b\in B}(q_{K',a}q_{K',b})^{1-p'}\n{\pi(h_{K',a})}^{p'}\n{\pi(h_{K',b})}^{p'} \\
 & \cdot \sum_{m_{1}\ge0,...,m_{r}\ge0}(q_{K',\beta_{1}}^{m_{1}}\cdot...\cdot q_{K',\beta_{r}}^{m_{r}})^{1-p'} \cdot\n{\pi(h_{K',\beta_{1}}^{m_{1}}\cdot...\cdot h_{K',\beta_{r}}^{m_{r}})}^{p'}\\
 \ll_{K'}& (\sum_{m_{1}\ge0}q_{K',\beta_{1}}^{m_{1}(1-p')}\n{\pi(h_{K',\beta_{1}}^{m_{1}})}^{p'})(\sum_{m_{2}\ge0}q_{K',\beta_{2}}^{m_{2}(1-p')}\n{\pi(h_{K',\beta_{2}}^{m_{2}})}^{p'})\\
 & \cdot...\cdot(\sum_{m_{2}\ge0}q_{K',\beta_{2}}^{m_{r}(1-p')}\n{\pi(h_{K',\beta_{r}}^{m_{r}})}^{p'}).
\end{align*}

Since all the eigenvalues of $\pi(h_{K',\beta_{i}})$ are
bounded by $q_{K',\beta_{i}}^{1-\nicefrac{1}{p}}$, it follows from
the theory of matrix norms that the sum converges (note that $\pi(h_{K',\beta_{1}})$
is usually not unitary or self-adjoint, so one should be a bit careful
here).
\end{proof}
For $d_{0}\in\R_{\ge0}$, choose $m_{i}$,
so that $q_{\beta_{i}}^{m_{i}-1}\le q^{d_{0}/2}\le q_{\beta_{i}}^{m_i}$ and denote $f_{d_0}=\sum_{i=1}^{r}(h_{\beta_{i}}^{m_{i}})^{\ast}h_{\beta_{i}}^{m_{i}}$, where $f^*(g) = \overline{f(g^{-1})}$ for a function $f\colon G\to \C$.
The following theorem proves Theorem~\ref{thm:good family padic intro}
\begin{thm}
\label{thm:Good family p-adic}The functions $f_{d_0}\in H_{K'}$
are a good family for the set $\Pi(G)_{K'-\sph}$:
\begin{enumerate}
\item For every $(\pi,V)\in\Pi(G)$ with a non-trivial
$K'$-invariant vector, it holds that 
\[
q^{d_{0}(1-\nicefrac{1}{p(\pi)})}\ll_{K'}\tr (\pi(f_{d_0})).
\]
\item For every $(\pi,V)\in\Pi(G)$ without a non-trivial
$K'$-invariant vector, it holds that 
\[
\tr (\pi(f_{d_0}))=0.
\]
\item It holds that $f_{d_0}(g)\ll_{\epsilon}q^{d_{0}\epsilon}\psi_{d_0}(g)$.
\end{enumerate}
\end{thm}

\begin{proof}
The second condition is obvious, since if there is no non-trivial $K'$-invariant
vector then $\pi(f_{d_0})=0$.

The fact that $f_{d_0}(g)\ll_{\epsilon}q^{d_{0}\epsilon}\psi_{d_0}(g)$
follows from Lemma~\ref{lem:convolution lemma}, as $h_{\beta_{i}}^{m_{i}}(g)\ll\chi_{d_{0}/2}$.

It remains to prove the first condition. Since $\pi(f_{d_0})$
is non-negative and self-adjoint it is diagonalizable on $V^{K'}$and
all its eigenvalues are non-negative.

On the other hand $\tr \pi(f_{d_0})\ge\n{\pi(f_{d_0})}$ (as a matter of fact, by uniform admissibility (\citep{bernshtein1974all}), $\dim V^{K'}\ll_{K'}1$, so $\tr \pi(f_{d_0})\asymp \n{\pi(f_{d_0})}$.
It holds that
\[
\max_{i=1}^{r}\n{h_{\beta_{i}}^{m_{i}}}^{2}\le\n{\pi(f_{d_0})}.
\]

Since $V$ is not $p'$-tempered for $p'<p$, by Theorem~\ref{thm:Temperedness by Riemann},
for some $1\le i\le r$ we have an eigenvalue $\lambda$ of $h_{\beta_{i}}$
with $\left|\lambda\right|\ge q_{K',\beta_{i}}^{1-\nicefrac{1}{p}}$.
Therefore, $h_{\beta_{i}}^{m_{i}}$ has an eigenvalue $\lambda^{m_{i}}$
with $\left|\lambda^{m_{i}}\right|\ge q_{K',\beta_{i}}^{m_{i}(1-\nicefrac{1}{p})}\asymp_{K'}q^{l(\beta_{i})m_{i}(1-\nicefrac{1}{p})}\asymp q^{\nicefrac{d_0}{2}(1-\nicefrac{1}{p})}$.
Therefore, $\n{\pi(f_{d_0})}\ge\n{h_{\beta_{i}}^{m_{i}}}^{2}\gg q^{d_0(1-\nicefrac{1}{p})}$,
as needed.
\end{proof}

\section{\label{sec:pointwise lower bound}Lower Bounds on Matrix Coefficients
for a Specific Representation}

In this section, we prove Theorem~\ref{thm: good family pointwise intro}.
We assume that $G$ is Archimedean and $(\pi,V)$ is an
irreducible unitary representation of $G$.

We will prove that $\left\{ \pi\right\} $ has a good family of functions.
Let $v\in V^{\tau}$, $\n v=1$, belong to a fixed $\tau$-type of
$K$, i.e. $K$ acts on $U_{0}=\spann\left\{ Kv\right\} $ as
a $K$-irreducible representation $\tau$. In particular, $v$ is
$K$-finite. Let

\[
g_{d_{0}/2}(g)=\chi_{d_{0}/2}(g)q^{l(g)p(\pi)^{-1}}\overline{\left\langle v,\pi(g)v\right\rangle },
\]
and let 
\[
f_{d_0}=g_{d_{0}/2}^{*}\ast g_{d_{0}/2}.
\]

We claim that $f_{d_0}$ satisfies the condition of Definition~\ref{def:Good family} for $\{\pi\}$. We first verify conditions $(2)$ and $(3)$.
Since $f_{d_0}$ is self-adjoint and non-negative, $(2)$ follows.

By the theory of leading exponents (which is described below), $\left|\left\langle v,\pi(g)v\right\rangle \right|\ll_{\pi,v,\epsilon}q^{-l(g)(\nicefrac{1}{p(\pi)}+\epsilon)}$,
so $\left|g_{d_{0}/2}(g)\right|\ll q^{d_{0}\epsilon}\chi_{d_{0}/2}(g)$.
It follows that also $\left|g_{d_{0}/2}^{*}(g)\right|\ll q^{d_{0}\epsilon}\chi_{d_{0}/2}(g)$.
By Lemma~\ref{lem:convolution lemma}, $\left|f_{d_0}(g)\right|\ll_{\epsilon}q^{d_{0}\epsilon}\psi_{d_0+2}(g)$.

It remains to prove $(1)$, which will concern the rest of this section. We note that since $f_{d_0}$ is self-adjoint and non-negative  
\[
\tr (\pi(f_{d_0}))\gg\n{\pi(f_{d_0})}=\n{\pi(g_{d_{0}/2})}^{2}.
\]

Now, 
\begin{align*}
\left\langle v,\pi(g_{d_{0}/2})v\right\rangle  & =\intop_{G}g_{d_{0}/2}(g)\left\langle v,\pi(g)v\right\rangle \\
 & =\intop_{G}\chi_{d_{0}/2}(g)q^{l(g)p(\pi)^{-1}}\overline{\left\langle v,\pi(g)v\right\rangle }\left\langle v,\pi(g)v\right\rangle dg\\
 & \ge\intop_{l(d)\le d_{0}/2}q^{l(g)p(\pi)^{-1}}\left|\left\langle v,\pi(g)v\right\rangle \right|^{2}dg,
\end{align*}
and conclude that
\[
\sqrt{\tr (\pi(f_{d_0}))}\gg\intop_{l(g)\le d_{0}/2}q^{l(g)p(\pi)^{-1}}\left|\left\langle v,\pi(g)v\right\rangle \right|^{2}dg.
\]

Therefore, to prove that $\tr (\pi(f_{d_0}))\gg_\epsilon q^{d_{0}(1-\nicefrac{1}{p(\pi)}-\epsilon)}$,
one should prove (after changing $d_{0}$ and $d_{0}/2$)
\begin{equation}
\intop_{l(g)\le d_{0}}q^{l(g)p(\pi)^{-1}}\left|\left\langle v,\pi(g)v\right\rangle \right|^{2}dg\stackrel{!}{\gg}_{\pi,\epsilon}q^{d_{0}(1-p(\pi)^{-1}-\epsilon)}.\label{eq:lower bound on matrix coeff}
\end{equation}

So far, our proof is essentially the same as (part of) the proof of
\citep[Theorem 3]{sarnak1991bounds}, which only concerns rank $1$.

We start by simplifying the left-hand side of equation~\eqref{eq:lower bound on matrix coeff}.

Applying the Cartan decomposition we get
\[
\intop_{l(g)\le d_{0}}q^{l(g)p(\pi)^{-1}}\left|\left\langle v,\pi(g)v\right\rangle \right|^{2}dg=\intop_{K}\intop_{K}\intop_{a\in A_{+},l(a)\le d_{0}}q^{l(a)p(\pi)^{-1}}S(a)\left|\left\langle v,\pi(kak')v\right\rangle \right|^{2}dkdk'da.
\]
Using the logarithm map, we identify $A_{+}\subset\a$, and give $a\in A_+$ coordinates $(x_1,...,x_{r})$, $x_{i}\ge0$ by $(x_1,...,x_{r})\to\sum_{i=1}^{r}x_{i}\omega_{i}$,
where $\omega_1,...,\omega_{r}$ are the fundamental coweights. Recall
that for $a\in A_{+}$, $l(a)=2\rho(a)$. For $\kappa>0$,
denote by $A_{+}^{\kappa}\subset A_{+}$ the set of $a\in A_{+}$
with $x_{i}>\kappa$). Then for $a\in A_{+}^{\kappa}$ it holds that
$S(a)\asymp_{\kappa}q^{l(a)}=q^{2\rho(a)}$. Then: 
\begin{align*}
 & \intop_{K}\intop_{K}\intop_{a\in A_{+},l(a)\le d_{0}}q^{l(a)p(\pi)^{-1}}S(a)\left|\left\langle v,\pi(kak')v\right\rangle \right|^{2}dkdk'da\\
 & \gg_{\kappa}\intop_{K}\intop_{K}\intop_{a\in A_{+}^{\kappa},l(a)\le d_{0}}q^{2\rho(a)(1+p(\pi)^{-1})}\left|\left\langle v,\pi(kak')v\right\rangle \right|^{2}dkdk'da,
\end{align*}
so we should prove that for some $\kappa>0$,
\begin{equation}
\intop_{K}\intop_{K}\intop_{a\in A_{+}^{\kappa},l(a)\le d_{0}}q^{2\rho(a)(1+p(\pi)^{-1})}\left|\left\langle v,\pi(kak')v\right\rangle \right|^{2}dkdk'da\stackrel{!}{\gg}_{\epsilon}q^{d_{0}(1-p(\pi)^{-1}-\epsilon)}.\label{eq:Simplified lower bound}
\end{equation}

\subsection{\label{subsec:Leading-Exponents}Leading Exponents}

We recall the Casselman-Harish Chandra-Milicic theory of \emph{Leading
Exponents}. We follow \citep[Chapter VIII]{knapp2016representation}.

Equation~\eqref{eq:Simplified lower bound} is very similar to \citep[Theorem 8.48,(b) implies (a)]{knapp2016representation},
which is one of the more technical parts of the theory. We will follow
the same proof closely, while deriving an explicit expression. A side
benefit is that the proofs below somewhat simplify the proof of \citep[Theorem 8.48,(b) implies (a)]{knapp2016representation}.

By the theory of leading exponents, we may associate with an irreducible
unitary representation (actually, to any irreducible admissible representation)
$(\pi,V)$ a finite subset called leading exponent $\mathcal{F}\subset \a_{\C}^{*}$ such that the following two theorems hold:
\begin{thm}[{\citep[Theorem 8.47]{knapp2016representation}}]
\label{thm:lead_coeff_upper_bound}The following are equivalent:
\begin{itemize}
\item For $\nu_{0}\colon a\to\R$ a real character, every $K$-finite matrix coefficient
$\left\langle v,av\right\rangle $ for $a\in A_{+}$ is bounded in
absolute value by $\ll_{\pi,v}e^{(\nu_{0}-\rho)(a)}l(a)^N$
where $N$ is some constant.
\item For every $\nu\in\mathcal{F}$, and every fundamental weight $\omega_{i}$,
$1\le i\le r$, $\re\nu(\omega_{i})\le\nu_{0}(\omega_{i})$.
\end{itemize}
\end{thm}

\begin{thm}[{\citep[Theorem 8.48]{knapp2016representation}}]
\label{thm:lead_coeff_Lp}The following are equivalent:
\begin{itemize}
\item Every $K$-finite matrix coefficient $\phi(g)=\left\langle v,\pi(g)v\right\rangle $
is in $L^{p}(G)$.
\item For every $\nu\in\mathcal{F}$, and every fundamental weight $\omega_{i}$,
$1\le i\le r$, $\re\nu(\omega_{i})<(1-\frac{2}{p})\rho(\omega_{i})$.
\end{itemize}
\end{thm}

Note that the second theorem implies that 
\[
p(\pi)=\min\left\{ p:\forall\nu\in\mathcal{F},1\le i\le r,\,\re\nu(\omega_{i})\le(1-\frac{2}{p})\rho(\omega_{i})\right\},
\]
and that some matrix coefficient is not in $L^{p(\pi)}(g)$.

To state the main theorem about leading exponents, let us set some
notations. Assume that $0\ne v\in V$ is $K$-finite, and let $E_{0}\colon V\to U_{0}$
be a projection onto a finite dimensional $K$-invariant subspace
$U_{0}\subset V$ such that $v\in U_{0}$. We define $F\colon A_+\to\End_{\C}(U_{0})$
by $F(a)=E_{0}\pi(a)E_{0}$.

We denote by $\mathcal{H}_{\End U_{0}}$ the set of holomorphic
functions $f\colon D^{r}\to\End (U_{0})$, where $D=\left\{ z\in\C:\left|z\right|<1\right\} $
is the open unit ball. Each such function has a convergent multiple
power series, which is absolutely and uniformly convergent on compact
subsets of $D^{r}$.

As before, we identify $a\in A_{+}\subset\a$ with coordinates $(x_1,...,x_{r})$,
$x_{i}\ge0$ by $(x_1,...,x_{r})\to\sum_{i=1}^{r}x_{i}\omega_{i}$.

We say that $\nu,\nu'\in\a_{\C}^{*}$ are integrally equivalent
if their difference $\nu-\nu'$ in an integral combination of simple
roots. If the difference is a \emph{non-negative} integral combination of simple roots, we write $\nu'\le\nu$.
\begin{thm}[{\citep[Theorem 8.32]{knapp2016representation}}]
\label{thm:Leading exponents expansion}There exist $n_{0}\in\N$ and a finite set $\mathcal{F}'$ satisfying:
\begin{enumerate}
    \item $\mathcal{F}\subset\mathcal{F}'$.
    \item Each $\nu'\in\mathcal{F}'$ satisfies $\nu'\le\nu$ for some $\nu\in\mathcal{F}$.
    \item It holds that for $x_{1}>0,\dots,x_{r}>0$, 
\begin{equation}
F(a)=F(x_1,...,x_{r})=\sum_{\nu\in\mathcal{F}'}\sum_{1\le n_{1}\le n_0\dots1\le n_{r}\le n_0}G_{\nu,n_1,\dots,n_{r}}(x_1,...,x_{r})e^{(\nu-\rho)(x_1,...,x_{r})}x_{1}^{n_{1}}\cdot\dots\cdot x_{r}^{n_{r}},\label{eq:Leading coefficient expression}
\end{equation}
such that for $\nu\in\mathcal{F}'$, $1\le n_{i}\le n_{0}$, $G_{\nu,n_1,...,n_{r}}\colon (0,\infty)^{r}\to\End_{\C}(U_{0})$
are functions given by $G_{\nu,n_1,...,n_{r}}(x_1,...,x_{r})=f_{\nu,n_1,...,n_{r}}(e^{-x_{1}},...,e^{-x_{r}})$,
where $f_{\nu,n_1,...,n_{1}}\in\mathcal{H}_{\End U_{0}}$.

Moreover, if $f_{\nu,n_1,...,n_{r}}\ne0$ then $f_{\nu,n_1,...,n_{r}}(0,\dots,0)\ne0$
and for each $\nu\in\mathcal{F}'$ there exist $n_1,\dots,n_{r}$
with $f_{\nu,n_1,...,n_{r}}\ne0$.

\end{enumerate}
\end{thm}

\begin{proof}
The theorem follows from \citep[Theorem 8.32]{knapp2016representation}
and the discussion following it. Let us explain: by \citep[Theorem 8.32]{knapp2016representation},
$F$ has a decomposition like Equation~\eqref{eq:Leading coefficient expression}
for a certain set $\mathcal{F}'$. If $f_{\nu,n_1,...,n_{r}}(0,...,0)= 0$, we may use its power series expansion to replace $\nu$ by other elements in $\a_{\C}^{*}$, which are integrally equivalent to it. 

It remains to prove that $\mathcal{F}\subset\mathcal{F'}$, and that each $\nu' \in \mathcal{F}'$ has $\nu \in \mathcal{F}$ with $\nu'\le \nu$.
By power series expansion, we have a \emph{unique}
decomposition (see \citep[Equation 8.52]{knapp2016representation})
\begin{align*}
F(x_1,...,x_{r}) & =\sum_{\nu\in\a_{\C}^{*}}F_{\nu-\rho}(x_1,\dots,x_{r})\\
F_{\nu-\rho}(x_1,\dots,x_{r}) & =\sum_{1\le n_{1}\le n_0,\dots,1\le n_{r}\le n_0}c_{\nu,n_1,...,n_r}e^{(\nu-\rho)(x_1,...,x_{r})}x_{1}^{n_{1}}\cdot\dots\cdot x_{r}^{n_{r}},
\end{align*}
for some $c_{\nu,n_1,...,n_r} \in \End U_0$. Each term $F_{\nu-\rho}$ can be calculated from $G_{\nu',n_1,\dots,n_{r}}$, for $\nu\le \nu'$.
The set of leading exponent is the set of maximal elements relative to $\le$ for $\nu\in\mathcal{F}$ with $F_{\nu-\rho}\ne0$. This immediately implies that $\mathcal{F}\subset\mathcal{F}'$. Moreover, each $\nu''\in\a_{\C}^{*}$ with $F_{\nu''-\rho}\ne0$ satisfies $\nu''\le\nu$ for some $\nu\in\mathcal{F}$,
which says that each $\nu'\in\mathcal{F}'$ satisfies $\nu'\le\nu$
for some $\nu\in\mathcal{F}$.
\end{proof}
We remark that Theorem~\ref{thm:Leading exponents expansion} does
not directly imply the upper bound given in Theorem~\ref{thm:lead_coeff_upper_bound},
since it does not give bounds for $x_{i}\to0$. Such bounds are available using asymptotic expansion near the walls (\citep[Chapter VIII, Section 12]{knapp2016representation}).

\subsection{Some Technical Lemmas}
\begin{lem}[{Compare \citep[Lemma B.24]{knapp2016representation}}]
\label{lem:Terrible technical lemma}Let $f\colon \R\to\C$ be a function
defined as $f(x)=e^{\beta x}\sum_{i=1}^{k}c_{i}e^{-\alpha_{i}x}x^{n_{i}}$
with $\alpha_{i},c_{i}\in\C$, $\re(\alpha_{i})\ge0$,
$\beta\in\R$, $\beta\ge0$, $n_{i}\in\N$. 

Assume that there is $0\le i\le k$
such that $\re(\alpha_{i})=0,\,c_{i}\ne0$ and let
\[
n_{0}=\max_{1\le i\le k}\left\{ n_{i}:\re(\alpha_{i})=0\text{ and }c_{i}\ne0\right\} .
\]
Then for $T$ large enough, 
\[
\intop_{0}^{T}\left|f(x)\right|dx\gg_{f}e^{\beta T}T^{n_{0}}.
\]
Moreover, if we assume that $n_{0}=\max_{1\le i\le k}\left\{ n_{i}:\re(\alpha_{i})=0\right\} $
then the underlying lower bound on $T$ and the constants are continuous
for small perturbations of the $c_{i}$.
\end{lem}

\begin{rem}
The condition on $n_0$ in the "moreover part" comes to deal with the case that after a small perturbation, some $c_i=0$ becomes non-zero. 
\end{rem}

\begin{proof}
During the proof, $\gg$ may depend on $f$. Fix $M$ large enough,
depending on $f$, to be chosen later. Then 
\[
\intop_{0}^{T}\left|f(x)\right|dx\ge\intop_{T-M}^{T}\left|f(x)\right|dx.
\]

Let $\alpha_{0}=\min_{1\le i\le k}\left\{ \re(\alpha_{i}):\re(\alpha_{i})\ne0\right\} $,
$N=\max_{1\le i\le k}\left\{ n_{i}\right\} $.

After re-arranging the summands, write $f=f_{0}+f_{1}$, with $f_{0}(x)=e^{\beta x}\sum_{i=1}^{l}c_{i}e^{-\alpha_{i}x}x^{n_{0}}$,
$f_{1}=e^{\beta x}\sum_{i=l+1}^{k}c_{i}e^{-\alpha_{i}x}x^{n_{i}}$,
where the summands $1\le i\le l$ contain all factors with $\re(\alpha_{i})=0$ and $n_{i}=n_{0}$. Moreover, we may assume that all the $\alpha_i$, $1\le i \le l$ are different.

Then for $\epsilon>0$ small enough and $T$ large enough, 
\[
\intop_{T-M}^{T}\left|f_{1}(x)\right|dx\le M\max_{T-M\le x\le T}\left\{ f_{1}(x)\right\} \ll((e^{(\beta-\alpha_0)(T-M)}+e^{(\beta-\alpha_0)T})T^N+e^{\beta T}T^{n_{0}-1})=o(e^{\beta T}T^{n_{0}}).
\]
Now, 
\begin{align*}
\intop_{T-M}^{T}\left|f_{0}(x)\right|dx & =\intop_{T-M}^{T}x^{n_{0}}e^{\beta x}\left|\sum_{i=1}^{l}c_{i}e^{-\alpha_{i}x}\right|dx\\
 & \ge(T-M)^{n_{0}}e^{\beta(T-M)}\intop_{T-M}^{T}\left|\sum_{i=1}^{l}c_{i}e^{-\alpha_{i}x}\right|dx
\end{align*}
Note that $\left|\sum_{i=1}^{l}c_{i}e^{\alpha_{i}x}\right|\gg\left|\sum_{i=1}^{l}c_{i}e^{-\alpha_{i}x}\right|^{2}=\sum_{i=1}^{l}\left|c_{i}\right|^{2}+\sum_{1\le i\ne j\le l}c_{i}\overline{c_{j}}e^{(\alpha_{i}-\alpha_{j})x}$,
since this value is bounded. 
\begin{align*}
\intop_{T-M}^{T}\left|\sum_{i=1}^{l}c_{i}e^{\alpha_{i}x}\right|dx & \gg\intop_{T-M}^{T}(\sum_{i=1}^{l}\left|c_{i}\right|^{2}+\sum_{1\le i\ne j\le l}c_{i}\overline{c_{j}}e^{(\alpha_{i}-\alpha_{j})x})dx\\
 & \ge M(\sum_{i=1}^{l}\left|c_{i}\right|^{2})-\sum_{1\le i\ne j\le l}\left|\frac{c_{i}\overline{c_{j}}}{\alpha_{i}-\alpha_{j}}\right|\\
 & \gg M-O(1)
\end{align*}
For $M$ large enough the last value is $\gg1$, so 
\[
\intop_{T-M}^{T}\left|f_{0}(x)\right|dx\gg e^{\beta T}T^{n_{0}}
\]
and 
\[
\intop_{0}^{T}\left|f(x)\right|dx\ge\intop_{T-M}^{T}\left|f(x)\right|dx\ge\intop_{T-M}^{T}\left|f_{0}(x)\right|dx-\intop_{T-M}^{T}\left|f_{1}(x)\right|dx\gg e^{\beta T}T^{n_{0}}.
\]
For the "moreover part" one follows the proof carefully and notices that it remains true for small perturbations in the $c_i$. 
\end{proof}
\begin{rem}
For $\beta>0$ our lower bound agrees with a similar upper bound. For
$\beta=0$ it is no longer true, but a similar proof will give the
right lower bound $T^{n_{0}+1}$.
\end{rem}

\begin{lem}
\label{lem:terrible technical lemma2}Let $M$ be an open subset of
a smooth Riemannian manifold (e.g., a Lie group), $F:M\times[R,\infty)$ a function defined by 
\begin{equation}
F(m,x)=\sum_{i=1}^{k}e^{s_{i}x}x^{n_{i}}F_{i}(m,x),\label{eq:terrible lemma 2}
\end{equation}
such that: $s_{i}\in\C$, $n_{i}\in\N$, and $F_{i}(m,x)=f_{i}(m,e^{-x})$
for some function $f_{i}(m,z)$ real analytic on $M\times\overline{D}_{e^{-R}}$,
where $\overline{D}_{r}\subset\C$ is the closed ball of radius $r$,
and holomorphic in the second variable. Assume also that for each
$1\le i\le k$ there is $m\in M$ with $f_{i}(m,0)\ne0$.
Let $s_{0}=\max_{1\le i\le k}\re s_i$ and assume that $s_{0}\ge0$.

Then for $T$ large enough
\[
\intop_{M}\intop_{R}^{T}\left|F(s,r)\right|dr\gg e^{s_{0}T}.
\]
\end{lem}

\begin{proof}
By restriction to a compact subset, we may assume that the closure of $M$ is compact. 

Decompose $f_{i}(m,z)=c(m)+g_{i}(m,z)z$,
where $g_{i}(m,z)$ in also holomorphic in $\left|z\right|\le e^{-R}$.
Then 
\[
F_{i}(m,r)=c_{i}(m)+e^{-r}G_{i}(m,r),
\]
where $G_{i}(m,r)$ is bounded for $m,r\in M\times[R,\infty)$.
Therefore, 
\begin{equation}
F(m,r)=\sum_{i=1}^{k}e^{s_{i}r}r^{n_{i}}c_{i}(m)+\sum_{i=1}^{k}e^{(s_{i}-1)r}r^{n_{i}}G_{i}(m,r).\label{eq:terrible lemma1}
\end{equation}

Without loss of generality, $\re(s_{1})=s_{0}$ and
$n_{1}=\max\left\{ n_{i}:\re(s_{i})=s_{0}\right\} $.
Let $m_{0}\in M$ be a point with $f_{1}(m_{0},0)\ne0$.
Choose a small enough neighborhood $M_{0}\subset M$ of $m_{0}$.
We have 
\[
\intop_{M}\intop_{R}^{T}\left|F(m,r)\right| dr dm 
\gg\intop_{M_{0}}\intop_{R}^{T} \left|F(m,r)\right| dr dm.
\]
Since $G_{i}(m,r)$ is bounded on $M_0$, for $T$ large enough the
second summand of Equation~\eqref{eq:terrible lemma1} satisfies
\[
\intop_{M_{0}}\intop_{R}^{T}\left|\sum_{i=1}^{k}e^{(s_{i}-1)r}r^{n_{i}}G_{i}(m,r)\right|dr=o(e^{s_{0}T}).
\]
As for the first summand of Equation~\eqref{eq:terrible lemma1}, by
Lemma~\ref{lem:Terrible technical lemma} and
the fact that $c_{1}(m_{0})\ne0$, it holds that 
\[
\intop_{M_{0}}\intop_{R}^{T}\left|\sum_{i=1}^{k}e^{s_{i}r}r^{n_{i}}c_{i}(m)\right|dr\gg T^{n_{1}}e^{s_{0}T}
\]
and we are done.
\end{proof}
We can finally prove Equation~\eqref{eq:Simplified lower bound}.
\begin{proof}
Recall that we chose $v\in V$, $\n v=1$ to span a representation
$\tau$ of $K$. We choose in Theorem~\ref{thm:Leading exponents expansion}
$U_{0}=\spann Kv$.

Using Theorem~\ref{thm:lead_coeff_Lp}, choose $\nu_{0}\in\mathcal{F}'$
and $1\le i\le r$ such that $\re\nu_{0}(\omega_{i})=(1-\frac{2}{p(\pi)})\rho(\omega_{i})$.
Without loss of generality $i=1$. Moreover, we assume that among all
$\nu\in\mathcal{F}'$ satisfying this condition, $\nu_{0}$ has maximal
$0\le N_{1}\le n_{0}$ such that for some constants $z_{2},...,z_{l}\ne0$,
$0\le n_{2},...,n_{r}\le n_{0}$,$\lim_{z_{1}\to0}f_{\nu,N_1,...,n_{k}}(z_1,z_{2},\dots,z_{r})\ne0$,
where $f_{\nu,N_1,...,n_{k}}(z_1,z_{2},\dots,z_{r})$
is taken from Theorem~\ref{thm:Leading exponents expansion} (notice that $\nu_{0}$ may belong to $\mathcal{F}'-\mathcal{F}$), so it may not be a \emph{leading exponent}).

By Theorem~\ref{thm:Leading exponents expansion}, if we let $M=K\times(0,\infty)^{r-1}\times K$,
we identify $m\in M$ with $m=(k_1,x_{2},...,x_{r},k_{2})$,
and let $G:M\times(0,\infty)\to\infty$ be 
\[
G(m,x_{1})=\left\langle v,\pi(k_{1}a(x_1,....,x_{r})k_{2})v\right\rangle ,
\]
then $G(m,x_{1})$ has the form of Equation~\eqref{eq:terrible lemma 2},
with $s_{1}=(\nu_{0}-\rho)(\omega_{1})=-\frac{2}{p(\pi)}\rho(\omega_{1})$,
$n_{1}=N_{1}$. Note that $G^{2}(m,x_{1})$ also has this
form, with $s_{1}=-\frac{4}{p(\pi)}\rho(\omega_{1})$.
We Let 
\[
F(m,x_{1})=e^{2\rho(\omega_{1})(1+p(\pi)^{-1})x_{1}}G^{2}(m,x_{1}),
\]
and $F$ also has a similar form, with $s_{1}=2\rho(\omega_{1})(1-p(\pi)^{-1})$.
Let $m_{0}=(k_1,x_{2},...,x_{r},k_{2})$ be a point
where the condition of Lemma~\ref{lem:terrible technical lemma2}
holds. By Lemma~\ref{lem:terrible technical lemma2}, for a small
neighborhood $M_{0}$ of $m_{0}$, it holds for $d_{0}$ large enough
and some constant $C$
\[
\intop_{M_{0}}\intop_{1}^{(d_{0}-C)/2\rho(\omega_{1})}\left|F(m,x_{1})\right|dx_{1}dm\gg e^{d_{0}(1-p(\pi)^{-1})}
\]
Finally, for $M_{0},\kappa$ small enough, for each $m=(k_1,x_{2},...,x_{r},,k_{2})\in M_{0}$
and $1\le x_{1}\le(d_{0}-C)/2\rho(\omega_{1})$,
it holds that $a=(x_1,...,x_{r})\in A_{+}^{\kappa}$
and $l(a)\le d_{0}$. Therefore, 
\begin{align*}
&\intop_{K}\intop_{K}\intop_{a\in A_{+}^{\kappa},l(a)\le d_{0}}q^{2\rho(a)(1-p(\pi)^{-1})}\left|\left\langle v,\pi(kak')v\right\rangle \right|^{2}dkdk'da \\
&\gg\intop_{M_{0}}\intop_{0}^{(d_{0}-1)/2\rho(\omega_{1})}\left|F(m,x_{1})\right|dx_{1}dm\\
&\gg e^{d_{0}(1-p(\pi)^{-1})}
\end{align*}
as needed in Equation~\eqref{eq:Simplified lower bound}.
\end{proof}

\section*{Index}

The following notations appear throughout the paper.
\begin{itemize}
\item $k$ -- a local field.
\item $q$ -- if $k$ is Archimedean $q=e$. Otherwise $q$ is the size
of the quotient field of $k$,
\item $G$ -- the $k$-rational points of a semisimple algebraic group
over $k$.
\item $\Gamma$ -- a lattice in $G$. If there is a sequence $(\Gamma_N) $
of lattices then $\Gamma_N$ is a finite index subgroup of $\Gamma_1$,
with $[\Gamma_{1}:\Gamma_N]\to\infty$.
\item $K$ -- a good maximal compact subgroup of $G$.
\item $X$ -- the locally symmetric space $\Gamma\backslash G/K$.
\item $\Pi(G)$ -- the set of equivalence classes of irreducible
unitary representations of $G$. A representation is usually denoted
by $(\pi,V)$.
\item $p(\pi)$ -- the minimal $p$ such that all $K$-finite
matrix coefficients of $(\pi,V)$ are in $L^{p+\epsilon}(G)$.
\item $\lambda$ -- an eigenvalue of the Casimir operator. Appears only
in the Archimedean case.
\item $l\colon G\to\R_{\ge0}$ -- a length function on $G$. A length is usually denoted $d_{0}$.
\item $\chi_{d_0}$ -- a smooth approximation for the characteristic
set $\left\{ g\in G:l(g)\le d_{0}\right\} $.
\item $\psi_{d_0}$ -- a smooth approximation for $q^{(d_{0}-l(g))/2}\chi_{d_0}$.
\item $b_{x_0,\delta}$ -- for $x\in\Gamma\backslash G/K$, $\delta\in\R_{>0}$.
In the non-Archimedean case it is the characteristic function of $\{x\}$.
In the Archimedean case it is a smooth approximation for the characteristic
function of a ball of radius $\delta$ around $x$.
\end{itemize}
\bibliographystyle{acm}
\bibliography{./database}

\end{document}